\documentclass[11pt,reqno]{amsart}

\usepackage{hyperref,amssymb,mathrsfs,graphicx,subfigure}
\usepackage{amsfonts,amssymb,amsmath,enumerate}
\usepackage{graphicx,colortbl,xcolor}
\usepackage{float}
\usepackage{soul}
\usepackage{comment}

\newtheorem{theorem}{Theorem}[section]
\newtheorem{lemma}{Lemma}[section]

\newtheorem{proposition}{Proposition}[section]
\newtheorem{remark}{Remark}[section]
\newtheorem{definition}{Definition}[section]

\DeclareMathOperator*{\essinf}{ess\,inf}
\DeclareMathOperator*{\esssup}{ess\,sup}
\newcommand{\Section}[1]{\section{#1}\setcounter{equation}{0}}

\newcommand{\DD}{\mathcal{D}}

\def\II{\mathfrak{I}}

\def\RR{\mathsf{R}}
\def\MM{\mathsf{M}}
\renewcommand{\O}{\mathcal{O}}

\newcommand{\R}{\mathbb{R}}
\newcommand{\N}{\mathbb{N}}

\newcommand{\eps}{\varepsilon}
\renewcommand{\t}{\tau}
\newcommand{\s}{\sigma}
\newcommand{\brho}{{\bar \rho}}
\newcommand{\bvv}{{\bar \vv}}
\newcommand{\mm}{\mathtt m} 
\newcommand{\vv}{\mathtt v}

\newcommand{\tv}{\mathrm{TV}\,}

\newcommand{\BV}{$BV$\ }

\newcommand{\DT}{{\Delta t}}
\newcommand{\osc}{{\rm osc}\,}

\newcommand\ds{\displaystyle}

\begin{document}

\markboth{
Amadori and 
Christoforou}{$BV$ solutions for a hydrodynamic model of flocking--type with all-to-all interaction kernel}

\title[]{$BV$ solutions for a hydrodynamic model of flocking--type\\ with all-to-all interaction kernel}

\author[]
{Debora Amadori}
\address{\newline 
Dipartimento di Ingegneria e Scienze dell'Informazione e Matematica (DISIM), University of L'Aquila -- L'Aquila, Italy}
\email{debora.amadori@univaq.it}

\author[]
{Cleopatra Christoforou}
\address{\newline Department of Mathematics and Statistics, University of Cyprus -- Nicosia, Cyprus}
\email{christoforou.cleopatra@ucy.ac.cy}

\date{}


\begin{abstract}
We consider a hydrodynamic model of flocking-type with all-to-all interaction kernel in one-space dimension and establish global existence 
of entropy weak solutions \textit{with concentration} to the Cauchy problem for any $BV$ initial data that has finite total mass confined in a bounded interval 
and initial density uniformly positive therein. In addition, under a suitable condition on the initial data, we show that entropy weak solutions \textrm{with concentration} 
admit time-asymptotic flocking.
\end{abstract}



\maketitle

\Section{Introduction}	

The mathematical modeling of self-organized systems such as flock of birds, a swarm of bacteria or a school of fish, has brought up new
mathematical challenges and this area of study is currently receiving widespread attention in the mathematical community. 
One of the major questions addressed concerns the emergent behavior in these systems and in particular, the emergence of flocking behavior. 
Many models have been introduced to appraise the emergent behavior of self-organized systems. 
Actually, many of these models have arised from the inspiring work of Cucker and Smale\cite{CuS1}, while led to many subsequent studies;
cf. Refs.~\cite{CFTV-2010,Shv2021} and the references therein. 
Most studies on flocking models deal with the behavior of particle model or the corresponding kinetic equation and its 
hydrodynamic formulation (cf. Refs.~\cite{HT08,HL09,MT11,KMT13,HaHuangWang2014,HaKangKwon2014,HaKangKwon2015,KMT15}). 
So far this topic for self-organized dynamics has been investigated mainly in the context of solutions with no discontinuities. 

In this article, we are interested in {\em weak solutions} to the hydrodynamic model of flocking-type in one-space dimension 
that takes the form
\begin{equation}\label{eq:system_Eulerian}
\begin{cases}
\partial_t\rho +  \partial_x (\rho \vv)  = 0, &\\
\partial_t (\rho \vv)  +  \partial_x \left(\rho \vv^2 + p(\rho) \right) =\int_\R \rho(x,t)\rho(x',t)\left(\vv(x',t)  - \vv(x,t) \right)\,dx'  &
\end{cases} 
\end{equation}
where $(x,t)\in \R\times [0,+\infty)$. Here $\rho\ge 0$ stands for the mass variable, $\vv$ the velocity and $p$ the pressure, 
while we denote by $\mm=\rho\vv$ the momentum.

In the work of Karper, Mellet and Trivisa\cite{KMT15}, they investigate the hydrodynamic limit ($\epsilon\to0$) of the 
kinetic Cucker-Smale flocking model on $(x,t,\omega)\in (0,T)\times\R^d\times\R^d$
\begin{equation}\label{S1eq:hydro}
f_t^\epsilon+\omega\cdot\nabla_x f^\epsilon+\text{div}_{\omega}(f^\epsilon L[f^\epsilon])=\frac1\epsilon\Delta_\omega 
f^\epsilon+\frac1\epsilon \text{div}_\omega(f^\epsilon(\omega-\vv^\epsilon))
\end{equation}
with $f^\epsilon\dot=f^\epsilon(x,t,\omega)$ being the scalar density of individuals and $L$ the alignment operator 
that is the usual Cucker-Smale operator
$$ 
L[f] (x,t,\omega)~\dot =~\int_{\R^d}\int_{\R^d} K(x,y) f(y,w)(w-\omega)dw\,dy
$$
where $K$ is a smooth symmetric kernel and $\epsilon>0$ a small positive parameter.

On the right hand side of \eqref{S1eq:hydro}, the first term is due to the presence of a stochastic forcing at the particle level, 
see Ref.~\cite{BCC-2011}, while the last term in~\eqref{S1eq:hydro} is the strong local alignment interaction with $\vv^\epsilon$ 
to be the average local velocity. 
This alignment term was derived in Ref.~\cite{KMT13} from the Motsch-Tadmor\cite{MT11} alignment operator (MT) as a singular limit. 
The MT operator was introduced in an effort to improve the standard Cucker-Smale model at small scales. 

In Ref.~\cite{KMT15}, they study the singular limit corresponding to strong noise and strong local alignment, i.e. $\epsilon\to0$.  
For this limit to hold, formally it must be
\begin{equation*}
    \Delta_\omega f^\epsilon+ \text{div}_\omega(f^\epsilon(\omega-\vv^\epsilon))
\to 0\qquad {\rm as }~\epsilon\to 0\,.
\end{equation*}
Then, the limit of $f^\epsilon$ will have the form
\begin{equation}\label{eq:conv-to-Maxwellian}
    f^\epsilon \to f(x,t,\omega)=\rho(x,t) e^{-\frac{|\omega-\vv(x,t)|^2}{2}}\,,
\end{equation}
while the macroscopic variables $\rho,\,\vv$, which are the $\epsilon\to0$ limits of
$$
\rho^\epsilon\dot=\int f^\epsilon d\omega,\qquad \rho^\epsilon \vv^\epsilon\dot=\int f^\epsilon \omega d\omega\;,
$$
satisfy the Euler-type flocking system~\eqref{eq:system_Eulerian} with pressure
\begin{equation}\label{gamma=1}
p(\rho)= \alpha^2\rho\,, \qquad \alpha>0\,.
\end{equation}
In Ref.~\cite{KMT15} the authors prove this limit rigorously, showing the convergence of weak solutions to the kinetic
equation~\eqref{S1eq:hydro} to strong (suitably smooth) solutions of the Euler system.

We remark that, in the literature, hydrodynamic models for flocking are often described by a pressureless Euler system. 
This results from a microscopic description of the particles motion, that does not contain a stochastic forcing. 
As a consequence, the kinetic equation does not contain the diffusion term, and the momentum equation can be closed by 
the mono-kinetic ansatz 
$$f(x,t,\omega)=\rho(x,t) \delta_{\omega - \vv(x,t)}\,,$$
differently from \eqref{eq:conv-to-Maxwellian}. The system with pressure received less attention than the pressureless one, 
in particular from the point of view of weak solutions. A result on smooth, space-periodic solutions 
to this model with pressure is established in Ref.~\cite{Choi2019}.

Concerning more recent development of the Euler-alignment system, see the detailed exposition of the current state of the theory 
in Ref.~\cite{Shv2021}, from \S 6, and references therein. In particular, a central topic of interest is the occurrence of 
flocking in presence of \emph{topological communication} interaction, that is, of communication that depends on a weighted distance 
based on the local mass, in contrast with the metric distance between individuals. This approach raises several difficulties; 
for instance it is lost, in general, the symmetry property of the interaction, that instead holds for the metric distance. 
The analysis of this type of problem was initiated in Ref.~\cite{MT11}; we refer to Refs.~\cite{ST2020,LRS22} for recent developments on the pressureless models with topological communication, 
where existence and global regularity results of multidimensional periodic solutions are proved.
Another interesting topic that has been studied for the Euler alignment system is related to critical thresholds
on the initial configurations that either lead to global regular solutions which must satisfy time-asymptotic flocking or provoke finite time blow up of solutions\cite{TT2014,CCTT-2016}. In particular we refer the reader to Ref.~\cite{CCTT-2016} for the study of sharp critical thresholds on the pressureless system and an investigation on the isothermal one.
We refer  the reader to Ref.~\cite{ABFHKPPS2019} for a general survey of mathematical models for collective dynamics, 
more specifically to \S 5 for an overview of kinetic and macroscopic equations arising in the context of species behavior.

In this article the hydrodynamic limit equations of Karper, Mellet and Trivisa\cite{KMT15} is the central theme of this paper 
for the {\em all-to-all} interaction kernel, that is, the situation of $K$ being independent on the positions $x$, $y$. 
We will assume that $K(x,y)\equiv 1$. In particular we are interested in the existence of weak solutions $(\rho,\mm)$
    \begin{equation*}
        \rho=\rho(x,t),\qquad \mm(x,t) = \rho(x,t) \vv(x,t)\,,
    \end{equation*}
to the Cauchy problem that consists of system~\eqref{eq:system_Eulerian} 
with initial data for the density $\rho$, and momentum $\mm=\rho\vv$:
\begin{equation}\label{eq:init-data}
(\rho,\mm)(x,0)=\left(\rho_0(x), \mm_0(x)\right)\,\qquad x\in\mathbb{R}\;.
\end{equation}

More precisely, our aim is twofold for the Cauchy problem~\eqref{eq:system_Eulerian},\,\eqref{eq:init-data} with~\eqref{gamma=1}.
First, we are interested in the global existence of solutions with initial data of bounded variation ($BV$), and second, 
in the description of flocking-type behavior.

Having in mind that the hydrodynamic limit of large particle self-organized systems yields the macroscopic density and momentum $(\rho,\mm)$,
$\mm=\rho \vv$, satisfying the equations of conservation of mass and momentum ~\eqref{eq:system_Eulerian}, cf. Ref.~\cite{KMT15}, 
we impose conditions on the initial data~\eqref{eq:init-data} that are suitable in this framework and illustrate the emergent behavior 
of self-organized systems. More precisely, we assume that the initial total mass $\rho$ is confined in a bounded interval, 
and is uniformly positive in there i.e.  there exist $a_0 < b_0$ such that,  for $I_0 ~\dot =~[a_0, b_0]$
\begin{equation}\label{hyp-init_data}
\begin{cases}
\essinf_{I_0} \rho_0 >0\,,&\\ 
\rho_0(x) = 0 &  \forall\,x\not \in I_0 \,. 
\end{cases}
\end{equation}
Regarding $\mm_0$, it is meaningful to define it only in the region where the density is positive; for simplicity, 
we choose to extend it to 0 outside that region. Therefore, in the following, we assume $\mm_0\in BV(\R)$ and that
\begin{equation}\label{hyp-init_data-v}
\mm_0(x) = 0\qquad  \forall\,x\not \in I_0\,.
\end{equation}
Moreover, we use $\vv_0=\mm_0/\rho_0$ only in $I_0$.

To introduce the notion of \emph{entropy weak solutions with concentration}, we need the functions 
$$
\eta(\rho,\mm)\,,\quad q(\rho,\mm)\,,
$$ 
defined on $(0,+\infty)\times \R$, in terms of $\rho> 0$ and $\mm$, that constitute a pair of entropy-entropy flux functions 
for the system \eqref{eq:system_Eulerian} if: 
they are differentiable on $(0,+\infty)\times \R$, $\eta$ is convex and the following relations hold,
\begin{equation*}
    \left(-\vv^2 + p'\right)\eta_{\mm} = q_\rho\,,\qquad 
    \eta_\rho + 2\vv \eta_\mm = q_{\mm}\,.
\end{equation*}

The definition of entropy weak solutions now follows.

\begin{definition}\label{entropy-sol} Given the initial data $(\rho_0,\mm_0)\in BV(\R)$, together with 
\eqref{hyp-init_data} and \eqref{hyp-init_data-v}, let $(\rho,\mm): [0,+\infty)\times \R \to \R^2$ 
be a function with the following properties:
\begin{itemize}
    \item the map $t\mapsto (\rho,\mm)(\cdot,t) \in L^1_{loc} \cap~ BV$  is continuous in $L^1_{loc}$;
    \smallskip
 \item 
 $\displaystyle\lim_{t\to 0+}(\rho, \mm)(\cdot,t)=\left(\rho_0, \mm_0\right)$ in $L^1_{loc}$;
    
     \smallskip
     
    \item there exist two locally Lipschitz curves $t\mapsto a(t)$, $b(t)$, $t\in[0,+\infty)$ and a value $\rho_{inf}>0$ such that the following holds:
\begin{itemize}
\item [(i)]
$
a(0)=a_0\,,\quad b(0)=b_0\,;\qquad a(t)<b(t)\qquad  \mbox{ for all }t>0\,;
$

\smallskip
\item[(ii)] having set
\begin{equation*}
    I(t) = [a(t),b(t)]\,,\quad t\ge 0\,,
\end{equation*}
the following holds:
\begin{equation}\label{solution-structure}
\begin{cases}
\essinf_{I(t)} \rho(\cdot,t) \ge \rho_{inf}>0\,,&\\[1mm] 
(\rho,\mm)(x,t) = 0 &  \forall\,x\not \in  I(t) \,.
\end{cases}
\end{equation}
\end{itemize}
\end{itemize}
Then $(\rho, \mm)$ is an entropy weak solution \textit{with concentration} along $a(t)$ and $b(t)$ of the problem~\eqref{eq:system_Eulerian},
\eqref{eq:init-data} with \eqref{gamma=1}, if 
\begin{itemize}
    \item[(a)] the following integral identities hold true for all test functions $\phi\in C^\infty_0(\R\times (0,\infty))$:
\begin{align}
&\iint_{\Omega} \left\{\rho\phi_t + \mm\phi_x \right\}\; dxdt=0\;,\label{S1:rho-eq-phi2}\\
&\iint_{\Omega}\left\{ \mm\phi_t
    + \left[ \frac{(\mm)^2}{\rho}+ p(\rho)  \right]\phi_x \right\}dx dt \nonumber\\
    &-\iint_{\Omega}\left[\mm
   \left( \int_\R   \rho(x',t)d x' \right)
    -\rho 
    \left(\int_\R 
     \mm(x',t)d x'+P_b(t)-P_a(t)\right)\right] \phi \,dx dt \nonumber\\
     &   - \int_0^\infty   \left[p(\rho(b(t)-,t)) \phi(b(t),t) -p(\rho(a(t)+,t)) \phi(a(t),t) \right] \,dt=0\;,\label{S1:m-eq-phi2}
\end{align}
with
    \begin{equation*}
\Omega= \{ (x,t);\ t> 0\,,\  x\in (a(t),b(t))\}\subset \R\times(0,+\infty)
\end{equation*} and
\begin{equation}\label{def:Pnu-intro}
\begin{cases}
\ds P_b(t) :=  \int_0^t e^{-M(t-s)}p(\rho(b(s)-,s))\, ds\,,  & \\[2mm]
\ds P_a(t) :=  \int_0^t e^{-M(t-s)} p(\rho(a(s)+,s))\, ds\,;
    \end{cases}
\end{equation}
     \smallskip
    \item[(b)] for every convex entropy $\eta$ for the system \eqref{eq:system_Eulerian}, with corresponding entropy flux $q$, the following inequality
 \begin{equation}\label{entropy-cond_rho-m}
    \partial_t \eta(\rho,\mm) + \partial_x q(\rho,\mm)\le \eta_{\mm} \int_\R \left(\rho(\cdot,t)\mm(x',t)  - \rho(x',t)\mm(\cdot,t) \right)\,dx'
\end{equation}

holds in the sense of distributions on the open set $\Omega$.
\end{itemize} 
\end{definition}
Throughout this article, we are interested in an \emph{entropy weak solution} in the sense of Definition~\ref{entropy-sol}
that satisfies the additional ad-hoc boundary condition:
\begin{equation}\label{ad-hoc vacuum}
\text{\emph{The vacuum region is connected with the non-vacuum one by a shock discontinuity.}} 
\end{equation}
This condition is imposed in order to capture the flocking behavior in which a sharp front with finite speed is expected to arise. 
We exclude the case of a rarefaction because a rarefaction connecting a vacuum region with a non-vacuum would necessarily 
have infinite speed and this is not suitable for our model. Consequently, it is shown that the entropy weak solution with concentration
constructed in this article consists of two singularities connecting the vacuum region with the non-vacuum one that emanate from $a_0$ 
and $b_0$, respectively. These would play the role of the curves called $a(t)$ and $b(t)$ in the Definition~\ref{entropy-sol}. 
For this reason, we define \emph{the total momentum} $\widehat\mm$ to be the distribution
\begin{equation}\label{def:m-hat}
    \widehat \mm(\cdot,t) := \mm (\cdot,t) + \delta_{b(t)} P_b(t) - \delta_{a(t)} P_a(t)\,,\quad t>0\,.
\end{equation}
where $\delta_a$ denotes the Dirac delta function at $a$. This new singularity of the total momentum $\widehat m$ along the free boundaries $a(t)$ and $b(t)$ is known as delta shock and references can be found in 
Ref.~\cite{Dafermosbook}, Chapter 9.
We further use the following standard notation $<\cdot,\cdot>$:
$$
<\widehat\mm(\cdot,t),\phi(\cdot,t)>:=\int_{I(t)} \mm (x,t)\phi(x,t)dx+ P_b(t)\phi(b(t),t) -  P_a(t)\phi(a(t),t),\quad t>0 
$$
as the value of the functional $\widehat \mm$ over $C_0^\infty$, for all test functions $\phi\in C_0^\infty(\R\times\R_+)$. Note here that $\mm=0$ for $x\notin I(t)$.

Now we state our first main result, on the global existence in time of entropy weak solutions with concentration to
\eqref{eq:system_Eulerian} with bounded support.

\begin{theorem}\label{Th-1} Assume that the initial data $(\rho_0,\mm_0)\in BV(\R)$ and satisfy \eqref{hyp-init_data},
\eqref{hyp-init_data-v} with pressure~\eqref{gamma=1}.
Then the Cauchy problem \eqref{eq:system_Eulerian}, \eqref{eq:init-data} admits an entropy weak solution with concentration $(\rho,\mm)$
in the sense of Definition~\ref{entropy-sol}. Moreover one has that
\begin{equation}\label{cons-of-mass}
\int_\R \rho(x,t)\,dx = \int_{I(t)} \rho(x,t)\,dx = \int_\R \rho_0(x)\,dx\,,\qquad\forall\, t\ge 0\,; 
\end{equation}
\begin{equation}\label{cons-of-momentum}
<  \widehat\mm(\cdot,t) ,\phi_1>= \int_{I(t)} \mm(x,t) \,dx + P_b(t) - P_a(t)
= \int_\R \mm_0(x) \,dx\,,\quad\forall\, t\ge 0  \,,
\end{equation}
for any test function $\phi_1 = \phi_1(x)$ that is equal to $1$ on $I(t)$.
\end{theorem}

As a consequence of Theorem~\ref{Th-1}, there is conservation of mass and momentum by the entropy weak solution constructed in time, 
with values:
\begin{align} \label{eq:def-M} 
 &
 M\dot =\int_\R \rho_0(x)\,dx\ >0\,, \\   \label{eq:def-M1}
 &
 M_1\dot =\int_\R \mm_0(x) \,dx \,.
 \end{align}
Also, the variable $\vv$ is used in the support $I(t)$ for all $t>0$ satisfying $\vv=\frac{\mm}{\rho}$, where it is well defined, 
while only the variables $\rho$ and $\mm$ are used in the complement of $I(t)$. By setting the average velocity to be
\begin{equation}\label{def:vbar}
\bar \vv ~\dot =~ {M_1}/{M} 
\end{equation}
and by means of \eqref{cons-of-mass}--\eqref{cons-of-momentum}, the integral term on the right hand side of \eqref{eq:system_Eulerian} can be rewritten as
\begin{align*}
\rho(x,t) \left\{ <  \widehat\mm(\cdot,t) ,\phi_1> -  \vv(x,t)  \int_\R \rho(x',t)\,dx'\right\}
&= \rho(x,t) \left(M_1 - \vv(x,t) M \right)\\
&= M \rho(x,t) \left(\bar \vv - \vv(x,t)\right)
\end{align*}
and system \eqref{eq:system_Eulerian} rewrites as
\begin{equation}\label{eq:system_Eulerian_M-M1}
\begin{cases}
\partial_t\rho +  \partial_x (\rho \vv)  = 0, &\\
\partial_t (\rho \vv)  +  \partial_x \left(\rho \vv^2 + p(\rho) \right) =  
- M \rho \left(\vv -\bar \vv \right) \,.
&
\end{cases} 
\end{equation}
In other words, solutions to system \eqref{eq:system_Eulerian} that conserve mass and momentum, i.e. obey \eqref{cons-of-mass} and
\eqref{cons-of-momentum}, then they also satisfy the system of balance laws~\eqref{eq:system_Eulerian_M-M1} and vice versa. 
For this reason, our plan of action to prove Theorem~\ref{Th-1} is to establish an entropy weak solution with concentration
to~\eqref{eq:system_Eulerian_M-M1} for the initial data assumed in the theorem. 

System~\eqref{eq:system_Eulerian_M-M1} belongs to the class of system of balance laws and an exposition of the current state of the theory
can be found in the book Ref.~\cite{Dafermosbook}. The Cauchy problem for strictly hyperbolic systems of balance has been studied initially 
by Dafermos-Hsiao\cite{DafermosHsiao} in the framework of entropy weak solutions and the condition of \emph{strongly dissipative} 
source terms has been introduced in Ref.~\cite{DafermosHsiao} (see also Ref.~\cite{AmadoriGuerranote}) and even relaxed to \emph{weakly
dissipative} sources in Ref.~\cite{D3}. Under these conditions, there is no arbitrary amplification in total variation due to the presence 
of the sources and uniform estimates in time allow to construct entropy weak solutions for all times. However, the source term
in~\eqref{eq:system_Eulerian_M-M1} does not fulfill these dissipativeness conditions and furthermore, our Cauchy problem is not 
strictly hyperbolic because of the vacuum present in $\R\setminus[a_0,b_0]$. Hence, the existing results cannot be applied to our situation.
Actually, the notion of the physical vacuum boundary for system~\eqref{eq:system_Eulerian_M-M1} was introduced in Refs.~\cite{Liu1996,LiuYang2000} and many results are available in the literature related to the blow up of regular solutions due to the singular behavior at the vacuum. Another interesting direction addressed for system~\eqref{eq:system_Eulerian_M-M1} is the asymptotic behavior of the solutions that are conjectured to obey the porous media equation. An important progress on this conjecture has been made, see for instance Refs.~\cite{HuangPan2006,HuangMarcati2005,HuangPanWang2011}. It should be clarified that the long time behavior to the porous media equation for the case of the pressure~\eqref{gamma=1} is considered only  in Ref.~\cite{HuangPan2006}, but for initial density $\rho_0$ that tends to a positive value as $x\to\pm\infty$. Thus, the result in Ref.~\cite{HuangPan2006} does not include the initial data~\eqref{hyp-init_data} in our problem. Also, the analysis in Ref.~\cite{HuangPan2006} employs the compensated compactness method that provides global solutions in $L^\infty$. However, in Refs.~\cite{HuangMarcati2005,HuangPanWang2011} the authors assume that the pressure is of the form $p(\rho)=\rho^\gamma$ with $\gamma>1$ and therefore, their analysis cannot be applied again here. 

An interesting feature of our model is that the Riemann solution around vacuum that is admissible in the sense of ``flocking behavior" consists a shock in contrast to the case of gas dynamics in which a rarefaction wave arises connecting the vacuum with the non-vacuum region. This feature can be evinced by the following example. Consider the initial data
\begin{equation}\label{eq:elem-init-data}
    \left(\rho_0(x),\mm_0(x)\right) = \begin{cases}
    \left(\bar\rho, \bar\rho\,\bar \vv \right) &\qquad  x\in[a_0,b_0]\\[2mm]
    \left(0,0\right) &\qquad  x\not \in[a_0,b_0]
     \end{cases}
\end{equation}
with $\bar\rho>0$ and $\bar \vv$ constant values. For this special case, the entropy weak solution with concentration
to~\eqref{eq:system_Eulerian}, \eqref{eq:init-data}, as constructed in Theorem~\ref{Th-1}, is the stationary solution translated 
with speed $\bar \vv$, i.e. 
\begin{equation}\label{eq:elem-solutions}
    \left(\rho(x,t),\mm(x,t)\right) = \begin{cases}
    \left(\bar\rho, \bar\rho\,\bar \vv \right) &\qquad  x\in I(t)\\[2mm]
    \left(0,0\right) &\qquad  x\not \in I(t)
     \end{cases}
\end{equation}
%
for all $t>0$, with $I(t)=[a_0+\bar \vv t,b_0+\bar \vv t]$, together with equal concentration of the momentum along the two discontinuities
with weight $P_a(t)=P_b(t)=p(\bar\rho) M^{-1} \left(1-e^{-Mt}\right)$ and thus, it corresponds to the notion of solution given in
Definition~\ref{entropy-sol}. This is consistent with the interpretation of the flocking model; 
in particular, no rarefactions appear between the region with positive density and the vacuum regions. 
A detailed discussion can be found in Subsection~\ref{S2new}.

Since the Cauchy problem~\eqref{eq:system_Eulerian_M-M1},~\eqref{eq:init-data} is not strictly hyperbolic and the techniques 
to construct a convergent approximate sequence to systems of conservation laws require strict hyperbolicity, we overcome this obstacle 
by transforming our problem into Lagrangian coordinates in the spirit of Wagner\cite{W87}. 
By recasting system~\eqref{eq:system_Eulerian_M-M1} from Eulerian $(\rho(x,t),\vv(x,t))$ into the Lagrangian variables $(u(y,t), v(y,t))$, we obtain the equations
\begin{equation}\label{S1eq:system_Lagrangian}
\begin{cases}
\partial_\t u -  \partial_y  v  = 0, &\\
\partial_\t v  +  \partial_y (\alpha^2/u) =   - M (v-\bar \vv) & 
\\ 
\end{cases} 
\end{equation}
with the domain 
$\{(y,t);\ t\ge 0\,,\ y\in (0,M)     
\}$\,. 

In Ref.~\cite{W87}, the equivalence of the Eulerian and Lagrangian equations for weak solutions is established even in the presence of vacuum 
under a condition on the total mass. Such condition requires infinite total mass and therefore the equivalence of weak solutions between the Eulerian 
and Lagrangian solutions does not apply in our framework since we require finite total mass.
Actually, because of the finite mass condition on the initial data~\eqref{hyp-init_data}, the problem in Lagrangian coordinates is not Cauchy any more. 
In short, although we proceed to study the problem in Lagrangian variables, we need to prove the equivalence between the systems in the sense of weak solutions.

Now, existence of weak solutions to the homogeneous system ($M=0$) corresponding to~\eqref{S1eq:system_Lagrangian} was obtained by Nishida\cite{Nishida68}
using the random choice method for initial data of large \BV and both the Cauchy and the boundary problems are discussed. Global existence of entropy weak solutions to systems with source of the form~\eqref{S1eq:system_Lagrangian} has been established in Refs.
~\cite{Dafermos_frictional,LuoNatYan,AmadoriGuerra01}, using either the random choice or the front tracking methods,
but for the Cauchy problem. 
Therefore, their results are not applicable in our case, since we deal with a boundary value problem in Lagrangian variables. On the other hand, in Frid\cite{Frid96} 
the author considers certain initial-boundary value problems on a bounded domain for systems arising in isentropic gas dynamics and elasticity theory, which are different from the one that is treated here.

We remark that, if $p(\rho)$ were chosen to be a generic smooth and nonlinear function with $p'>0$, it would have been natural to assume small $BV$ initial data, 
as in Ref.~\cite{Dafermos_frictional}. However, the choice of $p(\rho)$ in \eqref{gamma=1} allows us to deal with  any initial data with possibly large (but finite) total variation,
as long as the structure described in~\eqref{hyp-init_data}--\eqref{hyp-init_data-v} is respected.

Our strategy to attack this problem is to study the initial-boundary value problem~\eqref{S1eq:system_Lagrangian} in Lagrangian variables for 
{\em non-reflecting} boundary conditions at $y=0,M$ and we construct approximate solutions using the front tracking algorithm; 
cf. Bressan\cite{Bressan_Book} and Holden--Risebro\cite{HoldRisebro2015}. The boundary conditions are expressed by the fact that, 
when a wave-front reaches the boundary, there is no resulting emitted wave. As it is shown in Section~\ref{Sect:3}, 
this is the natural counterpart to the behavior of the free boundaries that, in Eulerian variables, delimit the non-vacuum region.

Then uniform estimates on the total variation in space and time allow us to pass to a convergent subsequence and recover an entropy weak solution to the problem in Lagrangian variables that conserves mass and momentum in time. We work, at the level of the approximate solutions, to show the equivalence between Eulerian and Lagrangian variables in the spirit of Ref.~\cite{W87} and within the domain $I(t)=[a(t),b(t)]$ where no vacuum is present. In this way, we construct the approximate solutions to~\eqref{eq:system_Eulerian_M-M1} that inherits the convergence property from the change of coordinates. Also, this coordinate transformation together with the approximation scheme allow us to pass from the non-reflecting boundary conditions at $y=0,M$ to the free boundaries $a(t)<b(t)$ in the Eulerian coordinates. 

An advantage of studying this problem in Lagrangian variables is that the analysis on the bounded domain $[0,M]$ 
(in Lagrangian) provides useful information to the community interested in the numerical analysis for this topic 
on self-organized systems. 

Next, we address the issue of the long-time behavior of the entropy weak solution with concentration to \eqref{eq:system_Eulerian},  \eqref{eq:init-data} with \eqref{gamma=1}, 
whose structure is established in Theorem~\ref{Th-1}. This is an important issue in the study of self-organized systems. The terminology ``flocking" corresponds to the phenomenon in which self-organized individuals using only limited environmental information and simple rules get organized into an ordered motion.
In the spirit of Refs.~\cite{HT08,HL09}, we provide the following definition of flocking behavior. 

\begin{definition}\label{def-flocking} We say that the solution $(\rho,\mm)(x,t)$ to system~\eqref{eq:system_Eulerian}, 
together with the initial data $(\rho_0,\mm_0)$ as in Theorem~\ref{Th-1}, admits time-asymptotic flocking if the following conditions hold true:
\begin{enumerate}
\item the support $I(t)$ of the solution remains bounded for all times, i.e.
\begin{eqnarray}\label{eq:bounded-diameter}
&&\sup_{0\le t <\infty} \{b(t) - a(t)\} <\infty\;.
\end{eqnarray}
\item  the velocity satisfies 
\begin{eqnarray}
&&\lim_{t\to\infty} \esssup_{x_1,x_2\in I(t)} |\vv(x_1,t) - \vv(x_2,t)|   
=0. \label{eq:v-shrinks-to-0}
\end{eqnarray}
\end{enumerate}
\end{definition}
Loosely speaking, condition~\eqref{eq:bounded-diameter} assures that the support of the solution is uniformly bounded, thus defining the
``flock", while condition~\eqref{eq:v-shrinks-to-0} yields that alignment occurs, i.e. the diameter of the set of velocity states 
within the support $I(t)$ goes to zero time-asymptotically.

Observe that condition \eqref{eq:v-shrinks-to-0} is equivalent to
\begin{equation*}
    \esssup_{x\in I(t)} |\vv(x,t) - \bar \vv| \to 0\qquad t\to\infty\,,
\end{equation*}
and hence, the dynamics of this flock will approach the same velocity $\bar \vv$. Actually, under the assumptions of Theorem~\ref{Th-1},
condition~\eqref{eq:bounded-diameter} is satisfied by the entropy weak solution with concentration $(\rho,\mm)$ constructed. Indeed, by
combining \eqref{solution-structure} and \eqref{cons-of-mass}, it follows that
\begin{equation*}
    b(t) - a(t)\le \frac{M}{\rho_{inf}}\qquad \forall\, t>0
\end{equation*}
and this yields \eqref{eq:bounded-diameter} under the assumptions of Theorem~\ref{Th-1}. The goal is to show that the \emph{time-asymptotic flocking} property,
hence~\eqref{eq:v-shrinks-to-0}, holds as well. To achieve this, we need to impose a special condition on the data that relies on the initial bulk $q$: 
\begin{equation}\label{eq:def-of-m}
q: = \frac 12 \tv \{\ln(\rho_0)\} +  \frac 1{2\alpha} \tv \{\vv_0\}    
\end{equation}
where {\rm TV} stands for the total variation in the support $I_0$. The condition is stated in the next theorem and indicates that flocking occurs when the initial bulk $q$ 
is controlled by the the initial density at the endpoints $a_0$ and $b_0$. In fact, we show that the velocity decays to $\bar \vv$ at an exponential in time rate.

Here, it is our second main result: 

\begin{theorem}\label{Th-2} 
Let  $(\rho,\mm)$ be the entropy weak solution with concentration to~\eqref{eq:system_Eulerian}, \eqref{gamma=1}, \eqref{eq:init-data} 
with the initial data $(\rho_0,\mm_0)\in BV(\R)$ satisfying \eqref{hyp-init_data},~\eqref{hyp-init_data-v}
with $q>0$ as obtained in Theorem~\ref{Th-1}. Suppose that
\begin{align}\label{Th-2assumption}
{e^{2q} M^2} <\alpha\max\left\{\rho_0(a_0+),\rho_0(b_0-)\right\},
\end{align}
holds true, then the solution $(\rho,\mm)$ admits \emph{time-asymptotic flocking}. 
More precisely, the oscillation of the velocity decays exponentially fast, i.e. there exists $t_0>0$ such that 
\begin{align}\label{Th-2exp}
\esssup_{x_1,x_2\in I(t)} |\vv(x_1,t) - \vv(x_2,t)|
\le C_2'e^{-C_1't},\,\qquad \forall\, t\ge t_0
\end{align}
for some positive constants $C_1',\,C_2'$.
\end{theorem}

\begin{remark} Let us note that for the trivial case $q=0$, the initial data reduce to the one in \eqref{eq:elem-init-data}, 
for some constant values $\bar\rho>0$ and $\bar \vv$. The corresponding solution, given in \eqref{eq:elem-solutions}, 
admits automatically \emph{time-asymptotic flocking} for every $\bar\rho>0$, $\bar \vv\in \R$. 


On the other hand, condition~\eqref{Th-2assumption} reduces here to $M^2 < \alpha \bar \rho$, with $M=\bar \rho(b_0-a_0)$ 
and $\alpha$ given in \eqref{gamma=1}. Therefore it imposes a restriction on the parameters $\alpha$, $\bar \rho$ and $b_0-a_0$. 
In this sense, condition~\eqref{Th-2assumption} is a sufficient condition for time-asymptotic flocking and we expect it may be relaxed. 
\end{remark}

The analysis of the long time behavior of the approximate solutions, makes use of the results in Refs.~\cite{AmadoriGuerra01,AC_SIMA_2008,ABCD_JEE_2015} 
about the {\em wave strength dissipation} for the system of isothermal flow, that is, the homogeneous version of \eqref{eq:system_Eulerian} 
with \eqref{gamma=1}. These properties are used in Subsection~\ref{subsect:Lxi}, to provide uniform bounds on the vertical traces of approximate
solutions, and in Subsections~\ref{S4.1}, \ref{S4.2} for the time-asymptotic analysis. In the present paper we identify a new functional,
see~\eqref{Vweightedgen}, that allows us to detect a {\em geometric decay property} in terms of the number of wave reflections for the 
homogeneous system mentioned above; we remark that no smallness of the total variation of the initial data is required at this level. 
This decay property allows us to control, under assumption~\eqref{Th-2assumption}, the possible increase of the functional which is due 
to the reflections produced by the damping term, leading to the time-exponential decay established in \eqref{Th-2exp}. 

An interesting research direction is to investigate the validity of Theorems~\ref{Th-1} and \ref{Th-2} 
in presence of a more general kernel $K(x,x')$ in the alignment term,
$$\int_\R K(x,x')\rho(x,t)\rho(x',t)\left(\vv(x',t)  - \vv(x,t) \right)\,dx'\,.$$
In particular, it would be interesting to consider
\begin{equation*}
    K(x,x')= \frac 1{ \left(1+ (x-x')^2\right)^{\frac\beta 2} }\,,\qquad \beta\ge 0
\end{equation*}%
as in Refs.~\cite{CuS1,HT08}. The analysis developed in the present paper, which is for $\beta=0$, aims to be a first step in this direction.

Although the assumption of an all-to-all interaction kernel $K=1$ simplifies the system and the nonlocal term 
turns into a local term, this special case still  possesses many obstacles in the analysis point of view, as already mentioned, 
in order to capture the existence of solutions and time-asymptotic flocking. 
In addition to that, the work on this special case indicates how the mechanism of the dissipative behavior of the solutions works 
and we expect this to be crucial in the extension of this analysis to general kernels.

The problem of uniqueness and stability of the solutions established in Theorem~\ref{Th-1} is challenging and it is beyond the scope of this paper.
However, we expect that the entropy condition inside the region of positive density in conjunction with the \emph{chosen} unique solution to the Riemann
problem that connects the vacuum region with the non-vacuum region, see condition~\eqref{ad-hoc vacuum}, are the two conditions that would guarantee the
uniqueness of the solutions established in Theorem~\ref{Th-1}. Certainly, the analysis would involve many obstacles and one should bring into play further
machinery involving appropriate stability functionals, in the spirit of Ref.~\cite{Bressan_Book}, Chapter 8. 

The structure of the paper is the following: In Section~\ref{newsection}, we first present the Riemann solution to~\eqref{eq:system_Eulerian_M-M1} around the vacuum 
that is admissible for the flocking model, motivate our notion of \emph{weak solution with concentration} and then recast the problem~\eqref{eq:system_Eulerian_M-M1}, \eqref{eq:init-data} in Lagrangian variables.
In Section~\ref{Sect:2}, we construct an approximate sequence of solutions using the front tracking algorithm, define appropriate Lyapunov functionals, 
show that the total variation in space and time of the approximate sequence remains bounded.  We conclude the section by proving the convergence to an entropy weak solution 
of the system in Lagrangian coordinates. Then, in Section~\ref{Sect:3}, we transform this analysis to the problem~\eqref{eq:system_Eulerian_M-M1} in Eulerian and then, prove Theorem~\ref{Th-1}. Last, in Section~\ref{Sect:4}, we establish decay estimates for the approximate sequence as time goes to infinity that allow us to capture the flocking behavior and establish Theorem~\ref{Th-2}. In Appendix, we include the proof of technical lemmas presented in Section~\ref{Sect:2}.

\Section{Set Up of the Problem}\label{newsection}

In this section, the aim is to construct an entropy weak solution \textrm{with concentration} to~\eqref{eq:system_Eulerian_M-M1} with the properties stated in Theorem~\ref{Th-1} using the wave front tracking algorithm. Since the solution satisfies the conservation of mass and momentum, i.e.~\eqref{cons-of-mass}--\eqref{cons-of-momentum} then the solution constructed would also be solution to system ~\eqref{eq:system_Eulerian}. Thus, from here and on, we study system~\eqref{eq:system_Eulerian_M-M1}, with \eqref{eq:def-M}, \eqref{eq:def-M1}.
However, 
before we proceed, we reduce the problem to zero average velocity. More precisely, by the definition of $\bar \vv$ at \eqref{def:vbar}, we perform the change of variables 
\begin{equation*}
x\mapsto x-\bar \vv \,t\,,\qquad \vv\mapsto \vv- \bar \vv
\end{equation*}
that allows us to reduce to the case of $M_1=0$, with $M_1$ being defined at \eqref{eq:def-M1}. Indeed, in the new variables called again $(x,t)$, the average of the momentum $\widehat\mm$ becomes zero: 
\begin{equation}\label{eq:zero-average-momentum}
\int_{I(t)} \mm(x,t) \,dx+P_b(t)-P_a(t) = 0\qquad \forall \,t\ge0
\end{equation}
and system \eqref{eq:system_Eulerian} or~\eqref{eq:system_Eulerian_M-M1} takes the form
\begin{equation}\label{eq:system_Eulerian_K=1}
\begin{cases}
\partial_t\rho +  \partial_x (\rho \vv)  = 0, &\\
\partial_t (\rho \vv)  +  \partial_x \left(\rho \vv^2 + p(\rho) \right) =  - M \rho \vv\, 
&
\end{cases} 
\end{equation}
for solutions that conserve mass and momentum. Therefore, from here and on, we can assume that $M_1=\bar \vv =0$ and consider
system~\eqref{eq:system_Eulerian_K=1}. 
Now, for system~\eqref{eq:system_Eulerian_K=1}, the integral identities of the entropy weak solution with concentration in the sense 
of Definition~\ref{entropy-sol} take the form~\eqref{S1:rho-eq-phi2} and
\begin{align}
\iint_{\Omega}&\left\{ \mm\phi_t
    + \left[\frac{(\mm)^2}{\rho}+p(\rho)  \right]\phi_x
    -M \mm \phi \right\}\,dx dt \nonumber\\
     &   - \int_0^\infty   \left[p(\rho(b(t)-,t)) \phi(b(t),t) -p(\rho(a(t)+,t)) \phi(a(t),t) \right] \,dt=0\;,\label{S1:m-eq-phi}
\end{align}
for all test functions $\phi\in C^\infty_0(\R\times (0,\infty))$.

The following subsections are structured as follows: First, we investigate the Riemann solution to~\eqref{eq:system_Eulerian_K=1} 
around the vacuum and select the admissible solution for the flocking model in Subsection~\ref{S2new}. Next, in Subsection~\ref{S2.1}, 
we describe the problem in Lagrangian coordinates and in Subsection~\ref{S2.2}, we present useful analysis on its Riemann problem. 

\subsection{Admissible Riemann solution around the vacuum}~\label{S2new}
In this subsection, we study the Riemann solution to system~\eqref{eq:system_Eulerian_K=1} having one of the two states being a vacuum. 
Without loss of generality, let us consider the data at $x=b_0=0$ with left state 
$(\rho,\mm)=(\rho_\ell,\mm_\ell)$, and right state $(\rho,\mm)=(0,0)$, taking $\rho_\ell>0$ and $\vv_\ell=\mm_\ell/\rho_\ell$.

Thanks to the assumption \eqref{gamma=1}, the pressure satisfies $p'(\rho)>0$ that ensures the strict hyperbolicity 
of the system \eqref{eq:system_Eulerian} for $\rho>0$, with characteristic speeds $$\mu_\pm=\vv\pm\sqrt{p'(\rho)} = \vv\pm \alpha.$$ 

For simplicity let's assume from here and on that $M=0$ and consider the homogeneous system of~\eqref{eq:system_Eulerian_K=1}
together with the initial data 
\begin{equation}\label{RH1-dataapprox}
(\rho,\mm)(x,0)=\left\{\begin{array}{ll}
(\rho_\ell,\rho_\ell\vv_\ell) & x<0 
\\
(\brho,\mm(\brho))& x>0\,.
\end{array}\right.
\end{equation}
Here $\mm(\brho)=\brho \,\vv(\brho)$ 
and $\bvv=\vv(\brho)$ have to be determined by the wave curve under consideration and then take the limit
$\brho\to0+$.

The Rankine-Hugoniot conditions are:
\begin{align}\label{RH1}
    \sigma(\rho_\ell-\brho) &=\rho_\ell \vv_\ell-\brho\, \bvv\\[2mm] \label{RH2}
    \sigma(\rho_\ell \vv_\ell-\brho \bvv) &=\rho_\ell \vv_\ell^2-\brho\, \bvv^2+p(\rho_\ell)-p(\brho) 
\end{align}
connecting the initial data~\eqref{RH1-dataapprox}. The shock wave curves for each family are given by the expressions:   
\begin{equation*}
 S_1:\quad 
   \vv (\brho)=
		\vv_\ell - \sqrt{\dfrac{\left(p(\brho)-p(\rho_\ell)\right) \left(\brho - \rho_\ell \right) }{\brho\rho_\ell}
		} , \qquad  0<\rho_\ell\le\brho\,,   
				\end{equation*}
\begin{equation}\label{S2reply} 
 S_2:\quad  \vv(\brho) =  
    \vv_\ell -\sqrt{\dfrac{\left(p(\brho)-p(\rho_\ell)\right) \left(\brho - \rho_\ell \right) }{\brho\rho_\ell}
		}  \qquad  0<\brho\le\rho_\ell\,, 
\end{equation} 
while, the rarefaction wave curves for each family are:
\begin{equation} \label{R1reply}  
 R_1:\quad 
   \vv (\brho)=
		\vv_\ell -\int_{\rho_\ell}^\brho \frac{1}{s}\sqrt{p'(s)}\, ds, \qquad \qquad 0<\brho\le\rho_\ell, 
				\end{equation}
\begin{equation*} 
 R_2:\quad  \vv(\brho) =  
    \vv_\ell +\int_{\rho_\ell}^\brho \frac{1}{s}\sqrt{p'(s)} \, ds, \qquad\qquad  0<\rho_\ell\le\brho\;. 
\end{equation*} 
Now, for $\rho_\ell>0$ fixed and letting $\brho$ tend to zero, there are only two cases to examine: either the solutions is a 2-shock 
and $\vv(\brho)$ is given by~\eqref{S2reply} or a 1-rarefaction and $\vv(\brho)$ is given by~\eqref{R1reply}.


First, for the case of the $1$-rarefaction, see \eqref{R1reply}, we observe that as $\brho\to0+$, the integral
\begin{equation}
\label{intinfty}-\int_{\rho_\ell}^\brho \frac{1}{s}\sqrt{p'(s)} ds\to\int_0^{\rho_\ell}\frac{\alpha}{s} ds 
\end{equation}
does not converge for $p(s)=\alpha^2 s$, i.e $\gamma=1$. As a consequence, the speed of the rarefaction becomes infinite. 
However, $\mm(\brho)$ tends to $0$ since $\vv(\brho)\simeq -\ln\brho$.

On the other hand, in the case of a 2-shock with $\vv(\brho)$ given by~\eqref{S2reply}, we see from \eqref{RH1} that
as $\brho\to0+$, 
the speed becomes $\sigma=\vv_\ell$. 
Then, from~\eqref{S2reply}, we observe that 
\begin{equation}\label{vinfty} \lim_{\brho\to0+} \vv(\brho)=-\infty,\end{equation}
while again
\begin{equation}\lim_{\brho\to0+} \mm(\brho)=\lim_{\brho\to0+} \brho\, \vv(\brho)=0\;. \end{equation}
The last follows from \begin{equation} \vv(\brho) \simeq -\sqrt{\dfrac{p(\rho_\ell)\rho_\ell}{\brho\,\rho_\ell}}=
-\frac{\sqrt{p(\rho_\ell)}}{\sqrt{\brho}}\to-\infty\,.\end{equation}
\\
Moreover, as $\brho\to0+$
\begin{equation}\label{RHenergyterms}
    -\brho \, \vv^2(\brho)+p(\rho_\ell)\to0\;.
\end{equation}
Hence, taking the limit on each side of~\eqref{RH2}, we see that 
\begin{equation*}
\sigma(\rho_\ell \vv_\ell-\brho\, \bvv)=\rho_\ell \vv_\ell^2-\brho\, \bvv^2+p(\rho_\ell)-p(\brho) \qquad \xrightarrow{\brho\to0+}
     \qquad \sigma\rho_\ell\vv_\ell=\rho_\ell\vv_\ell^2
\end{equation*}
which holds true since $\sigma=\vv_\ell$. Thus, the Rankine-Hugoniot conditions hold true in the limit as $\brho\to0+$
and they converge to a limit which is meaningful for the limiting solution 
 \begin{equation}\label{def:RP-sol-tilde-rho-m}
(\widetilde \rho,\widetilde \mm)(x,t):=\left\{\begin{array}{ll}
(\rho_\ell,\mm_\ell
)\,, &\quad x< t \vv_\ell  \\
(0,0)\,,&\quad x>t \vv_\ell\,.
\end{array}
\right.
\end{equation}
 
For our problem capturing the flocking behavior, we consider as admissible the solution being a $2$-shock and hence, 
satisfying the boundary condition~\eqref{ad-hoc vacuum}. 
We motivate this admissibility criterion by the fact that the speed $\sigma$ remains bounded, in contrast with the choice of the rarefaction 
as explained above. This is due to the choice of the pressure $p(\rho)$ to be a linear function of $\rho$. To further motivate this
admissibility criterion, we bring up the following simple case: If the initial data are constant on their support $I_0$, it is natural to
consider the solution as the initial data simply translated with the constant velocity $\vv$ 
(the value of $\vv_0$ in the support; 
here denoted by the $\vv_\ell$ on the left of $b_0$) that is represented by the choice of a shock; see~\eqref{eq:elem-solutions}. 
Thus the $1$-rarefaction solution is excluded. Similarly, at $x=a_0$, the Riemann problem with vacuum is solved by a $1$-shock. 
In other words, we choose the shock solution at both edges to satisfy condition~\eqref{ad-hoc vacuum} since the goal is to construct a weak solution having finite support. 

It is interesting to mention that in Liu-Smoller\cite{LS80}, the authors deal with a similar issue but in the context of gas dynamics that is different from flocking. 
More precisely, in Ref.~\cite{LS80} the shock solution is excluded in gas dynamics when the authors reach \eqref{vinfty} 
(see the beginning of Sect.2 in Ref.~\cite{LS80}) 
while they consider the variables $(\rho,\vv)$ as their states, in contrast to $(\rho,\mm)$, although the Rankine-Hugoniot conditions hold true in the $(\rho,\mm)$ variables. At the same time, they assume that 
\begin{equation}\label{intg=1infty}
\int_0^{\rho_\ell}\frac{1}{s}\sqrt{p'(s)} ds<\infty\,,
\end{equation}
which is satisfied by $\gamma>1$, in contrast to our case~\eqref{intinfty}. Under the assumption \eqref{intg=1infty}, the rarefaction curve $R_1$ intersects the $\vv-$axis in the $\rho-\vv$ plane. 
This immediately provides a rarefaction wave fan with the leading front having a finite speed and the rarefaction wave becomes the Riemann solution for the problem in Ref.~\cite{LS80}.


\begin{remark}
Let us clarify that the pressure is taken to have the form ~\eqref{gamma=1} as it turns out from the analysis in Ref.~\cite{KMT15}. In terms of our analysis, one could assume having isentropic pressure, i.e. $\gamma>1$, however several points should be adjusted in this work. The most critical ones are: (i) a limitation on the size of the total variation of the initial data should be added in contrast to having  large data in this paper; (ii) the Riemann solution for $\gamma>1$ around vacuum may consists of a rarefaction with finite speed (see Liu-Smoller\cite{LS80}) in contrast to the infinite speed. However, the rarefaction case may be chosen to be excluded again as here since it does not fit into the flocking behavior, reducing again the Riemann solution to a shock; (iii) the Lyapunov functionals incorporated in the proofs would not only involve the linear parts corresponding to the total variation but a combination of this with the quadratic Glimm-type functional. However, by taking into account these technical points, we expect that results similar to Theorems~\ref{Th-1} and~\ref{Th-2} hold for the isentropic case.
\end{remark}

Now we turn our attention to the proper notion of weak solution satisfied in the limit as $\brho\to0+$ by the solution obtained 
in the limit with the $2$-shock, that is, \eqref{def:RP-sol-tilde-rho-m}.
For $\brho \in (0,\rho_\ell)$, let
\begin{equation}\label{def:RP-sol-rho-m}
(\rho, \mm)(x,t)=\left\{\begin{array}{ll}
(\rho_\ell,\mm_\ell
)\,, & \quad x< \sigma t \\
(\brho, \mm(\brho))\,,&\quad  x>\sigma t
\end{array}
\right.
\end{equation}
be the solution to~\eqref{eq:system_Eulerian_K=1} with $M=0$ and \eqref{RH1-dataapprox} where $\mm(\brho) = \brho \, \vv(\brho)$ 
is obtained from \eqref{S2reply} and $\sigma$ given by \eqref{RH1}, \eqref{RH2}.     
As $\bar \rho\to 0$, $(\rho,\mm)$ converges to the function \eqref{def:RP-sol-tilde-rho-m}.
The weak formulation of  \eqref{eq:system_Eulerian_K=1}$_2$ with $M=0$, for $\brho>0$ reads as
\begin{equation*}
\iint \left\{\mm\,\phi_t + \left(\frac{\mm^2}{\rho} + p(\rho) \right)\phi_x \right\}\; dxdt=0
\end{equation*}
for all $\phi \in C_0^\infty(\R\times (0,+\infty))$. 
In the limit $\brho\to 0$, all the terms converge to the corresponding term of $(\widetilde \rho, \widetilde \mm)$, 
except $\mm^2/\rho$ which is undefined for $x > \vv_\ell t$. A direct computation shows that
\begin{equation*}
    \iint_{\{x> \sigma t \}} \frac{\mm^2}{\rho}\phi_x \; dxdt 
    ~~\xrightarrow{\bar \rho\to0+}~~ - p(\rho_\ell) \int_0^{+\infty} \phi( \vv_\ell t,t)\;dt \,.
\end{equation*}

Therefore, the limit function $(\widetilde \rho, \widetilde \mm)$ satisfies the integral identities \eqref{S1:rho-eq-phi2} and
\begin{equation*}
 \int_0^{+\infty} \left\{\int_{\R}\widetilde\mm \,\phi_t \right.  +  \left. p(\widetilde\rho)\phi_x \;dx + 
    \int_{x< t \vv_\ell}\left(\frac{\widetilde\mm^2}{\widetilde\rho}\right)\phi_x\;dx -p(\rho_\ell)\phi(v_\ell t,t) \right\}\; dt=0\,.
\end{equation*}
One can express this last identity by a generalized definition of the momentum that includes a \emph{Dirac delta} that is concentrated 
along the discontinuity $\gamma=\{(\vv_\ell t,t);\ t\ge 0\}$ as it is done with $\widehat \mm$ in~\eqref{def:m-hat} for the Cauchy problem.
We also recall that in the above computations of this subsection, we take $M=0$ and therefore, the source term is not present.

\subsection{Recasting into Lagrangian coordinates}\label{S2.1}
We perform the change of variables from Eulerian to Lagrangian coordinates
\begin{align*}
&(x,t) \mapsto \left(y=\chi(x,t),\t=t \right)\,,\\
&\chi(x,t) = \int_{-\infty}^x \rho(x',t)\,dx' \in [0,M]
\end{align*}
for solutions that conserve mass, i.e~\eqref{cons-of-mass}, and momentum, i.e.~\eqref{eq:zero-average-momentum},
with $M$ being the total mass (see \eqref{eq:def-M}). Since $\partial_x \chi= \rho\ge0$, then $x\mapsto  \chi(x,t)$ is non-decreasing 
and it maps $\R$ into $[0,M]$. Under the structure of the solution given in \eqref{solution-structure}, its pseudoinverse
$y\mapsto\chi^{-1}(y,t)$ is injective from $(0,M)$ to $I(t) = (a(t),b(t))$. In addition, the derivatives change as follows:
\begin{equation*}
\partial_\t = \partial_t + \vv \partial_x\,,\qquad \partial_y = \frac 1 \rho \partial_x = u \partial_x\,,\qquad u~\dot =~1/\rho\,,\qquad
v(y,t)\dot=\vv(x,t)\;.
\end{equation*}
Recalling the conservation of mass and momentum in \eqref{cons-of-mass} and \eqref{eq:zero-average-momentum}, 
we rewrite the two integral terms as follows:
\begin{align*}
&\int_\R \rho(x',t)\,dx' = \int_0^{M} \,dy' = M
\\
&\int_{I(t)} \mm(x',t)\,dx' +P_b(t)-P_a(t)= \int_0^{M} v(y',t) \,dy' +P_b(t)-P_a(t)= M_1 = 0 \,.
\end{align*}
Therefore we obtain the following system:
\begin{equation}\label{eq:system_Lagrangian}
\begin{cases}
\partial_\t u -  \partial_y  v  = 0, &\\
\partial_\t v  +  \partial_y (\alpha^2/u) =   - M v\,, & 
\\ 
\end{cases} 
\end{equation}
for the unknown $(u(y,t), v(y,t))$ while $(y,t)\in (0,M)\times [0,\infty)$\,.

\par\smallskip\noindent
Recalling \eqref{hyp-init_data}, the initial data $\rho_0$ is bounded and bounded away from 0; 
this ensures that the change of variable $$x\mapsto \chi(x,0)= \int_{a_0}^x \rho_0(x')\,dx'$$ 
is bi-Lipschitz continuous from $(a_0,b_0)$ to $(0,M)$\,.
Therefore, the conditions \eqref{hyp-init_data}, \eqref{hyp-init_data-v}, \eqref{eq:def-M1} with $M_1=0$,
given in terms of the initial data $(u_0,v_0)=(\rho_0^{-1},\vv_0)$, lead to:
\begin{equation}\label{eq:init-data-lagr}
(u_0,v_0)\in BV(0,M)\,,\qquad \essinf_{(0,M)} u_0 >0\,, \qquad  \int_0^M v_0(y)\, dy=0\,.
\end{equation}

Furthermore, at the boundaries $y=0$, $y=M$, we consider \emph{non-reflecting} boundary conditions for system~\eqref{eq:system_Lagrangian} of the following type:
\begin{equation}\label{eq:bc-lagrangian}
    \begin{aligned}
    y=0: & \mbox{ \emph{any} state $(u,v)(0+,t)\in (0,+\infty)\times \R$ is admissible} \\
    y=M: &  \mbox{ \emph{any} state $(u,v)(M-,t)\in (0,+\infty)\times \R$ is admissible\,. }
    \end{aligned}
\end{equation}
The choice of this type of boundary condition can be better illustrated in terms of the Eulerian variables $(\rho,\mm)$,
in the proof of Theorem~\ref{Th-1}, in the Section~\ref{Sect:3}.

We have thus reformulated our problem in Lagrangian variables and the aim is to construct entropy weak solution to the boundary problem~\eqref{eq:system_Lagrangian} with the data described above that conserve mass and momentum in order to hope for retrieving solutions to~\eqref{eq:system_Eulerian} or~\eqref{eq:system_Eulerian_M-M1} according to Theorem~\ref{Th-1}. As already mentioned in Introduction the equivalence of weak solutions between the systems is not immediate from Ref.~\cite{W87} due to the combination of the presence of vacuum and the finite mass.

\begin{theorem}\label{newthm}
Consider the initial-boundary problem~\eqref{eq:system_Lagrangian},~\eqref{eq:bc-lagrangian} with initial data $(u_0,v_0)$ satisfying \eqref{eq:init-data-lagr}. Then, there exists a function $(u,v):(0,M)\times [0,+\infty)\to  \R^2$ such that $t\mapsto (u,v)(\cdot,t) \in L^1\left(0,M\right)$ is continuous on $[0,+\infty)$, it satisfies 
\begin{equation}\label{eq:init-data-uv}
(u,v)(\cdot,0) = (u_0,v_0)    
\end{equation}
and it is a distributional solution to~\eqref{eq:system_Lagrangian} 
in the open set $(0,M)\times(0,+\infty)$\,.
Moreover, there exist positive constants $u_{inf}$, $u_{sup}$ such that for all $t>0$
\begin{equation}\label{eq:Linfty-bound-uv-new}
0<u_{inf}\le u(y,t) \le u_{sup}\,,\quad 
    |v(y,t)|\le \widetilde C_0
\qquad \mbox{for a.e.}\ y
\end{equation}   
for some constant $\widetilde C_0>0$. In addition, for every convex entropy $\tilde\eta$ for the system \eqref{eq:system_Lagrangian}, with corresponding entropy flux $\tilde q$, the following inequality holds in ${\DD}'((0,M)\times(0,+\infty))$:
\begin{equation*}
    \partial_t \tilde\eta(u, v) + \partial_x \tilde q(u, v)\le -M \tilde\eta_v v\,.
\end{equation*}
\end{theorem}

The proof of Theorem~\ref{newthm} is established at the end of Section~\ref{Sect:2}.

\subsection{Characteristic curves and solution to the Riemann problem}\label{S2.2}

To present the characteristic curves associated with system~\eqref{eq:system_Lagrangian}, let us first define the function
\begin{equation}\label{def:sigma}
\sigma(u) = - p(1/u) = -\alpha^2/u \,.  
\end{equation}
Then, characteristic speeds for system \eqref{eq:system_Lagrangian} are $\lambda=\pm \alpha/u$, and the ra\-re\-faction-shock curves, issued at a point $(u_\ell,v_\ell)$, are given by:

\smallskip\noindent
$\bullet$\quad curve of the \textbf{first} family:
\begin{equation} 
 \begin{cases}  
   u<u_\ell, \qquad v = 
		v_\ell - \sqrt{(u-u_\ell) \left(\sigma(u) - \sigma(u_\ell) \right) } = v_\ell - \alpha \left(  \sqrt{\frac{u_\ell}{u}}- \sqrt{\frac{u}{u_\ell}} \right) , \\  
   u>u_\ell, \qquad v =  
    v_\ell + \int_{u_\ell}^u \sqrt{\sigma'(s)}\, ds   = v_\ell + \alpha \ln \left(\frac u {u_\ell}\right);
 \end{cases}
    \label{RH-1}  
\end{equation} 

\medskip\noindent
$\bullet$\quad curve of the \textbf{second} family:
\begin{equation}   
 \begin{cases}
   u<u_\ell, \qquad v = 
   v_\ell - \int_{u_\ell}^u \sqrt{\s'(s)}\, ds   =  v_\ell - \alpha \ln \left(\frac u {u_\ell}\right), \\  
   u>u_\ell, \qquad v = 
		v_\ell - \sqrt{(u-u_\ell) \left(\sigma(u) - \sigma(u_\ell) \right) } = v_\ell -  \alpha \left(  \sqrt{\frac{u}{u_\ell}} - \sqrt{\frac{u_\ell}{u}} 
		 \right).  
 \end{cases} 
    \label{RH-2}  
\end{equation}  
\par\medskip\noindent

\begin{remark}(Propagation speed)\label{rem:2.1}
From \eqref{RH-1} and \eqref{RH-2}, the propagation speed $\Lambda$ of a shock depends only on $u_\ell$, $u$ and satisfies
\begin{equation*}
    |\Lambda| = \left|\frac{v(u)-v_\ell}{u-u_\ell}\right| = \frac{\alpha}{
    \sqrt{u u_\ell}}\,.
\end{equation*}
Similarly, let $(u_\ell,v_\ell)$, $(u,v)$ be connected by a rarefaction wave. By \eqref{RH-1}$_2$ and \eqref{RH-2}$_1$, the Rankine-Hugoniot speed $\Lambda$ 
corresponding to the (non-entropic) jump $(u_\ell,v_\ell)$, $(u,v)$ satisfies
$$
|\Lambda| = \left|\frac{v-v_\ell}{u-u_\ell}\right| =\alpha \frac{1}{\tilde u}\qquad \tilde u\in [\min\{u_\ell,u\}, \max\{u_\ell,u\} ]\,.
$$
\end{remark}

The two curves \eqref{RH-1}, \eqref{RH-2} can be conveniently parametrized, in a unified way, as follows: 
\begin{equation}\label{eq:lax13}
u \mapsto v(u; u_\ell,v_\ell) = v_\ell + 2\alpha h(\eps_j)\,,\qquad u>0\,,\quad j=1,2\;,
\end{equation}
where the 
$\eps_j$ 
are defined by
\begin{equation}\label{eq:strengths}
\eps_1=\frac{1}{2}\ln\left(\frac{u}{u_\ell}\right),\qquad \eps_2=\frac{1}{2}\ln\left(\frac{u_\ell}{u}\right)
\end{equation}
and the function $h$ is given by 
\begin{equation}\label{h}
h(\eps)= \begin{cases}
\eps& \mbox{ if } \eps \ge 0\,,\\
\sinh \eps& \mbox{ if } \eps < 0\,.
\end{cases}
\end{equation}
Notice that $\eps>0$ corresponds to the rarefaction curve, whereas $\eps<0$ corresponds to the shock curve. Also, assume that
the states $(u_\ell,v_\ell)$ and $(u,v)$ are connected by a simple wave of size $\eps$ with $j\in\{1,2\}$, that is \eqref{eq:lax13} holds. 
Then one has 
\begin{equation}\label{eq:strengths-deltav}
    |\eps|=\frac 12 \left|\ln(u) - \ln (u_\ell) \right|\,,\qquad |h(\eps)|=\frac 1 {2\alpha} |v-v_\ell|\,.
\end{equation}
We will refer to $\eps_j$, $j=1,2$ given in \eqref{eq:strengths}, as the \textit{strength} of a wave of the $j$-characteristic family\,. 

\smallskip
In the next proposition, stated for instance in Ref.~\cite{AC_SIMA_2008} (see also Ref.~\cite{Nishida68}), we solve the Riemann problem for the homogeneous system
\begin{equation}\label{eq:system-nosource-gamma1}
\partial_t u -  \partial_y  v  = 0, \qquad
\partial_t v  +  \partial_y  (\alpha^2/u) =  0
\end{equation}
and initial data
\begin{equation}\label{eq:Riemann-data}
(u,v)(x,0)=\begin{cases}
(u_\ell,v_\ell) & x<0\;,\\
(u_r,v_r) & x>0\,.
\end{cases}
\end{equation}

\begin{proposition}\label{prop:2.1} 
For any pair $(u_\ell,v_\ell)$, $(u_r,v_r)\in (0,+\infty)\times \R$, the Riemann problem  
\eqref{eq:system-nosource-gamma1}--\eqref{eq:Riemann-data} has a unique solution $(u,v)(x,t)\in (0,+\infty)\times \R$ that consists of simple Lax waves. 

Moreover, let $\eps_j$ be the strength of the $j=1$, $2$ wave. Then one has
\begin{align}\label{eq:one}
\eps_2 - \eps_1 &= \frac 12 \ln \left(\frac{u_\ell}{u_r} \right) =  \frac 12 \ln \left(\frac{p_r} {p_\ell}\right)\\
\label{eq:two}
h(\eps_1) + h(\eps_2) &= \frac{v_r - v_\ell}{2\alpha}
\end{align} 
where $p_{r,\ell} = p(u_{r,\ell}) = \alpha^2 / u_{r,\ell}$.

Finally, the following property holds,
\begin{equation}\label{eq:bound-on-size-of-Rp}
|\eps_1| + |\eps_2|\le \max\left\{ \frac 12 \left| \ln \left(\frac{u_r} {u_\ell}\right)\right|, \frac{|v_r - v_\ell|}{2\alpha} \right\}\,.
\end{equation}
\end{proposition}

\begin{proof}
We report the proof of \eqref{eq:one}, \eqref{eq:two} for later use. If $(u^*,v^*)$ denotes the intermediate value of the solution 
to the Riemann problem, then 
$$
 \eps_1  = \frac{1}{2}\ln\left(\frac{u^*}{u_\ell}\right)\,,\qquad  \eps_2=\frac{1}{2}\ln\left(\frac{u^*}{u_r}\right)
$$
and therefore
$$
\eps_2 - \eps_1 = \frac{1}{2}\ln\left(\frac{u_\ell}{u_r}\right)\,,
$$
that proves \eqref{eq:one}\,. Regarding  \eqref{eq:two}, it is enough to apply \eqref{eq:lax13} and get
$$
 v^* = v_\ell + 2\alpha h(\eps_1)\,,\qquad v_r = v^*+ 2\alpha h(\eps_2)\,.
$$
In conclusion we find that
$$
v_r - v_\ell =  (v_r - v^*) + (v^* - v_\ell) = 2\alpha h(\eps_1) + 2\alpha h(\eps_2)
$$
and \eqref{eq:two} follows\,.

Finally, to prove \eqref{eq:bound-on-size-of-Rp}, we observe that:

\smallskip
- if $\eps_1\eps_2\le0$ then $|\eps_1| + |\eps_2| = |\eps_2 - \eps_1|$ and \eqref{eq:bound-on-size-of-Rp} follows thanks to \eqref{eq:one};

\smallskip

- if $\eps_1\eps_2 >0$, instead, we use \eqref{eq:two} and the fact that $|x|\le h(x)$ for all $x\in\R$.\\
The proof is complete\,.
\end{proof}


%
\Section{Global existence of weak solutions}\label{Sect:2}

In the following subsections, we prepare the ground to prove Theorem~\ref{newthm}. More precisely, in Subsection~\ref{S2.3}, we construct approximate solutions to system~\eqref{eq:system_Eulerian_K=1} using the 
front-tracking algorithm in conjunction with the operator splitting. Then in the following Subsections \ref{S2.4}--\ref{subsec:vert-traces}, we introduce the Lyapunov functionals that allow us to obtain uniform bounds on the total variation and also estimate the vertical traces. We conclude 
in Subsection~\ref{S2.8} with the convergence of the approximate solutions. The analysis is employed in Section~\ref{Sect:3} towards the proof of Theorem~\ref{Th-1}.

\subsection{Approximate solutions}\label{S2.3}
Here, we construct approximate solutions $(u^\nu, v^\nu)$, $\nu\in\mathbb{N}$, to the initial-boundary value problem \eqref{eq:system_Lagrangian}\,. 
The construction is the same as in Ref.~\cite{AmadoriGuerra01} and follows the standard application of the operator splitting scheme in conjunction 
with the front tracking algorithm, together with the adaptation to the non-reflecting boundary conditions at $y=0$, $y=M$. 
Below we describe the construction in three steps: 

\medskip
\paragraph{{\bf Step 1.}}\quad Choose approximate initial data $(u^\nu_0, v^\nu_0)$, $\nu\in\N$, which are piecewise constant and  satisfy
\begin{equation*}
\tv \left(u^\nu_0, v^\nu_0\right) \le \tv \left(u_0, v_0\right)\,,\qquad \| \left(u^\nu_0, v^\nu_0\right) -  \left(u_0, v_0\right)\|_\infty \le \frac 1 \nu\,.
\end{equation*}
For later use, define the approximate states at the boundaries
\begin{equation}\label{def:u0-nu_uM-nu}
\tilde{u}_0^\nu:=\displaystyle\lim_{y\to0+}u^\nu(y,0)\qquad 
\tilde{u}_M^\nu:=\displaystyle\lim_{y\to M-}u^\nu(y,0)
\end{equation}
and
\begin{equation}\label{def:v0-nu}
    \tilde{v}_0^\nu :=\displaystyle\lim_{y\to0+}v^\nu(y,0)\,.
\end{equation}
Observe that, by construction, $\tilde{u}_0^\nu$ and $\tilde{u}_M^\nu$ are both positive and, as 
${\nu\to\infty}$,
\begin{equation}\label{eq:limits-init-data}
\tilde{u}_0^\nu \to u_0(0+)\,,\qquad \tilde{u}_M^\nu \to  u_0(M-)\,.
\end{equation}

As a consequence of \eqref{eq:init-data-lagr}$_3$, one has that
\begin{equation*}
\left|\int_0^M v^\nu_0(y)\,dy\right| \le \frac M\nu \to 0 \qquad \nu\to\infty\,.
\end{equation*}

\medskip
\paragraph{{\bf Step 2.}}\quad Fix a time step $\DT=\DT_\nu>0$ and set $t^n=n\DT$, $n=0,1,2,\dots$. Then, fix also a parameter $\eta=\eta_\nu>0$, that controls: 
\begin{itemize}
\item the size of rarefaction fronts, and 
\item the errors in the speeds of shocks and rarefaction fronts.
\end{itemize}
Allow now both parameters $\DT_\nu$, $\eta_\nu$ to converge to 0 as $\nu\to\infty$.

\medskip
\paragraph{{\bf Step 3.}}\quad The approximate solution $(u^\nu, v^\nu)$ of~\eqref{eq:system_Lagrangian} is obtained as follows:
\begin{enumerate}
    \item[(i)] On each time interval $[t^{n-1},t^n)$, the approximate solution $(u^\nu, v^\nu)$ is the $\eta$-front tracking approximate solution to the homogeneous system of~\eqref{eq:system_Lagrangian} with initial data $(u^\nu(y,t^{n-1}+), v^\nu(y,t^{n-1}+))$. We recall that $\eta$-front tracking approximate solutions (see Def. 7.1 in Ref.~\cite{Bressan_Book}) are piecewise constant functions with discontinuities along finitely many lines in the $(t,y)$ half-plane and interactions between two incoming fronts. Indeed, by modifying the wave speeds slightly, by a quantity less than $\eta$, we can assume that there are only interactions of two incoming fronts. 
    \par\noindent
    To give more details of this construction, solve the Riemann problems at $t=t^{n-1}$ around each discontinuity point of the initial data $(u^\nu(y,t^{n-1}+), v^\nu(y,t^{n-1}+))$. Then, retain shocks as obtained in the Riemann solution, but approximate rarefactions by fans of fronts each one of them consisting of  two constant states connected by a single discontinuity of strength less than $\eta$. More precisely, each rarefaction of strength $\eps$ is approximated by $N$ rarefaction fronts, with $N=[\eps/\eta]+1$, each one of strength $\eps/N<\eta$ and with speed to be equal to the characteristic speed of the right state.  
   Then prolong the approximate solution until some wave fronts interact at a point, or some wave reaches one of the boundaries. 
    \par\noindent
  - At the interaction, there are exactly two incoming fronts as mentioned and the approximate solution is prolonged by solving the Riemann problem and approximating rarefactions, if they arise in the solution, as just described. Applying this Riemann solver at interaction times, obtain the approximate solution for $t^{n-1}\le t<t^n$.
   \par\noindent
  - At times $t=\tau$ that a front meets the boundaries, i.e. $y=0$ or $y=M$, then that wave front just exits the domain 
and it does not play a role in the further construction for $t>\tau$. In other words, we say that the boundary absorbs the wave fronts when the fronts reach the boundary. 
  
    \item[(ii)]  At the time $t=t^n$, apply the operator splitting technique to system~\eqref{eq:system_Lagrangian} and define
\begin{equation}\label{eq:u-v_fractional-step}
u^\nu(y,t^n+) = u^\nu(y,t^n-)\,,\qquad  v^\nu(y,t^n+) = v^\nu(y,t^n-) \left(1-M\DT\right)
\end{equation}
taking into account the source term. As in (i), by changing possibly the speed of the fronts, 
there is no interaction of fronts at the times steps $t^n$. 
\item[(iii)] Proceed by iteration of (i) and (ii), with $n\in\mathbb{N}$.
\end{enumerate}

In this way, we obtain the approximate solution as long as the iteration at Step 3 can be carried out. Actually, one can continue with this construction 
as long as the interactions do not accumulate. Let us mention that, by Proposition~\ref{prop:2.1}, we can always solve the Riemann problem around each interaction, 
even if the initial data are arbitrarily large.

It should also be noted that, in general when applying the Riemann solver, except of shocks and rarefaction fronts, there are also the so-called non-physical fronts. 
Here, we adopt the simplified version of front tracking algorithm for $2\times 2$ systems that does not require the presence of non-physical fronts as developed in Ref.~\cite{AmadoriGuerra01}.

\begin{remark}(Wave-fronts at the boundaries)\label{rem:boundary-fronts}\rm \quad
At time $t$, not only the fronts within the interval $(0,M)$ exist, but also the fronts that reached the boundary $y=0,M$ at times $\tau$, $\tau\le t$. 
These fronts that reached the boundary at $\tau$, $\tau\le t$, exist for all times greater than the exit time $\tau$, 
continue having the strength $\eps$ with which they touched the boundary at $\tau$ and are considered as standby fronts. 

Another way to interpret them is to fix a constant wave speed $\hat \lambda>0$, and proceed as follows: when a front 
reaches the boundary $y=0$ at time $\tau$, then prolong it on $y<0$ with speed $-\hat \lambda$, and similarly if a front 
reaches $y=M$, then prolong it on $y>M$  with constant speed $\hat \lambda$. 
\end{remark}

\begin{remark}\label{Rmk:2.3}
For the special case that the initial data are  
$$
\rho_0(x)=\bar\rho>0,\qquad \mm_0(x)= \vv_0(x)= 0,\qquad x\in[a_0,b_0]
$$
with $\bar\rho>0$ a positive constant, we have in Lagrangian coordinates
$$
u_0(y)=\frac{1}{\bar\rho},\qquad v_0(y)=0,\qquad y\in[0,M]\;.
$$
Note that from~\eqref{eq:def-of-m}, the initial bulk  is $q=0$ in this case and also the approximate solution remains constant for all times, i.e. $u^\nu(y,t)=\frac{1}{\bar\rho}$, $v^\nu(y,t)=0$, for all $t$ and $y\in[0,M]$.
\end{remark}

\subsection{The linear functionals}\label{S2.4} Here, we define Lyapunov 
functionals that assist in the proof of global existence to system~\eqref{eq:system_Lagrangian}.
From now on and until Lemma~\ref{Lemmauinfsup}, the analysis will concern any given approximate solution $(u^\nu,v^\nu)$ with $\nu$ fixed, and we will omit 
the $\nu$ dependence for simplicity. 

Fix a time $t$ at which no wave interaction occurs and different from times steps $t^n$. Let $\mathcal{J}(t)$ be the set of fronts that exist at time $t$ 
in the approximate solution constructed in Subsection~\ref{S2.3}. This consists of the fronts contained in $(0,M)$ together with the standby fronts that reached 
the boundary at times $\tau$, $\tau\le t$.

Next, let $\{y_j\}_1^N$ be the points in $(0,M)$ such that
\begin{equation}\label{eq:disc-points}
  0<y_1< y_2<\ldots< y_{N(t)}<M  
\end{equation}
and at those points 
$(u,v)(\cdot,t)$ is discontinuous, for some integer $N=N(t)$ depending on $t$. Also, let $\eps_j$ be the corresponding strength 
of the front located at $y_j$.  Then we define the \emph{total linear functional}
\begin{equation}\label{def:L}
L(t) = \sum_{\beta\in J(t)} |\eps_\beta|\,,
\end{equation}
and the \emph{inner linear functional}
\begin{equation}\label{def:Lin}
L_{in}(t) = \sum_{j=1}^{N(t)} |\eps_j|\,.
\end{equation}
The total linear functional $L(t)$ counts the total strength of the fronts that are present at time $t$ including those that already reached the boundary at $\tau$, $\tau\le t$,
and stay put retaining their strength unchanged, while the inner linear functional $L_{in}(t)$ counts only the total strength of the fronts that are present in $(0,M)$ at time $t$. 

We also define the \emph{outer linear functionals} $L_{0,out}(t)$ and $L_{M,out}(t)$, that are, the total strength of the standby fronts that reached the boundary 
$y=0$, $y=M$ respectively at times $\tau$, with $\tau\le t$. In fact, these functionals have initial value zero, i.e. $L_{0,out}(0+)=0= L_{M,out}(0+)$, and for $t>0$,
\begin{equation*}
L_{0,out}(t-) = \sum_{\tau<t} \Delta L_{0,out}(\tau)
\end{equation*}
with $\Delta L_{0,out}(\tau)=|\eps|$ if a front of strength $\eps$ reached the boundary $y=0$ at time $\tau$. A similar formula holds 
for $L_{M,out}(t-)$ along the boundary $y=M$. 
Then, immediately, we have
\begin{equation}\label{def:L-Lout}
L(t)=L_{in}(t)+L_{0,out}(t) + L_{M,out}(t),\quad \forall t\,.
\end{equation}


In next lemma, we show the relation between $L_{in}$ and the total variation in $x$ of the approximate solution.

\begin{lemma}\label{prop:equivalence-Lin}
The following identities hold,
\begin{equation}\label{eq:identity-tvlnu}
\frac 12 \tv \{\ln(u)(\cdot,t)\} =  \frac 12 \tv \{\ln(\sigma(u(\cdot,t))\} = L_{in}(t)\,,
\end{equation}
where $\sigma$ is defined at \eqref{def:sigma}, and
\begin{equation}\label{eq:tv-v}
\tv \{v(\cdot,t)\} \le 2\alpha \cosh(c)\, L_{in}(t)
\end{equation}
where $c=\max{|\eps|}$ is the maximum of the sizes of the waves on $(0,M)$ at time $t$.
\end{lemma}

\begin{proof}
By recalling the identities in \eqref{eq:strengths-deltav} and the definition $\sigma(u)=-\alpha^2/u$, we deduce immediately that \eqref{eq:identity-tvlnu} holds, and that  

\begin{equation*}
    \tv \{v(\cdot,t)\} = 2\alpha \sum_{j=1}^{N(t)} |h(\eps_j)|\,.  
\end{equation*}
By the definition of the function $h$ in \eqref{h}, we use the elementary inequality $0\le \sinh(x)\le x \cosh(c) $ for $x\in[0,c]$, 
to find that
\begin{align*}
   \sum_{j=1}^{N(t)} |h(\eps_j)|\le \cosh(c)   \sum_{j=1}^{N(t)} |\eps_j| = \cosh(c) L_{in}(t)\,.
\end{align*}
This completes the proof of \eqref{eq:tv-v}.
\end{proof}

\subsection{Bounds on the total variation} 
Here, we establish bounds on the approximate solution valid for all times and independent of $\nu$ using the functionals already introduced. To study the time variation of the functionals, we refer to a result from Ref.~\cite{AmadoriGuerra01} on the change of a strength $\eps^-$ across a time step. For later convenience, we state and prove it under the current notation.

\begin{proposition}\label{S2:prop:estimate-time-step}
Let $\eps_2^-$ be the size of a 2-wave before a time step $t^n$, $n\ge 1$, and let $\eps_1^+$, $\eps_2^+$ be the sizes of the outgoing waves. 
\begin{enumerate}
\item[$(a)$] The following identities hold:
\begin{align}\label{eq:one-timestep}
\eps_2^+ - \eps_1^+ &= \eps_2^- 
\\
\label{eq:two-timestep}
h(\eps_1^+) + h(\eps_2^+) &= h(\eps_2^-)\left( 1 - M\DT \right)\,. 
\end{align} 
where $h$ is given at \eqref{h}.
\item[$(b)$] The following sign rules hold: $\eps_1^+ \eps_2^- <0$, $\eps_2^+\eps_2^->0$.
\end{enumerate}
A similar conclusion holds for any 1-wave of size $\eps_1^-$.
\end{proposition}
The proof is deferred to Appendix~\ref{subsec:app1}.

In addition we need to prove an estimate on the size of the reflected waves at the time step, that will depend on the parameters $q$ and $M$ given at~\eqref{eq:def-of-m}, and~\eqref{eq:def-M}, respectively. 
This is stated in the next proposition, that refines the statement of Proposition~\ref{S2:prop:estimate-time-step}.

\begin{proposition}\label{prop:estimate-time-step}
Let $q>0$ and $\eps_2^-$ the size of a 2-wave before a time step $t^n$,  $n\ge 1$, with $|\eps_2^-|\le q$. Suppose that $\eps_1^+$, $\eps_2^+$ 
are the sizes of the outgoing waves. Then
\begin{align}\label{time-step-reflected-wave}
c_1(q) &\le \frac{|\eps_1^+|}{{M}\DT |\eps_2^-| }  \le C_1^\pm(q) 
\end{align}
where
\begin{equation*}
c_1(q)= (1+\cosh(q))^{-1}\,,\qquad C_1^\pm(q) = 
\begin{cases} 
\frac{1} {2}  & \mbox{ if }\eps_2^->0\\
\frac{\cosh(q)} {2} & \mbox{ if }\eps_2^-<0\,.
\end{cases}
\end{equation*}
\smallskip
A similar conclusion holds for any 1-wave of size $\eps_1^-$.
\end{proposition}
The proof is deferred to Appendix~\ref{subsec:app2}.

\par\smallskip\noindent
\textbf{Notation.}\quad Here we introduce a notation that will be extensively used in the rest of the section. Given a function of time $t\mapsto G(t)$,
by $\Delta G$ we denote the change of the function $G$ around $t$, i.e. $\Delta G (t)=G(t+)-G(t-)$. Similarly, given a function of space $y\mapsto g(y)$,
we denote by $\Delta g(y)$ the variation of $g$ at $y$ as follows, $\Delta g(y)= g(y+)-g(y-)$.

\smallskip
The next lemma provides a uniform bound on the linear functional $L(t)$, 
and hence on the total variation of the approximate solutions, 
as soon as the approximate solution is defined.

\begin{lemma}\label{lem:bounds-on-bv}
  Suppose that the approximate solution $(u^\nu, v^\nu)=(u, v)$ is defined for $t\in[0,T]$. Then $L_{in}(t)$, $L(t)$ are 
  non-increasing in time, and
  \begin{equation}\label{LLin-decreases}
L_{in}(t)\le L(t)\le L(0+) \le  q\qquad
\forall\, t\in[0,T]\,,
\end{equation}
where $q$ is defined at \eqref{eq:def-of-m}.

Moreover, there exists a constant $C_1>0$ independent of $\nu$ and $t$ such that
\begin{equation}\label{eq:unif-bound-tv-v}
\tv \{v(\cdot,t)\} \le C_1 \,, \qquad t\in[0,T]\,.
\end{equation}
\end{lemma}

\begin{proof}
Following the analysis in Ref.~\cite{AmadoriGuerra01}, both in between the time steps and at the time steps $t=t^n=n\DT$, one has
\begin{equation}\label{L-decreases}
\Delta L(t)\le 0 \qquad \Delta L_{in}(t)\le 0 \qquad \forall\, t >0
\end{equation}
as soon as the approximate solution is defined, that is, as soon as no accumulation of interaction points occur. In particular, 

$\bullet$ at interaction times when two waves of different families interact, then the fronts cross each other without changing their strengths and in such cases $\Delta L(t)=\Delta L_{in}(t)=0$; 

$\bullet$ at interaction times when two waves of the same family interact, then $\Delta L(t)=\Delta L_{in}(t)\le 0$. See also the forthcoming Remark~\ref{rem:L_in_xi=1}; 

$\bullet$ at the exit point $y=0$ or $y=M$, when a front reaches the boundary at $\tau$ and retains its strength $\eps$, but it transforms itself into a standby front, then $\Delta L(t)=0$, while $\Delta L_{in}(t)=-|\eps|$;

$\bullet$ last, at the time steps $t^n$, we use \eqref{eq:one-timestep} and {\it (b)} in Proposition~\ref{S2:prop:estimate-time-step} to get that
$$
|\eps_2^+| + |\eps_1^+| = |\eps_2^-|
$$
and hence that $\Delta L(t^n)=\Delta L_{in}(t^n)=0$. 

\smallskip
In view of the above, $L(t)$ and $L_{in}(t)$ are non-increasing, and it is clear by the definition that $L_{in}(t)\le L(t)$. 
At time $t=0+$, thanks to \eqref{eq:bound-on-size-of-Rp}, we can estimate $L(0+)$ and $L_{in}(0+)$ by
\begin{equation*}
L(0+)=L_{in}(0+) \le  \frac 12 \tv \{\ln(u_0)\} +  \frac 1{2\alpha} \tv \{v_0\} :=q
\end{equation*}
as in \eqref{eq:def-of-m}\,. This completes the proof of \eqref{LLin-decreases}. 

Finally we address \eqref{eq:unif-bound-tv-v}. By \eqref{LLin-decreases}, the size of any wave present at any $t>0$ is bounded by $q$. Then \eqref{eq:tv-v} is valid with $c=q$, that is
\begin{equation}\label{eq:tv-v_m}
\tv \{v(\cdot,t)\} \le 2\alpha \cosh(q)\, L_{in}(t)
\end{equation}
and therefore inequality \eqref{eq:unif-bound-tv-v} holds with $C_1 = 2\alpha q \cosh(q)$.
\end{proof}

\begin{remark}\label{rem:0}
By \eqref{def:L-Lout} and the monotonicity of $L(t)$ stated in Lemma~\ref{lem:bounds-on-bv}, we obtain 
\begin{equation*}
L_{in}(t)+L_{0,out}(t) + L_{M,out}(t) \le L_{in}(s)+L_{0,out}(s) + L_{M,out}(s)
\qquad \forall~t\ge s\,.
\end{equation*}
In particular, since $L_{in}(t)\ge 0$ and  $L_{0,out}$, $L_{M,out}$ are non-decreasing in time,
\begin{align*}
0\le\, & \underbrace{L_{0,out}(t) - L_{0,out}(s)}_{\ge 0} +  \underbrace{L_{M,out}(t) - L_{M,out}(s)}_{\ge 0}
      \le L_{in}(s) \qquad \forall~t\ge s\,.
\end{align*}
Proceeding as in Lemma~\ref{prop:equivalence-Lin}, we can obtain an estimate for the variation of\quad $t\to v(0+,t)$:
\begin{align*}
\tv\{ v(0+,\cdot); (s,t)\} &\le 2\alpha \cosh(q) \left(L_{0,out}(t) - L_{0,out}(s)\right)\\
&\le 2\alpha \cosh(q) L_{in}(s) \qquad \qquad \qquad \qquad \forall~t\ge s\,.
\end{align*}
In particular
\begin{equation}\label{eq:v-along-x=0}
\left| v(0+,t-) -  v(0+, s) \right|   \le  2\alpha \cosh(q) L_{in}(s)  \qquad \forall~t\ge s\,.
\end{equation}
\end{remark}

\begin{remark}\label{rem:1}
The a-priori uniform bound on the size of the waves will be a parameter in the subsequent analysis, based on Refs.~\cite{AC_SIMA_2008,ABCD_JEE_2015}, of the decay of the reflected waves.
\end{remark}

In the following lemma, we establish bounds on the approximate solutions, which are independent on $\nu$ and $t$. Note that the $\nu$ parameter appears as an index from here and on.

\begin{lemma}\label{Lemmauinfsup}
Having the initial data~\eqref{eq:init-data-lagr}, the approximate sequence $(u^\nu, v^\nu)$ to~\eqref{eq:system_Lagrangian} 
is defined for all times $t>0$, and the following holds.
\begin{itemize}
    \item[(i)] There exist $u^\nu_{inf}$ and $u^\nu_{sup}$ independent of 
$t$ such that
\begin{equation}\label{uinfsup}
0<u^\nu_{inf}\le u^\nu(y,t) \le u^\nu_{sup} \qquad \forall\, (y,t)\in(0,M)\times [0,+\infty)\,,
\end{equation}
and also, there exist $u_{inf}$ and $u_{sup}$ such that $0<u_{inf} \le u_{sup}$ and
\begin{equation}\label{eq:unif-bounds}
    \lim_{\nu\to\infty} u^\nu_{inf} = u_{inf}\,,\qquad
    \lim_{\nu\to\infty} u^\nu_{sup} = u_{sup}\,.
\end{equation}
\item[(ii)] There exists a constant $\widetilde C_0>0$ such that
\begin{equation}\label{eq:unif-bounds-v_nu}
    |v^\nu(y,t)|\le \widetilde C_0\qquad \forall\, y\,,\ t\,,\ \nu\,. 
\end{equation}
\end{itemize}
\end{lemma}
\begin{proof} The proof is divided into four steps.

\smallskip
\textbf{Step 1: A priori bounds on $u^\nu$.}
Suppose that the approximate solution $(u^\nu, v^\nu)$ is defined for $t\in[0,T]$. Let $\tilde{u}_0^\nu$, $\tilde{u}_M^\nu$ be the positive values defined at \eqref{def:u0-nu_uM-nu}. 

By construction one has 
\begin{align*}
\frac{1}{2} \left| \ln\left(\frac{u^\nu(y,t)}{\tilde{u}_0^\nu} \right)\right|
& \le \frac{1}{2} \left| \ln\left(\frac{u^\nu(0+,t)}{\tilde{u}_0^\nu} \right)\right| +  \frac{1}{2} \left|\sum_{0< y_j<y} \ln\left(\frac{u^\nu(y_j+,t)}{u^\nu(y_j-,t)}\right)\right| \\
& \le L_{0,out}(t) + L_{in}(t) \le L(t)
\,.
\end{align*}
Since $0\le L(t)\le q$ for all $t$, we get
$$
- q 
\le\frac{1}{2}\ln\left(\frac{u^\nu(y,t)}{\tilde{u}_0^\nu} \right)
\le q\qquad \forall\, 0<y<M
$$
and therefore
$$
\tilde{u}_0^\nu e^{-2q}\le u^\nu(y,t)\le\tilde{u}_0^\nu e^{2q},\qquad \forall \, t\in [0,T]
\,.
$$
Similarly, 
\begin{align*}
\frac{1}{2} \left|\ln\left(\frac{u^\nu(y,t)}{\tilde{u}_M^\nu} \right)\right|
& \le \frac{1}{2} \left| \ln\left(\frac{u^\nu(M-,t)}{\tilde{u}_M^\nu} \right)\right| +  \frac{1}{2} \left|\sum_{y< y_j<M} \ln\left(\frac{u^\nu(y_j+,t)}{u^\nu(y_j-,t)}\right)\right| \\
& \le L_{M,out}(t) + L_{in}(t) \le L(t)
\,,
\end{align*}
so we get
$$
\tilde{u}_M^\nu e^{-2q}\le u^\nu(y,t)\le\tilde{u}_M^\nu e^{2q},\qquad \forall \, t\in [0,T]
\,.
$$
In particular, we find that
\begin{equation}\label{eq:bound-on-u0-uM}
   e^{-2q} \le \frac{\tilde{u}_0^\nu}{\tilde{u}_M^\nu}\,,\quad  \frac{\tilde{u}_M^\nu}{\tilde{u}_0^\nu}\le    e^{2q}
\end{equation}
and
$$u^\nu(y,t) \in [u_{inf}^\nu,u_{sup}^\nu]\subset (0,+\infty)\;,$$
where we define 
$$
0< u_{inf}^\nu = e^{-2q} \max\{\tilde{u}_0^\nu,\tilde{u}_M^\nu\}\,,\qquad  
u_{sup}^\nu =  e^{2q} \min\{\tilde{u}_0^\nu,\tilde{u}_M^\nu\}\,. 
$$
This proves \eqref{uinfsup}. We observe that inequality $u_{inf}^\nu \le u_{sup}^\nu$ holds if
\begin{equation*}
    \frac{\max\{\tilde{u}_0^\nu,\tilde{u}_M^\nu\}}{\min\{\tilde{u}_0^\nu,\tilde{u}_M^\nu\}} \le  e^{4q}
\end{equation*}
which is valid because of \eqref{eq:bound-on-u0-uM}.

Next, recalling \eqref{eq:limits-init-data}, the limits in~\eqref{eq:unif-bounds} hold with
$$
u_{inf}:= e^{-2q} \max\{u_0(0+),{u}_0(M-)\} >0\,,\qquad   
u_{sup}:= e^{2q} \min\{u_0(0+),{u}_0(M-)\},$$
which are independent of $\nu$ and of time. 

\smallskip
\textbf{Step 2: Bound on the size of the rarefactions.}

By construction, at time $t=0+$, the size of each rarefaction is bounded by $\eta_\nu$ and we know that after the interaction with other waves, the size does not increase.

On the other hand, newly generated rarefactions at the time steps may result as a reflected wave of a shock. 
Then from estimate \eqref{time-step-reflected-wave}, we find that
$$
|\eps_1^+|\le C_1^-(q) {M}\DT |\eps_2^-| \,.
$$
Therefore, if we choose $\DT=\DT_\nu$, $\eta=\eta_\nu$ such that
\begin{equation}
     C_1^-(q) {M}\DT q  \le \eta\,,
\end{equation}
then every newly generated rarefaction has size $\le \eta$, since $|\eps_2^-|\le q$. As a consequence, the rarefaction would not be divided into two or more, according 
to the algorithm in Subsection~\ref{S2.3}. 

\smallskip
\textbf{Step 3: Finite number of interactions.}
Bounds in \eqref{uinfsup}, \eqref{eq:unif-bounds} show that the propagation speeds of the waves, see Remark~\ref{rem:2.1}, are uniformly bounded and that the ranges of values of each family are separated.
Thus, we can apply the analysis in Ref.~\cite{AmadoriGuerra01}, Sect. 3, and prove that, except for a finite number of interactions, there is at most one outgoing wave of each family for each interaction. 
It should be clarified that the presence of boundary does not interfere with the analysis at this stage of the proof, since no reflected waves appear. 
Therefore, we conclude that the approximate sequence $(u^\nu, v^\nu)$ is defined for all times $t>0$ and hence \eqref{uinfsup}, \eqref{eq:unif-bounds} are valid for all $t>0$.

\smallskip
\textbf{Step 4: Global bound on $v^\nu$.}
Finally, it remains to prove bound  \eqref{eq:unif-bounds-v_nu}. 

Let $\tilde{v}_0^\nu$ the value given at \eqref{def:v0-nu}. Similarly to the proof of \eqref{uinfsup}, we will take into account the standby waves to the boundary $y=0$ so that we can refer to the reference point $\tilde{v}_0^\nu$ which is independent of time.

Indeed, by recalling \eqref{eq:tv-v} and the proof of \eqref{eq:unif-bound-tv-v}, we obtain that
\begin{equation*}
|\tilde{v}_0^\nu - v^\nu(0+,t)| + \tv \{v^\nu(\cdot,t)\}\le 2\alpha \cosh(q)\, L(t) \le C_1 \,, \qquad t\ge 0
\end{equation*}
where $C_1$ is the constant at \eqref{eq:unif-bound-tv-v}. Then
\begin{equation*}
    |v^\nu(y,t)|\le |\tilde{v}_0^\nu| + C_1 \le \widetilde C_0\qquad  \forall\, y\,,\ t
\end{equation*}
for some $\widetilde C_0>0$ independent of $\nu$\,. This completes the proof of \eqref{eq:unif-bounds-v_nu}.
\end{proof}

\subsection{A weighted functional} \label{subsect:Lxi}
In this subsection, we present a functional that better captures the structure of the solution and in particular the amount of waves cancellation. 
This functional was introduced in Ref.~\cite{AmadoriGuerra01} and refined in Refs.~\cite{AC_SIMA_2008,ABCD_JEE_2015}. This is employed in the next subsection.

To begin with, we state a result (see Lemma 5.6 in Ref.~\cite{AC_SIMA_2008} and Lemma 5.4 in Ref.~\cite{ABCD_JEE_2015}) 
that provides a useful estimate on the size of the reflected wave produced at an interaction of waves. 

\begin{lemma}  
\label{lem:shock-riflesso}
Consider the interaction of two waves of the same family, of sizes $\alpha_i$ and $\beta_i$, $i=1$ or $2$, producing two outgoing waves $\eps_1$, $\eps_2$. 
Let $q>0$ and assume that $$|\alpha_i|\le q,\qquad |\beta_i|\le q\,.$$ Then the followings hold.
\begin{itemize}
\item[(a)] If the incoming waves are both shocks, then the resulting wave is a shock
and it satisfies $$\max \{|\alpha_i|,|\beta_i|\} < |\eps_i| < |\alpha_i| + |\beta_i|\;,$$ 
while the reflected wave is a rarefaction.

\medskip
\item[(b)] If the incoming waves have different signs, then 

- the reflected wave is a shock, 

- both the amount of shocks and the amount of rarefactions of the $i^{th}$ family decrease across
the interaction; 

- and, finally, for 
\begin{equation}\label{def:c(m)}
c(q) := \frac{\cosh(q)-1}{\cosh(q)+1}    
\end{equation}
one has
\begin{equation}\label{eq:chi_def}
|\eps_j| \le c(q) \cdot \min\{|\alpha_i|,|\beta_i|\}\,,\qquad j\ne i\,.
\end{equation}
\end{itemize}
\end{lemma}
\begin{remark}
The result of Lemma 5.4 in Ref.~\cite{ABCD_JEE_2015} provides an estimate which is slightly stronger than \eqref{eq:chi_def}, where $c(q)$ is replaced by $
c(\cdot)$ computed at the modulus of the size of the shock. For instance, if 
$\alpha_i < \beta_i$, then $c=c(|\alpha_i|)$. Because of the assumptions above, 
$c(|\alpha_i|)\le c(q)$ and hence \eqref{eq:chi_def} holds, which is sufficient for our purposes in this article.
\end{remark}

\medskip
Next, we define the weighted functional
\begin{equation}\label{def:Lxi}
L_\xi(t) = \sum_{j=1,\ \eps_j>0}^{N(t)} |\eps_j| + \xi \sum_{j=1,\ \eps_j<0}^{N(t)} |\eps_j|\qquad \xi\ge1\,.
\end{equation}
In other words, all shocks are multiplied by the weight $\xi$ which is greater than $ 1$. In the case of $\xi=1$, the functional coincides with $L_{in}(t)$.

It is immediate to verify that $L_\xi(t)$ is globally bounded, since 
   $L_{in}(t)\le L_\xi(t) \le \xi L_{in}(t)$ for all $t$
and hence, by Lemmas~\ref{lem:bounds-on-bv} and~\ref{Lemmauinfsup}, 
\begin{equation}\label{bound-on-Lxi}
    L_\xi(t) \le \xi L_{in}(0+)\qquad \forall \, t>0\,.
\end{equation}
In the following lemma, we recall a result from Ref.~\cite{ABCD_JEE_2015}, that here applies at every time $t$ which is not a time step. 

\begin{lemma}\label{lem:Delta-L-xi}
Let $q>0$ and assume that $L(0+)\le q$\,.
For $c(q)$ given in \eqref{def:c(m)}, if 
\begin{equation}\label{bounds-on-xi}
1\le \xi\le \frac 1 {c(q)}\,, 
\end{equation}
then $\Delta L_\xi(t)\le 0$ for any $t>0$, $t\not = t^n$ for $n\ge 1$. 
In particular, let $t$ be an interaction time between two waves of the same family, and denote by $\eps_{refl}$ the size of the reflected wave.
Then one has
\begin{equation}\label{eq:refined-decay-Lxi}
 \Delta L_\xi(t) + \left(\xi-1\right)|\eps_{refl}| \le 0\,.
\end{equation}
\end{lemma}

For convenience of the reader, we present the proof in  Appendix~\ref{subsec:app3}. 

\begin{remark}\label{rem:L_in_xi=1}
For $\xi=1$, the inequality \eqref{eq:refined-decay-Lxi} reduces to $\Delta L_{in}(t)\le 0$, \eqref{L-decreases}.

Also observe that, from \eqref{bounds-on-xi}, the range of values for $\xi$ shrinks as $q$ increases, being however defined for every $q>0$.
\end{remark}

Next, we examine the behavior of the  $L_\xi$ across the time steps, for $\xi>1$.

While in between time steps the functional $L_\xi$ decays, at the time steps it may well increase. 
Indeed, let's denote by $\eps^\pm$ the size of a wave before and after the time step, and by $\eps_{refl}$ the size of the new "reflected" wave. 
Recalling Proposition~\ref{S2:prop:estimate-time-step} and the sign rule, we have that 
$\eps^+\cdot \eps^->0$ and $\eps_{refl}\cdot \eps^-<0$.

Therefore
\begin{align*} \nonumber
\Delta L_\xi(t^n) &= \sum_{\eps^->0} \left[ \xi |\eps_{refl}| + |\eps^+| - |\eps^-| \right]  
+ \sum_{\eps^- <0} \left[  |\eps_{refl}| + \xi \left( |\eps^+| - |\eps^-| \right) \right] \\  \nonumber
&= \sum_{\eps^->0}  (\xi-1) |\eps_{refl}|  + \sum_{\eps^- <0} \underbrace{(1-\xi)}_{<0}   |\eps_{refl}| \\
&\le \sum_{\eps^->0}  (\xi-1) |\eps_{refl}| \\
&\le  \frac{M}2  \DT \, (\xi-1)\sum_{\eps^->0}|\eps^-|\;,
\end{align*}
where we used \eqref{time-step-reflected-wave}. Since the last sum is bounded by $L_{in}(t)$, we conclude that
\begin{align} 
\Delta L_\xi(t^n) \le \frac{M}2  \DT \, (\xi-1) L_{in}(t^n)\,.
\label{eq:Delta-L-xi-time-step}
\end{align}

\subsection{Estimates on vertical traces}\label{subsec:vert-traces}
In this subsection, we control the variation of the approximate solutions in time for fixed $y$. Given $y\in (0,M)$ and $t>0$, in a similar spirit as in the previous subsection, we consider the functional
\begin{equation*}
    W^\nu_y(t) = \frac 12 \tv\{ \ln (u^\nu)(y,\cdot); (0,t)\}\,. 
\end{equation*}
By the definition of the sizes at \eqref{eq:strengths}, and as in \eqref{prop:equivalence-Lin}, this quantity is the sum of the strengths
$|\eps|$ of the waves that cross $y$ within the time interval $(0,t)$. The definition of $W^\nu_y$ is inspired by Ref.~\cite{CG_2016}. 

In the next lemma, we provide a uniform bound for $W^\nu_y$ on bounded intervals of time, independently of $y\in [0,M]$. We set 
$$
u^\nu(0,t) = \lim_{y\to 0+} u^\nu(y,t)\,,\qquad u^\nu(M,t) = \lim_{y\to M-} u^\nu(y,t)\,.
$$ 

\begin{lemma} There exist $\widetilde C_1(q)$,  $\widetilde C_2(q)$ such that, for every $y\in[0,M]$, $\nu\in \N$ and $T>0$, 
the following holds:
\begin{align}\label{eq:bound-on-TV_at_y}
    W^\nu_y(T)\le  \widetilde C_1  L_{in}(0) + \widetilde C_2 M \int_0^{T} L_{in}(t)\,dt \,.   
\end{align}
\end{lemma}

\begin{proof} 
Let $y\in (0,M)$ be fixed. Let $A_y(t)$ denote the sum of the sizes of the waves that, at time $t$, are approaching the position $y$; 
that is, the waves of the $2^{nd}$ family, located to the right of $y$ and the waves of the $1^{st}$ family, located to the left of $y$.
Then consider
\begin{equation*}
    t\mapsto W^\nu_y(t) + A_y(t) +  \kappa 
     L_\xi(t)\,, \qquad \kappa=\frac1{\xi-1}\,.
\end{equation*}
We claim that the following holds,

\textit{(1)}\quad if $t\not = t^n$ then
\begin{equation}\label{Delta-W-one}
\Delta\left(W^\nu_y + A_y + \kappa L_\xi\right)(t) \le 0;    
\end{equation}

\textit{(2)}\quad at every time $t=t^n$, one has
\begin{equation}\label{Delta-W-two}
\Delta\left(W^\nu_y + A_y + \kappa L_\xi\right)(t^n) \le \DT\,L_{in}(t^n) \, M\,\widetilde C_2\,,\qquad 
\widetilde C_2 = \frac{\cosh(q) + 1}2  \,.
\end{equation}

{\bf Proof of \eqref{Delta-W-one}.} The sum $W^\nu_y(t) + A_y(t) + \kappa L_\xi(t)$ is piecewise constant in time and can change value only in the following cases.

\smallskip{}
{\bf a.}\quad If a wave front of size $\eps$ crosses the position $y$, then \eqref{Delta-W-one} holds because
$$
\Delta W^\nu_y(t)=|\eps|\,,\quad  \Delta A_y(t) = - |\eps|\,,\quad  \Delta L_\xi(t)=0;
$$

\smallskip{}
{\bf b.}\quad If a wave front of size $\eps$ reaches $y=0$ or $y=M$, then 
$$
\Delta W^\nu_y(t)=0= \Delta A_y(t)\,,\quad  \Delta L_\xi(t)\le-  |\eps|;
$$
and~\eqref{Delta-W-one} holds.

\smallskip{}
{\bf c.}\quad If at time $t$ an interaction between two wave fronts occur at a point $\not = y$, then 
$\Delta W^\nu_y(t)=0$. Moreover, if the interacting waves are approaching $y$, then 
$\Delta A_y(t)\le 0$ and $\Delta L_\xi(t)\le 0$. 

On the other hand, if the interacting waves are not approaching $y$, then the reflected wave is approaching $y$ and hence
$\Delta A_y(t)> 0$. More specifically, in the notation of Lemma~\ref{lem:shock-riflesso}, 
$$
\Delta A_y(t) = |\eps_j|\,, \qquad \Delta L_\xi(t)\le - (\xi-1)|\eps_j|;
$$
thanks to the choice of $\kappa$ we get immediately \eqref{Delta-W-one}.

Finally, if the interaction occurs at the point $y$, then
$$
\Delta W^\nu_y(t)= |\eps_i|\,,\quad  \Delta A_y(t)= - |\alpha_i| - |\beta_i|\,,\quad  \Delta L_\xi(t)\le  - (\xi-1)|\eps_j|.
$$
We easily find that
$$
\Delta W^\nu_y(t) + \Delta A_y(t)\le   |\eps_i|+|\eps_j| - |\alpha_i| - |\beta_i|= \Delta L_{in}(t)\le 0 
$$
and hence \eqref{Delta-W-one} holds in all cases.

\smallskip
{\bf Proof of \eqref{Delta-W-two}.} If $t=t^n$,  we claim that
\begin{equation}\label{eq:DeltaW-case_d}
    \Delta W^\nu_y(t) + \Delta A_y(t)\le \DT\, C_1^-(q)M L_{in}(t)\,.
\end{equation}
Indeed, if no wave reaches the position $y$ at time $t^n$, then 
$$
\Delta W^\nu_y(t)=0  \,,\quad  \Delta A_y(t)\le \DT\, C_1^-(q)M L_{in}(t)
$$
with $C_1^-(q) = \cosh(q)/2$, see Proposition~\ref{prop:estimate-time-step}.

On the other hand, if a wave of size $\eps$ reaches $y$ at time $t^n$, then
$$
\Delta W^\nu_y(t)=|\eps'|  \,,\quad  \Delta A_y(t)\le -|\eps| +  \DT\, C_1^-(q)M \, L_{in}(t),
$$
where $\eps'$ is the size of the transmitted wave. Recalling Proposition~\ref{S2:prop:estimate-time-step}, one has that 
$|\eps'|< |\eps|$ and hence \eqref{eq:DeltaW-case_d} holds also in this case. This concludes the proof of  \eqref{eq:DeltaW-case_d}.

Regarding $\Delta L_\xi(t^n)$, this is bounded as in \eqref{eq:Delta-L-xi-time-step}. Then, we combine \eqref{eq:DeltaW-case_d} 
and \eqref{eq:Delta-L-xi-time-step} to obtain
\begin{equation*}
\Delta\left(W^\nu_y + A_y + \kappa L_\xi\right)(t^n) \le \DT\,L_{in}(t^n) \, M \left(C_1^-(q) + \frac 12 \underbrace{\kappa (\xi -1)}_{=1}  \right)
\end{equation*}
and hence \eqref{Delta-W-two} holds with $\widetilde C_2 \, \dot = \,C_1^-(q) + \frac 12 = \frac{\cosh(q) + 1}2$\,.

\smallskip
Finally, by combining \eqref{Delta-W-one} and \eqref{Delta-W-two}, we find that
\begin{align*}
    W^\nu_y(T)&\le  W^\nu_y(T) + A_y(T) + \kappa L_{\xi}(T)\\
    &= \underbrace{W^\nu_y(0)}_{=0} + \underbrace{A_y(0)}_{\le L_{in}(0)} 
    + \kappa \underbrace{L_{\xi}(0)}_{\le \xi  L_{in}(0)} + \sum_{0<t<T}\Delta\left(W^\nu_y + A_y + \kappa L_{\xi}\right)(t)\\
    &\le \left(1 + \frac{\xi}{\xi-1}\right)L_{in}(0)  + \sum_{j=1}^n \Delta\left(W^\nu_y + A_y + \kappa L_{\xi}\right)(t^j)\\
    &\le  \left(2 + \frac{1}{\xi-1}\right)L_{in}(0) +  \widetilde C_2 M \int_0^{t^n} L_{in}(t)\,dt 
\end{align*}
with $n$: $t^n < T \le t^{n+1}$, where we used that $L_{in}(t)$ is non-increasing and then
$$
\DT\,L_{in}(t^n) \le \int_{t^{n-1}}^{t^n} L_{in}(t)\,dt\,.
$$
The inequalities above hold for every $\xi$ that satisfies exactly bound \eqref{bounds-on-xi}. Therefore, we can take 
$\xi=c(q)^{-1} = (\cosh(q) +1)/(\cosh(q) -1 )$ and get
$$
2 + \frac{1}{\xi-1} = \frac{3+\cosh(q)}{2}\, =:\, \widetilde C_1 \ge 1\,.
$$
Now we use the bound $\int_0^{t^n} L_{in} \le \int_0^T L_{in}$ and conclude that \eqref{eq:bound-on-TV_at_y} holds for $0<y<M$.

Finally, we consider the case of $y=0$ and $y=M$. Recalling the definition of $L_{0,out}(t)$, $L_{M,out}(t)$ and \eqref{def:L-Lout}, we find that
\begin{equation*}
    L_{0,out}(t)= \frac 12 \tv\{ \ln (u^\nu)(0,\cdot); (0,t)\}= W^\nu_0(t)\,,
\end{equation*}
and similarly $L_{M,out}(t)= W^\nu_M(t)$. Employing~\eqref{LLin-decreases}, we get
\begin{equation}\label{bound-tv-y0-yM}
   W^\nu_0(t) + W^\nu_M(t) \le L(t)\le L(0) = L_{in}(0)\qquad t\ge 0\,.
\end{equation}
Since $\widetilde C_1\ge 1$, then \eqref{eq:bound-on-TV_at_y} is valid also for $y=0$ and $y=M$.
\end{proof}

In the following lemma, we establish the uniform bounds on the total variation of $(u^\nu,v^\nu)$ with respect to time, 
as well as their stability of the $L^1$ norm on bounded intervals of time. As done for $u^\nu$, we also define 
$$v^\nu(0,t) = \lim_{y\to 0+}v^\nu(y,t) \qquad v^\nu(M,t)= \lim_{y\to M-}v^\nu(y,t).$$
\begin{lemma}\label{lem:vertical-bounds}
For every $T>0$, there exist positive constants $C$ and  $L$ independent of $\nu\in\N$ such that for all $y\in[0,M]$
\begin{align}\label{bound-bv-time}
    \tv\{u^\nu(y,\cdot);[0,T)]\} \le C\,,\quad
     \tv\{v^\nu(y,\cdot);[0,T)]\} \le C\,,
\end{align}
and 
\begin{align}\label{bound-integral-time}
   \int_0^T \left| v^\nu(y_1,t) -  v^\nu(y_2,t)\right|\, dt \le L \left| y_2 - y_1  \right|\qquad \forall\, y_1\,,\, y_2\in [0,M]\,.
\end{align}
An inequality analogous to \eqref{bound-integral-time} holds for $u^\nu$.
\end{lemma}

\begin{proof}
The first inequality in \eqref{bound-bv-time} follows by \eqref{eq:bound-on-TV_at_y} and from the uniform bounds on $u^\nu$, given at \eqref{uinfsup} and \eqref{eq:unif-bounds}.
Indeed, by using the mean value theorem for the function $\ln (u)$, we find that
\begin{align*}
 \tv\{u^\nu(y,\cdot);[0,T)]\} &\le 2 u_{sup}^\nu W^\nu_y(T)\\
 &\le  2 u_{sup}^\nu \, q \left( \widetilde C_1 + \widetilde C_2 M T \right)\,.
\end{align*}
where we used the bound $L_{in}(t)\le q$ for all $t\ge 0$. By means of \eqref{eq:unif-bounds}, then the first inequality in \eqref{bound-bv-time} holds for a suitable $C>0$ 
that is independent of $\nu$.

To prove the second inequality in \eqref{bound-bv-time}, we proceed as in \eqref{eq:tv-v_m} but, in addition, we have to take into account the variation at the time steps. 
Therefore let $n$ be the integer defined by $t^n < T \le t^{n+1}$, then
\begin{align*}
\tv \{v^\nu(y,\cdot);[0,T)\} &\le 2\alpha \cosh(q)\,  W^\nu_y(T) + \sum_{j=1}^n |\Delta v^\nu(y,t^j)|\\
&\le 2\alpha \cosh(q)\, q \left( \widetilde C_1  + \widetilde C_2 M T \right) +  M \|v^\nu\|_\infty n \DT\,.
\end{align*}
Thanks to the $L^\infty$ bound on $v^\nu$ in~\eqref{eq:unif-bounds-v_nu}, the last term here above is bounded by $ M \widetilde C_0 T$. Then, we conclude that there exists a constant $C=C(T)>0$
such that both inequalities in \eqref{bound-bv-time} hold true.

Let's now turn to the stability property \eqref{bound-integral-time}, concerning the map $y\mapsto \int_0^T v^\nu(y,t)dt$. By Remark~\ref{rem:2.1} and Lemma~\ref{Lemmauinfsup}, the propagation speed of a wave $y(t)$ satisfies 
$|y'(t)|\ge \alpha/u^\nu_{sup}$; therefore we can reparametrize it with $t=t(y)$, and its propagation speed is bounded by $u^\nu_{sup}/\alpha$.
The fractional step variation of $v^\nu$ does not affect the $y$-variation of the integral in $t$, since it induces discontinuities on $v^\nu$ across straight lines, which are constant in $y$.

Thanks to the uniform bound on the $t$-total variation, one has
\begin{equation*} 
\left\|v^\nu\left(y_1,\cdot \right) - v^\nu\left(y_2,\cdot \right)\right\|_{L^1(0,T)} \le  L \,|y_2 - y_1| \qquad \forall\, y_1\,, y_2 \in [0,M]
\end{equation*}
where $L$ is the product of a bound on $\sup_y \tv\{v^\nu(y,t); [0,T)\}$ and a bound on the propagation speed. 
We remark that $L$ may depend on $T$ and can be chosen independently of $\nu$. A similar procedure leads to an analogous estimate for $u^\nu$.
This concludes the proof of Lemma~\ref{lem:vertical-bounds}.
\end{proof}

\subsection{Convergence of approximate solutions}\label{S2.8}
In this subsection, we show a compactness property of the sequence of approximate solutions $(u^{\nu}, v^{\nu})(y,t)$, constructed in Subsection~\ref{S2.3},
as $\nu\to\infty$ and $\DT=\DT_\nu\to 0$. In particular, up to a subsequence, we find in the limit an entropy weak solution of the problem \eqref{eq:system_Lagrangian}-\eqref{eq:init-data-lagr}-\eqref{eq:bc-lagrangian}, in the sense specified here below.

\begin{theorem}\label{Th-1-lagr} Assume that $(u_0,v_0)\in BV(0,M)$ and satisfies \eqref{eq:init-data-lagr}:
$$
\essinf_{(0,M)} u_0 >0\,, \qquad  \int_0^M v_0 =0\,.
$$
Then there exists a subsequence of $(u^{\nu}, v^{\nu})$ 
which converges, as $\nu\to\infty$, to a function $(u,v)$ in $L^1_{loc}\left((0,M)\times[0,+\infty) \right)$, which
satisfies the following properties:

\begin{itemize}
\item [a)] the map $t\mapsto (u,v)(\cdot,t)\in L^1(0,M)$ is well-defined and it is Lipschitz continuous in the $L^1$--norm. Moreover it satisfies 
\begin{equation*}
(u,v)(\cdot,0)=(u_0,v_0)\,,     
\end{equation*}
and for all $t>0$
\begin{equation}\label{eq:Linfty-bound-uv}
0<u_{inf}\le u(y,t) \le u_{sup}\,,\quad 
    |v(y,t)|\le \widetilde C_0
\qquad \mbox{for a.e.}\ y
\end{equation}   
with $u_{inf}$, $u_{sup}$ given in \eqref{eq:unif-bounds} and $\widetilde C_0$ as in  \eqref{eq:unif-bounds-v_nu}\,. 
In particular, as $\nu\to\infty$ 
\begin{equation}\label{eq:conv-int_u_dy}
    \int_0^y u^\nu(y',t)\, dy'  ~\to~ \int_0^y u (y',t)\, dy'\qquad \forall\, y\in(0,M)\,,\ t\ge 0\,.
\end{equation}

\item [b)] For every $T>0$, the map $y\mapsto (u,v)(y,\cdot) \in L^1(0,T)$ is well defined on $[0,M]$ and Lipschitz continuous, with Lipschitz constant $\widetilde L$ possibly depending on $T$. In particular $v(0+,t)$ 
is well defined in $L^1_{loc}([0,+\infty))$ and, as $\nu\to\infty$, 
\begin{equation}\label{eq:conv-int_v-0_dt}
     \int_0^t v^\nu(0+,s) \, ds ~\to~ \int_0^t v(0+,s) \, ds\qquad \forall\, t>0\,.
\end{equation}
\end{itemize}
\end{theorem}

\begin{proof} 
To prove the convergence, we show that the sequence $(u^{\nu},v^{\nu})$ satisfies the assumptions of Helly's compactness theorem (see Th. 2.3 in Ref.~\cite{Bressan_Book}). 
More precisely, we employ a variant of that theorem, since we need to pass to the limit in the traces of $(u^{\nu},v^{\nu})$ both in $y$ and $t$. The proof of this variant follows the same strategy of Th. 2.3 in Ref.~\cite{Bressan_Book}, and relies on the pointwise convergence of $(u^{\nu},v^{\nu})(\cdot,t)$ for all rational $t$  and of $(u^{\nu},v^{\nu})(y,\cdot)$ for all rational $y$.

Then, we show that compactness assumptions hold. Thanks to Lemmas~\ref{lem:bounds-on-bv}--\ref{Lemmauinfsup}, the quantities $\tv\{u^{\nu}(\cdot,t)\}$, $\tv\{v^{\nu}(\cdot,t)\}$ are uniformly bounded in $t\ge 0$ and $\nu\in\N$, and the same holds for the $L^\infty$ norm of $(u^{\nu},v^{\nu})$ in $(y,t)$. 
The Lipschitz dependence on time of the map
$
t\mapsto\int_0^M (u^{\nu}, v^{\nu})(y,t)dy
$ can be proved similarly to the proof of \eqref{bound-integral-time}. 
More precisely, one can prove that
\begin{equation}\label{eq:int-y-lip-cont-t}
\left\|u^\nu\left(\cdot,t_2 \right) - u^\nu\left(\cdot,t_1 \right)\right\|_{L^1(0,M)}
\le \widetilde L \,|t_2 - t_1| \qquad 
\end{equation}
where $\widetilde L$ is the product of a bound on $\tv\{u^\nu(\cdot,t); (0,M)\}$ and a bound on the propagation speed. 
Thanks to Lemmas~\ref{prop:equivalence-Lin}--\ref{Lemmauinfsup}, the constant $\widetilde L$ can be chosen independently of $y$, $t$ and $\nu$.

Taking also into account of Lemma~\ref{lem:vertical-bounds}, by Helly's theorem there exists a subsequence of $(u^{\nu},v^{\nu})$ that converges to a limit $(u,v)(y,t)$ in $L^1_{loc}((0,M)\times [0,+\infty))$ and such that
\begin{align*}
&(u^{\nu},v^{\nu})(\cdot,t)\to (u,v)(\cdot,t)\quad\mbox{ in } L^1(0,M)
\,,\quad\forall\, t \\
&(u^{\nu},v^{\nu})(y,\cdot)\to (u,v)(y,\cdot)\quad\mbox{ in } L^1(0,T)
\,, \quad\forall\, y\in (0,M)\,,\ T>0.
\end{align*}
As a consequence, the convergence in \eqref{eq:conv-int_u_dy} holds for every $y\in(0,M)$. Similarly, the convergence in \eqref{eq:conv-int_v-0_dt} holds, thanks also to \eqref{bound-integral-time}. Moreover, bounds in \eqref{eq:Linfty-bound-uv} hold as well.
\end{proof}

{\bf Proof of Theorem~\ref{newthm}.}\quad
Let $(u,v)\in L^1_{loc}\left((0,M)\times[0,+\infty) \right)$ be the function obtained in Theorem~\ref{Th-1-lagr}. 
Then it satisfies the initial data as in \eqref{eq:init-data-uv}, as well as the property \eqref{eq:Linfty-bound-uv-new} 
because of \eqref{eq:Linfty-bound-uv}. Therefore it remains to show that $(u,v)$ is a distributional solution to 
\begin{equation}\label{eq:system_Lagrangian-2}
\begin{cases}
\partial_\t u -  \partial_y  v  = 0, &\\
\partial_\t v  +  \partial_y (\alpha^2/u) =   - M v & 
\\ 
\end{cases} 
\end{equation}
on $(0,M)\times(0,+\infty)$, and that for every convex entropy $\tilde\eta$ for the system \eqref{eq:system_Lagrangian-2}, 
with corresponding entropy flux $\tilde q$, the following inequality holds in ${\DD}'((0,M)\times(0,+\infty))$:
\begin{equation*}
    \partial_t \tilde\eta(u, v) + \partial_x \tilde q(u, v)\le -M \tilde\eta_v v\,.
\end{equation*}%
Since the test functions are supported in the open set $(0,M)\times(0,+\infty)$ and hence the boundaries are not touched, 
then the proof is the same as for the Cauchy problem in Ref.~\cite{AmadoriGuerra01} and we omit it.

\Section{Proof of Theorem~\ref{Th-1}} \label{Sect:3}
This section is devoted to the proof of Theorem~\ref{Th-1}. The structure is the following: Approximate solutions to system~\eqref{eq:system_Eulerian} in Eulerian variables and the approximate boundaries $a^\nu$, $b^\nu$ are defined in Subsection~\ref{Sect:3.1}. Convergence of the free boundaries and uniform bounds are obtained in Subsections~\ref{Sect:3.2} and~\ref{Sect:3.3} respectively. Next, the admissibility conditions are investigated and the total mass and momentum are estimated in Subsections~\ref{Sect:3.4} and~\ref{Sect:3.5} at the level of the approximate sequence. In Subsection~\ref{Sect:3.6}, the proof of Theorem~\ref{Th-1} is completed by combining the previous analysis.

\subsection{Change of variables, from Lagrangian to Eulerian} 
\label{Sect:3.1}
For each $\nu\in\N$, define the approximate boundaries:
\begin{equation}\label{def:a-nu}
    a^\nu(t)\, \dot = \,a_0 + \int_0^t v^\nu(0+,s) ds\,.
\end{equation}
and
\begin{equation}\label{def:b-nu}
    b^\nu(t)~ \dot = ~ 
    a^\nu(t) + \int_0^M u^\nu\left(y,t \right) dy \,.
\end{equation}
These functions have the role of \textit{free boundaries} that delimit the region where the density $\rho^\nu$ is positive, in the Eulerian variables. See Figure~\ref{fig:setting}.

\begin{figure}[htbp]
{\centering \hfill

\scalebox{0.8}{
\includegraphics{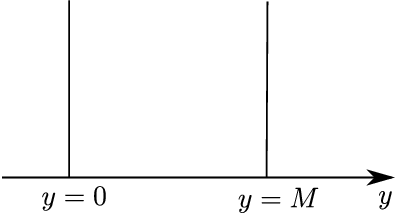}
\quad \includegraphics{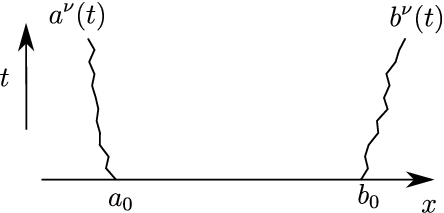}}\par}
\caption{On the left, the domain in Lagrangian 
variables for $(u^\nu,v^\nu)$ and on the
right: the domain in Eulerian variables for $(\rho^\nu,\vv^\nu)$.}
 \label{fig:setting}
\end{figure}

\begin{lemma}\label{S3:lemma1} Let
\begin{equation}\label{eq:def-of-inverse-chi}
    x^\nu(y,t) := a^\nu(t) + \int_0^y u^\nu\left(y',t \right) dy'
    \,,\qquad y\in (0,M),\ t\ge 0
\end{equation}
and
\begin{align}
I^\nu(t) &=~  \left(a^\nu(t), b^\nu(t)\right)\,, \nonumber\\
\Omega^\nu   &=~ \{ (x,t);\ t\ge 0\,,\  x\in I^\nu(t) \}\subset \R\times[0,+\infty)\,. \nonumber
\end{align}
Then the map
\begin{align*}
    (y,t)&\mapsto \left(x^\nu(y,t),t\right) \\
  (0,M)\times [0,\infty)&\mapsto \Omega^\nu  
\end{align*}
is invertible, and
\begin{align}
\label{eq:bi-Lipsch_x}
   u^\nu_{inf} |y_2-y_1| \le |x^\nu(y_2,t)- x^\nu(y_1,t)|&\le u^\nu_{sup} |y_2-y_1|\,, \qquad y_1,\ y_2\in (0,M)\,,\\[1mm]
\label{eq:xnu-Lipsch_t}
    |x^\nu(y,t_2)- x^\nu(y,t_1)| &\le L \, |t_2-t_1|\qquad \qquad  
    t_1,\ t_2\ge 0
\end{align}
for some constant $L>0$ independent of $y\in(0,M)$, $\nu$, $t$. Moreover, $a^\nu(\cdot)$, $b^\nu(\cdot)$ are Lipschitz continuous on $[0,\infty)$, uniformly with respect to $\nu$. 
\end{lemma}

\begin{proof}
From the bounds on $u^\nu$ in \eqref{uinfsup}, we deduce that the map
\begin{align*}
 y&\mapsto x^\nu(y,t)\\
(0,M) &\mapsto I^\nu(t)
\end{align*}
is Lipschitz continuous, strictly increasing, bijective  and it satisfies \eqref{eq:bi-Lipsch_x}.

Concerning \eqref{eq:xnu-Lipsch_t}, the definition \eqref{eq:def-of-inverse-chi} yields
\begin{align*}
     |x^\nu(y,t_2)- x^\nu(y,t_1)| &\le |a^\nu(t_2)- a^\nu(t_1)| +  \int_0^y \left| u^\nu\left(y',t_2 \right) - u^\nu\left(y',t_1 \right)\right| dy'\\
     &\le \left( \widetilde C_0 + \widetilde L \right)|t_2 - t_1| \;,
\end{align*}
where we used the bound on $\|v^\nu\|_\infty$ in \eqref{eq:unif-bounds-v_nu},
and \eqref{eq:int-y-lip-cont-t}. Then \eqref{eq:xnu-Lipsch_t} holds for 
$L = \widetilde C_0 + \widetilde L$\,.  In particular, for 
$y=0$ we find that $a^\nu(t)$ is Lipschitz continuous with Lipschitz constant $\widetilde C_0$, and for $y\to M$, we deduce that also $b^\nu(t)$ is Lipschitz continuous on $[0,+\infty)$ with constant $\widetilde C_0 + \widetilde L$.
\end{proof}

For every $t\ge 0$, define the map
\begin{align*}
(x,t)&\mapsto y=\chi^\nu(x,t)\\
I^\nu(t) \times[0,\infty)&\mapsto (0,M)  
\end{align*}
as the inverse of $y\mapsto x^\nu(y,t)$\,.
Then define $\rho^\nu$, $\vv^\nu$, $\mm^\nu$ as follows,
\begin{align}
\label{def-rho-v-m-in}
x\in I^\nu(t):& \quad 
\begin{cases}
\rho^\nu(x,t) = \{u^\nu (\chi^\nu(x,t),t)\}^{-1}
\,,&\\[1mm]
\vv^\nu(x,t)=v^\nu (\chi^\nu(x,t),t)\,,&\\[1mm]
\mm^\nu(x,t)  = \rho^\nu(x,t) \vv^\nu(x,t)\,,
\end{cases}
    \\[2mm]
   \label{def-rho-m-out}
     x\not\in I^\nu(t):& \qquad  \rho^\nu(x,t) =0  ~= \mm^\nu(x,t)\,. 
\end{align}
The specific value of $\vv^\nu$ outside $I^\nu(t)$ is not relevant.

About the measures $dy$ on $(0,M)$ and $dx$ on $I^\nu(t)$ for $t\ge0$, they are mutually absolutely continuous and we have
\begin{equation*}
    dx = \frac{\partial x^\nu}{\partial y} dy = u^\nu(y,t)dy\,,\qquad dy = \frac{1}{u^\nu(\chi^\nu(x,t),t)} dx  =  \rho^\nu(x,t) dx\,.
\end{equation*}
Therefore, we have the identities for all $\nu$ and $t\ge 0$,
\begin{align}\label{eq:constant-mass-nu}
    &\int_{I^\nu(t)} \rho^\nu(x,t) \, dx = \int_0^M dy = M\\
    &\int_{I^\nu(t)} dx = \int_0^M  u^\nu(y,t)\, dy = b^\nu(t) - a^\nu(t)\,. 
    \label{eq:support-length}
\end{align}

\subsection{Convergence of the free boundaries}\label{Sect:3.2}

Next, we consider the limiting behavior of $a^\nu$, $b^\nu$ as $\nu\to\infty$. We define
\begin{equation}\label{def:a-b}
    a(t)\, \dot = \,a_0 + \int_0^t v(0+,s) ds\,,\qquad  b(t)~ \dot = ~ 
    a(t) + \int_0^M u\left(y',t \right) dy' \,,
\end{equation}
with $(u,v)$ being the solution to system~\eqref{eq:system_Lagrangian-2} obtained in Theorem~\ref{Th-1-lagr}. 
The convergence result is stated in the following lemma.

\begin{lemma}\label{lem:conv-a-b}
As $\nu\to\infty$,
\begin{equation*}
    a^\nu(\cdot)\to a(\cdot)\,,\qquad b^\nu(\cdot)\to b(\cdot)
\end{equation*}
uniformly on compact subsets of $[0,+\infty)$.
\end{lemma}
\begin{proof} By \eqref{eq:conv-int_u_dy} and \eqref{eq:conv-int_v-0_dt}, for every $t>0$ and $\nu\to\infty$, one has
\begin{equation*}
    \int_0^t v^\nu(0+,s) ds\to \int_0^t v(0+,s) ds\,,\qquad 
    \int_0^M u^\nu\left(y',t \right) dy'\to \int_0^M u\left(y',t \right) dy'\,.
\end{equation*}
Therefore,  $a^\nu(t)\to a(t)$ and $b^\nu(t)\to b(t)$ pointwise. Since $a^\nu(\cdot)$ and $b^\nu(\cdot)$ are Lipschitz continuous on $[0,+\infty)$ uniformly in $\nu$ and $t$, then the pointwise convergence implies the local uniform convergence on $[0,\infty)$.
\end{proof}

\subsection{Uniform bounds on approximate solutions}
\label{Sect:3.3}

Here, we deduce some uniform bounds on the $L^\infty$ norm and the total variation of $\rho^\nu$, $\vv^\nu$, 
that follow from the corresponding bounds on $u^\nu$, $v^\nu$.

From \eqref{uinfsup}, we find immediately that
\begin{equation*}
    0<(u^\nu_{sup})^{-1}\le \rho^\nu(x,t) \le (u^\nu_{inf})^{-1} \qquad \forall\, x\in I^\nu(t) \,, \quad t\ge 0
\end{equation*}
while, from \eqref{eq:identity-tvlnu} and \eqref{LLin-decreases}, we deduce that
\begin{equation*}
    \frac 12 \tv \{\ln(\rho^\nu)(\cdot,t);I^\nu(t)\} = L_{in}(t) \le L_{in}(0+)\le q\,.
\end{equation*}
Therefore, we conclude that 
\begin{align*}
     \tv \{\rho^\nu(\cdot,t);\R \} &\le
      \tv \{\rho^\nu(\cdot,t); I^\nu(t)\} + 2 \|\rho^\nu(\cdot,t)\|_\infty\\
     &\le u^\nu_{sup} \tv \{\ln(\rho^\nu)(\cdot,t);I^\nu(t)\} + \frac 2  {u^\nu_{inf}} 
     \end{align*}
and hence it is uniformly bounded in terms of $\nu$ and $t$\,.

About $\vv^\nu$, from its definition \eqref{def-rho-v-m-in} it is clear that the $L^\infty$ bounds are the same valid for $v^\nu(\cdot,t)$. 
Moreover,  thanks to \eqref{eq:unif-bound-tv-v}, we obtain
\begin{align*}
\tv \{\vv^\nu(\cdot,t);I^\nu(t)\} = \tv \{v^\nu(\cdot,t);(0,M)\} \le C_1 \,.
\end{align*}%
As well, the total variation in space of the \textit{momentum} $\mm^\nu$ given in \eqref{def-rho-v-m-in} and \eqref{def-rho-m-out},
is uniformly bounded in terms of $\nu$ and $t$  thanks to the bounds above and the $L^\infty$ bounds on $\rho^\nu$ and $\vv^\nu$\,.

\subsection{Admissibility conditions and weak formulation} 
\label{Sect:3.4}
Let $0<y_1^\nu(t)<\ldots<y_{N^\nu(t)}^\nu(t)<M$ be the location of discontinuities for $(u^\nu,v^\nu)(\cdot,t)$, as in 
\eqref{eq:disc-points}. Define
\begin{equation*}
    x^\nu_j(t) = a^\nu(t) + \int_0^{y_j^\nu(t)} u^\nu\left(y',t \right) dy'\qquad j=1,\ldots,N^\nu(t)
\end{equation*}
and 
\begin{equation*}
x^\nu_0(t) = a^\nu(t)\,,\qquad x^\nu_{N^\nu(t)+1}(t) = b^\nu(t)\,, 
\end{equation*}
which are the discontinuity locations of $(\rho^\nu,\mm^\nu)$ over $\R$. 

In the next lemma, we show that the Rankine-Hugoniot conditions are approximately satisfied across the piecewise linear curves $x^\nu_j(t)$.
\begin{lemma}\label{lem:RH-cond-rhov}
There exists a constant $C>0$ independent of $j=0,\ldots,N^\nu(t)+1$, $t$ and $\nu$ such that
\begin{equation}\label{RH-rho-tilde v}
    \frac{d}{dt} x^\nu_j(t) = \frac{\Delta \mm^\nu}
    {\Delta \rho^\nu}\left(x^\nu_j(t)\right) + \O(1)\eta_\nu\,, \quad {j=0,\dots,  N^\nu(t)+1}
\end{equation}
\begin{equation}\label{RH-rho-tilde m}
    \frac{d}{dt} x^\nu_j(t) = \frac{\Delta \left( \frac{(\mm^\nu)^2}{\rho^\nu}+p(\rho^\nu)\right)
    }{\Delta \mm^\nu}\left(x^\nu_j(t)\right) + \O(1)\eta_\nu\,, \quad {j=1,\dots,  N^\nu(t)}
\end{equation}
and $|\O(1)|\le C\,$.
\end{lemma}
\begin{proof} First we prove \eqref{RH-rho-tilde v} for all $j=0,\ldots,N^\nu(t)+1$ by dividing the proof into the following three cases.

$\bullet$ \fbox{$j=0$}\quad First we consider the case $j=0$, that is, $x^\nu_0(t) = a^\nu(t)$. Recalling \eqref{def:a-nu}, one finds that
\begin{align}
\frac{d}{dt} x^\nu_0(t) & = v^\nu(0+,s) = \vv^\nu (x^\nu_0(t)+,t)\nonumber\\
& =  \frac{\mm^\nu(x^\nu_0+)  - \mm^\nu (x^\nu_0-)}{\rho^\nu (x^\nu_0+ )- \rho^\nu (x^\nu_0 - ) }\;, 
\label{eq:approx-speed-a_nu-b_nu}
\end{align}
where we used \eqref{def-rho-v-m-in}, \eqref{def-rho-m-out}. Hence \eqref{RH-rho-tilde v} holds for $j=0$, with no error term.

\smallskip
$\bullet$ \fbox{$j=1,\ldots,N^\nu(t)$}\quad
Using \eqref{eq:def-of-inverse-chi} and~\eqref{eq:approx-speed-a_nu-b_nu}, the propagation speed of $x^\nu_j$ satisfies
\begin{align*}
\frac{d}{dt} x^\nu_j(t) = v^\nu(0+,t) + 
 \frac{d}{dt} \int_{0}^{y_j^\nu(t)} u^\nu(y,t)dy\,. 
\end{align*}
Let's focus on the last term. For the discontinuity at $y^\nu_\ell(t)$, we define $\lambda^\nu_\ell$ to be the exact speed given by the Rankine-Hugoniot condition for system~\eqref{eq:system_Lagrangian-2}:
\begin{equation}\label{eq:RH-v}
- \Delta v^\nu = \lambda^\nu_\ell \, \Delta u^\nu
\,,\quad 
\Delta(\alpha^2/u^\nu) = 
\lambda^\nu_\ell \, \Delta v^\nu
\qquad \mbox { at } y= y^\nu_\ell\,.    
\end{equation}
By construction of the scheme, one has that
\begin{equation*}
    |(y^\nu_j)' - \lambda^\nu_j| \le \eta_\nu\qquad \forall j\,.
\end{equation*}
Therefore, we have
\begin{align*}
 \frac{d}{dt} \int_{0}^{y_j^\nu(t)} u^\nu(y,t)dy &=  u^\nu(y^\nu_j-) \, (y_j^\nu)' - \sum_{\ell=1}^{j-1}  \Delta u^\nu(y^\nu_\ell) (y^\nu_\ell)' \\
 &= (I) + (II)\;,
\end{align*}
where
\begin{align*}
    (I)&=   u^\nu(y^\nu_j-) \, \lambda^\nu_j - \sum_{\ell=1}^{j-1} \Delta u^\nu(y^\nu_\ell)\lambda^\nu_\ell 
    \,,\\
    (II)&=   u^\nu(y^\nu_j-)\{(y_j^\nu)' - \lambda^\nu_j \} - \sum_{\ell=1}^{j-1}\Delta u^\nu(y^\nu_\ell) \{(y^\nu_\ell)'- \lambda^\nu_\ell  \}
\,. \end{align*}

About $(I)$, with the help of \eqref{eq:RH-v}, we find that
\begin{align*}
    (I)&=   \sum_{\ell=1}^{j-1} \left[ \Delta v^\nu(y^\nu_\ell)
    + u^\nu(y^\nu_j-) \, \lambda^\nu_j\right]
    \\
&=   - v^\nu(y^\nu_{1}-) + v^\nu(y^\nu_{j-1}+) + u^\nu(y^\nu_j-) \, \lambda^\nu_j
\\
&= - v^\nu(0+)  + v^\nu(y^\nu_{j}-)  \underbrace{- u^\nu(y^\nu_j-) \, 
\frac{\Delta v^\nu}{\Delta u^\nu}(y^\nu_j)}_{=(*)}
\,.
\end{align*}
Now, one has that
\begin{align*}
(*)&=  - \frac 1 {\rho^\nu(x^\nu_j-)} \, 
\frac{\Delta \vv^\nu(x^\nu_j)}{\Delta \frac1{\rho^\nu (x^\nu_j)}}\\
&=  \rho^\nu(x^\nu_j+)\, \frac{\Delta \vv^\nu(x^\nu_j)}{\Delta \rho^\nu (x^\nu_j)}\,.
\end{align*}
Therefore, 
\begin{align*}
    (I)&= - v^\nu(0+)+ \frac{v^\nu(y^\nu_{j}-)  
    \Delta \rho^\nu (x^\nu_j) + \rho^\nu(x^\nu_j+)\Delta \vv^\nu(x^\nu_j)
    }{\Delta \rho^\nu (x^\nu_j)}\;,
\end{align*}    
that leads to
\begin{align*}
\frac{d}{dt} x^\nu_j(t) &=  \frac{\vv^\nu(x^\nu_{j}+) \rho^\nu (x^\nu_j+) - \vv^\nu(x^\nu_j-)\rho^\nu(x^\nu_j-)
    }{\Delta \rho^\nu (x^\nu_j)}
+ (II)\,.
\end{align*}
Thanks to the bounds on $u^\nu$ in \eqref{uinfsup} and on the relation $L_{in}(t) = \frac 12 \tv \{\ln(u^\nu)(\cdot,t)\}$, we deduce that there exists a constant $C>0$ independent of $j$, $t$ and $\nu$ 
such that
\begin{equation}\label{eq:bound-on-speed-error}
\left| (II)\right| \le C \eta_\nu \,,
\end{equation}
and this yields~\eqref{RH-rho-tilde v} for $0<j<N^\nu+1$.

\smallskip
$\bullet$ \fbox{$j=N^\nu+1$}\quad
We proceed as above,
and deduce that 
\begin{align*}
\frac{d}{dt} b^\nu(t) &= v^\nu(M-,t) + (II)\;,
\end{align*}
where $(II)$ is the term that takes into account of the error in the speed, and satisfies the same bound as \eqref{eq:bound-on-speed-error}.
As for $j=0$, we use \eqref{def-rho-v-m-in}, \eqref{def-rho-m-out} to find that
\begin{equation*}
    v^\nu(M-,t) = \vv^\nu(b^\nu(t)-,t)
    = \frac{\mm^\nu (b^\nu(t)+)  - \mm^\nu(b^\nu(t)-) }{\rho^\nu (b^\nu(t)+ )- \rho^\nu (b^\nu(t)-)}\,.
\end{equation*}
Hence \eqref{RH-rho-tilde v} holds for $j=N^\nu(t)+1$.

Let's turn our attention to \eqref{RH-rho-tilde m}. Thanks to \eqref{eq:RH-v} and by direct calculations, one finds that
\begin{equation*}
     \frac{\Delta \left({v^\nu/u^\nu}\right)}{\Delta \left(1/u^\nu\right)} = \frac{\Delta \left({[(v^\nu)^2 + \alpha^2]/u^\nu}\right)}{\Delta \left(v^\nu/u^\nu\right)}\,,
\end{equation*}
and then
\begin{equation*}
    \frac{\Delta \left({\rho^\nu \vv^\nu}\right)}{\Delta \rho^\nu} = \frac{\Delta \left({\rho^\nu(\vv^\nu)^2 + \alpha^2 \rho^\nu}\right)}
    {\Delta \left(\rho^\nu \vv^\nu\right)}\,.
\end{equation*}
Hence \eqref{RH-rho-tilde m} follows immediately from \eqref{RH-rho-tilde v}.
The proof is complete.
\end{proof}

\smallskip

\paragraph{\textbf{\underline{Shocks and rarefactions}}}

It is well known that the ra\-re\-faction-shock curves for system~\eqref{eq:system_Lagrangian-2}, 
issued at a point $(u_\ell,v_\ell)$ with $u_\ell>0$, can be translated into the ra\-re\-faction-shock curves for 
system~\eqref{eq:system_Eulerian} issued at $(\rho_\ell=(u_\ell)^{-1},\vv_\ell=v_\ell)$. In more detail, 

\smallskip\noindent
$\bullet$\quad curve of the \textbf{first} family:
\begin{equation} 
 \begin{cases}  
   \rho>\rho_\ell, \qquad \qquad  \vv = \vv_\ell - \alpha \frac{|\rho- \rho_\ell|}{\sqrt{\rho\rho_\ell}}, \\  
   0<\rho<\rho_\ell, \qquad \vv = 
    \vv_\ell + \alpha \ln \left(\frac {\rho_\ell} {\rho}\right);
 \end{cases}
    \label{RH-1-euler}  
\end{equation} 

\medskip\noindent
$\bullet$\quad curve of the \textbf{second} family:
\begin{equation}   
 \begin{cases}
   \rho>\rho_\ell, \qquad  \qquad 
   \vv = \vv_\ell - \alpha\ln\left(\frac{\rho_\ell}{\rho}\right), 
   \\  
   0<\rho<\rho_\ell, \qquad \vv = \vv_\ell - \alpha \frac{|\rho- \rho_\ell|}{\sqrt{\rho\rho_\ell}}\,.
    \label{RH-2-euler}  
    \end{cases}
\end{equation} 
It is clear that, when passing from the  $(u^\nu,v^\nu)$ to the $(\rho^\nu,\vv^\nu)$ approximate solutions, shocks translate into shocks and rarefactions into rarefactions. 

The discontinuities at $x=a^\nu(t)$ and at $x=b^\nu(t)$ separate the half-plane $\mathbb{R}\times[0,\infty)$ into the region 
\begin{equation}\label{def:Omega-nu}
\Omega^\nu\dot =\{(x,t):\ t\ge 0\,,\ x\in I^\nu(t) \}\;,    
\end{equation}
where $\rho$ is uniformly positive and  the external region $\mathbb{R}\times[0,\infty)\setminus\Omega^\nu $ with $\rho=0$. We claim that 
$a^\nu$ and $b^\nu$ can be interpreted as a 1-shock and a 2-shock, respectively. 

Indeed, let $(\rho_\ell,\vv_\ell)$ with $\rho_\ell>0$ and consider $(\rho, \vv(\rho))$ from \eqref{RH-2-euler}$_2$. The discontinuity with
$(\rho_\ell,\vv_\ell)$  on the left, and $(\rho, \vv(\rho))$ on the right, is a 2-shock. By letting $\rho\to 0$, we find that
$$
\rho \vv(\rho) = \rho \vv_\ell - \alpha \sqrt{\rho}\frac{|\rho- \rho_\ell|}{\sqrt{\rho_\ell}} \to 0\qquad  \mbox{as }\rho\to 0\,.
$$
Hence, for any value $(\rho, \rho\vv)$ with $\rho>0$ at $x=b^\nu(t)-$, the jump to $(0,0)$ can be interpreted as a 2-shock 
with propagating speed $\vv_\ell$.
 
At $x=a^\nu(t)$ the argument is analogous, however we have to consider the inverse 1-shock curve: for a fixed state $(\rho,\rho\vv)$ 
on the right of the discontinuity, the value $(\rho_\ell,\vv_\ell)$ on the left is obtained from \eqref{RH-1-euler}$_1$:
$$
 \vv_\ell = \vv + \alpha \frac{|\rho- \rho_\ell|}{\sqrt{\rho\rho_\ell}}\,,\qquad 0<\rho_\ell<\rho\,.
$$
As $\rho_\ell\to 0$ and for $(\rho,\rho\vv)$ fixed,
$$
 \rho_\ell \vv_\ell = \rho_\ell \vv + \alpha \sqrt{\rho_\ell}\frac{|\rho- \rho_\ell|}{\sqrt{\rho}}\to 0\,.
$$
Therefore, at $x=a^\nu(t)$ the discontinuity with left state $(0,0)$ and any value $(\rho, \rho\vv)$ with $\rho>0$ as a right state, 
can be interpreted as a 1-shock with propagating speed $\vv$.

\begin{remark}
In Ref.~\cite{LS80}, the problem of the vacuum state for isentropic gas dynamics has been considered, under assumptions on $p(\rho)$ 
that include powers $\rho^\gamma$ with $\gamma>1$; in particular, it is assumed that $\sqrt{p'(\rho)}/\rho$ is integrable at $\rho=0$.
In our case this property does not hold; as a consequence, $\vv(\rho)$ diverges as $\rho\to 0$ in
\eqref{RH-1-euler}, \eqref{RH-2-euler}, while $\rho \vv(\rho)$ converges to 0. 
\end{remark}

\smallskip
\paragraph{\textbf{\underline{Weak formulation}}} We conclude this subsection by providing 
the weak formulation of the equations satisfied by our approximate solution $(\rho^\nu,\mm^\nu)$. For any test function $\phi\in C^\infty_0\left(\R\times(0,\infty)\right)$ we define
\begin{align}
    & \RR^\nu := 
    \iint_{\R\times\R_+} \left\{\rho^\nu\phi_t + \mm^\nu\phi_x \right\}\; dxdt\,,\label{sec4.3:Rnu}\\
    &
    \MM^\nu := 
    \iint_{\Omega^\nu}\left\{ \mm^\nu\phi_t
    + \left[p(\rho^\nu) + \frac{(\mm^\nu)^2}{\rho^\nu} \right]\phi_x
    -M\mm^\nu \phi \right\}\,dx dt \nonumber\\
       & \qquad \qquad  - \int_0^\infty   \left[p(\rho_b^\nu(t)) \phi(b^\nu(t),t) -p(\rho_a^\nu(t)) \phi(a^\nu(t),t) \right] \,dt\;,\label{sec4.3:Mnu}
\end{align}
where
\begin{equation}\label{def:rhonu-b-a}
\rho_b^\nu(t)=\rho^\nu(b^\nu(t)-,t),\qquad \rho_a^\nu(t)=\rho^\nu(a^\nu(t)+,t)\;.
\end{equation}
We claim that 
\begin{align}\label{S4.3limits}
    \lim_{\nu\to\infty} \RR^\nu= 0 = \lim_{\nu\to\infty} \MM^\nu\,.
\end{align}
The proof of the claim follows by standard techniques used in the front-tracking method and for this reason, we show only the analysis for the limit of $\MM^\nu$ that is interesting because it involves the second integral with the pressure terms along the boundaries, which is not usually present. As the following computations indicate, this arises in $\MM^\nu$ due to the vacuum. So, following the notation in Lemma~\ref{Sect:3.4}, we write 
$$\iint_{\Omega^\nu}\left\{ \mm^\nu\phi_t
    + \left[\frac{(\mm^\nu)^2}{\rho^\nu}+p(\rho^\nu)  \right]\phi_x
    -M\mm^\nu \phi \right\}\,dx dt=I_1^\nu+I_2^\nu\;,
    $$
where
\begin{align*}I_1^\nu:=&\int_0^T\sum_{j=1}^{N^\nu(t)} \left[(x_j^\nu)'\Delta\mm^\nu-\Delta\left( \frac{(\mm^\nu)^2}{\rho^\nu}+p(\rho^\nu)\right)
\right]\phi(x_j(t),t)\,dt\\
&+ \int_0^T\left[  (b^\nu(t))' (-\mm^\nu(b^\nu(t)-,t))+\frac{(\mm^\nu)^2}{\rho^\nu}(b^\nu(t)-,t)+p(\rho^\nu_b(t)) \right]\phi(b^\nu(t),t)\,dt
\\
&+\int_0^T\left[ (a^\nu(t))' (\mm^\nu(a^\nu(t)+,t))- \frac{(\mm^\nu)^2}{\rho^\nu}(a^\nu(t)+,t)-p(\rho^\nu_a(t)) \right]\phi(a^\nu(t),t)\,dt\;,
\end{align*}
and
\begin{align*}I_2^\nu:=&\sum_n\int_\R \left[\mm^\nu(x,t^n-)-\mm^\nu(x,t^n+)\right]\phi(x,t^n) \,dx- \iint_{\Omega^\nu} M\mm^\nu(x,t) \phi (x,t)\,dxdt
\\
\stackrel{\eqref{eq:u-v_fractional-step}}{=}& M\sum_{n=0}^\infty \int_{t^n}^{t^{n+1}} \int_{\R} \left[ \mm^\nu(x,t^n-) \phi(x,t^n)- \mm^\nu(x,t) \phi (x,t) \right]\,dxdt\;.
\end{align*}
We note that the last two terms in $I_1^\nu$ correspond to the terms $j=N^\nu(t)+1$ and $j=0$, respectively and are different from the first term because these touch the boundaries of $\Omega^\nu$. Also the first term in $I_2^\nu$ accounts to the jump of $\mm^\nu$ across time steps. From Lemma~\ref{lem:RH-cond-rhov}, we deduce
\begin{align*}
I_1^\nu=&\int _0^\infty \left[  p(\rho^\nu_b(t)) \phi(b^\nu(t),t) - p(\rho^\nu_a(t)) \phi(a^\nu(t),t)\right]\,dt\\
&+\O(1)\|\phi\|_\infty \,T\,\sup_{0<t<T}\left\{ \tv \left\{\mm^\nu(\cdot,t);\R\right\}\right\}\cdot  \eta_\nu\;,
\end{align*}
while
\begin{align*}
I_2^\nu&\le  M\|\phi\|_{\infty} \sum_{n=0}^N \int_{t^n}^{t^{n+1}}\int_{-L}^L|\mm^\nu(x,t^n-)-\mm^\nu(x,t)| dxdt \\
&+  M\|\phi_t\|_{\infty}  \,T\,\sup_{0<t<T}\left\{ \int_\R |\mm^\nu(x,t)|\,dx\right\}\cdot  \DT_\nu\;,
\end{align*}
where we use that the test function has support within $[-L,L]\times (0,T)$. Using that $t\mapsto \int_\R\mm^\nu(t)$ is Lipschitz continuous, we get immediately that $I_2^\nu$ tends to zero and
\begin{align*}
I_1^\nu-\int _0^\infty \left[  p(\rho^\nu_b(t)) \phi(b^\nu(t),t) - p(\rho^\nu_a(t)) \phi(a^\nu(t),t)\right]\,dt\rightarrow0
\end{align*}
as $\nu\to\infty$. This establishes the claim that $\MM^\nu\to 0$ as $\nu\to\infty$.

\begin{remark}
We remark that the weak formulation given above can be expressed in terms of the distribution
\begin{equation}\label{def:m-hat2}
    \widehat \mm^\nu(\cdot, t) := \mm^\nu (\cdot,t) + \delta_{b^\nu(t)} P_b^\nu(t) - \delta_{a^\nu(t)} P_a^\nu(t)\,,\quad t>0\,.
\end{equation}
where
\begin{equation}\label{def:Pnu}
    P^\nu_b(t) :=  \int_0^t e^{-M(t-s)}p(\rho_b^\nu(s))\, ds\,,\quad P^\nu_a(t) :=  \int_0^t e^{-M(t-s)} p(\rho_a^\nu(s))\, ds\,,
\end{equation}
as follows: for any test function $\phi\in C^\infty_0\left(\R\times(0,\infty)\right)$, it holds
\begin{align*}
    & \RR^\nu = 
    \iint_{\R\times\R_+} \left\{\rho^\nu\phi_t + \mm^\nu\phi_x \right\}\; dxdt\longrightarrow 0\,,\\
    &
    \MM^\nu = 
    \int_0^\infty <\widehat \mm^\nu, \phi_t(\cdot,t)-M\phi(\cdot,t)> dt + \iint_{\Omega^\nu} \left[\frac{(\mm^\nu)^2}{\rho^\nu}+p(\rho^\nu)  \right]\phi_x\;dx dt\longrightarrow 0\;,
\end{align*}
as $\nu\to \infty$. The above identity for $\MM^\nu$ can be shown by direct computations and one observes that the Dirac deltas
in~\eqref{def:m-hat2} absorb the second term in $\MM^\nu$ present in the expression~\eqref{sec4.3:Mnu}, while $\RR^\nu$ remains unchanged.
\end{remark}

\subsection{Conservation of mass and momentum}\label{Sect:3.5}

In this subsection we address the questions of conservation of mass and momentum for the approximate solutions $\rho^\nu$, $\mm^\nu$ 
as defined in \eqref{def-rho-v-m-in}, \eqref{def-rho-m-out}. While the first one follows by construction, see \eqref{eq:constant-mass-nu},
for the momentum one needs to take into account the two singular terms that appear as the result of the limit $\rho\to 0$ described in the
Subsection~\ref{Sect:3.4}.

Let's define the functional
\begin{align}\label{def:Jnu}
\II^\nu(t) & := \int_{a^\nu(t)}^{b^\nu(t)} \mm^\nu  (x,t)\,dx 
+ P^\nu_b(t)-  P^\nu_a(t)
\end{align}
using~\eqref{def:Pnu}, that is interpreted as the \textit{approximate total momentum} and observe that
\begin{align*}
Q^\nu(t):=\int_0^t e^{-M(t-\tau)}\left[ p(\rho_b^\nu(\tau)) - p(\rho_a^\nu(\tau))\right]\;d\tau\le\alpha^2 \frac{1}{M} \sup_t \tv \{\rho^\nu(\cdot,t);\R \}
\end{align*}
which implies that $|Q^\nu(t)|$ remains bounded for all times from \S~\ref{Sect:3.3}. 

In the next lemma we show that, for every fixed $t\ge 0$, $\II^\nu(t)$ approaches 0 as $\nu\to+\infty$.

\begin{lemma}\label{lem:int-of-v} 
For every $t\ge 0$, one has  
\begin{equation}\label{eq:time-estimate-total-momentum}
\left|\II^\nu(t)\right|\le e^{-M t}\cdot e^{M \Delta t_\nu}\cdot \frac {M}\nu  + \frac{\widetilde C}M (\eta_\nu +\DT_\nu) + 
\widetilde C \eta_\nu \DT_\nu +\widetilde C \, \DT_\nu
\end{equation}
for a suitable constant $\widetilde C>0$, which is independent of $t$ and $\nu$.
\end{lemma}
\begin{proof} 

\textbf{Step 1.}\quad Let $t\in(t^n,t^{n+1}$), with $t$ not being a time of interaction. As in  Lemma~\ref{lem:RH-cond-rhov}, let 
\begin{equation*}
\mu^\nu_j = \frac{\Delta \mm^\nu}{\Delta \rho^\nu}\left(x^\nu_j(t)\right) \,,\qquad j=0,\ldots,N^\nu(t)+1    
\end{equation*}
be the exact speed given by the Rankine-Hugoniot conditions, related to the pair of states in the $(\rho,\mm)$ variables that are 
on the left and right side of the discontinuity at $x_j(t)$. In this notation, we recall that
\begin{equation}\label{def:exact-speed-RH-ab}
    \mu^\nu_0 = \vv^\nu(a^\nu(t)+,t) \,,\qquad \mu^\nu_{N^\nu(t)+1} = \vv^\nu(b^\nu(t)-,t)\,.
\end{equation}

One has that
\begin{align*}
   \frac{d}{dt} 
   \int_{a^\nu(t)}^{b^\nu(t)} \mm^\nu (x,t)\,dx
   & = - 
   \sum_{j=0}^{N^\nu(t)+1} \Delta \mm^\nu(x^\nu_j)\, (x_j^\nu)' = (I) + (II)\;,
\end{align*}
where
\begin{align*}
    (I)&= - 
   \sum_{j=0}^{N^\nu(t)+1} \Delta\mm^\nu
   (x^\nu_j) \mu^\nu_j \;,
\end{align*}
and 
\begin{align*}
    (II)&= - \sum_{j=0}^{N^\nu(t)+1} \Delta \mm^\nu (x^\nu_j) \left[(x_j^\nu)'- \mu^\nu_j \right]\;.
\end{align*}
About $(II)$, we use \eqref{RH-rho-tilde v} to find that
\begin{align*}
    |(II)|&\le C\,\eta_\nu \, \sup_{t>0} \tv\{\mm^\nu (\cdot,t);\R\}\\
    & = \widetilde C\, \eta_\nu \;,
    \end{align*}
for some $\widetilde C>0$, which is independent on $t$ and $\nu$. 

About $(I)$, the Rankine-Hugoniot condition for \eqref{eq:system_Eulerian_K=1}$_2$ gives
\begin{equation*}
    \Delta\mm^\nu(x^\nu_j) \mu^\nu_j = \Delta\left(\mm^\nu \vv^\nu + \alpha^2 \rho^\nu \right)\,,\quad  j=1,\ldots, N^\nu(t)\,.
\end{equation*}
On the other hand, for $j=0$ and $j=N^\nu(t) +1$ one has 
\begin{align*}
    \Delta\mm^\nu (a^\nu(t),t) \mu^\nu_0 &= (\mm^\nu \vv^\nu) (a^\nu(t)+,t)\,,\\[1mm]
    \Delta\mm^\nu (b^\nu(t),t) \mu^\nu_{N^\nu(t) +1} &= - (\mm^\nu \vv^\nu) (b^\nu(t)-,t)\,,
\end{align*}
where we used the expression for the propagation speed in \eqref{def:exact-speed-RH-ab}\,.

Therefore, $(I)$ rewrites as a telescopic sum:
\begin{align*}
    (I)&= - \sum_{j=1}^{N^\nu(t)} \Delta\left(
    \mm^\nu \vv^\nu\right)(x^\nu_j) -    (\mm^\nu \vv^\nu) (a^\nu(t)+,t) +  (\mm^\nu \vv^\nu) (b^\nu(t)-,t)\\
   &= -  \alpha^2 \rho^\nu (b^\nu(t)-)  + \alpha^2 \rho^\nu (a^\nu(t)+) \,.
\end{align*}

Thus, we have
\begin{align*}
\frac{d}{dt} \II^\nu(t) & = (I) + (II) +
 \frac{d}{dt} Q^\nu(t)\\
 &=(I) + (II) +
p(\rho_b^\nu(t)) - p(\rho_a^\nu(t))
 - M Q^\nu(t)\\
    & = (II) - M Q^\nu(t) \le \widetilde C\, \eta_\nu - M Q^\nu(t)
\end{align*}
and hence for $t'<t''$, with $t',\ t''\in (t^{n-1},t^n)$, we deduce
\begin{equation}\label{eq:bound-int-rho-v}
     \II^\nu(t'') - \II^\nu(t')
    \le \widetilde C\, \eta_\nu\, (t'' - t') -M\int_{t'}^{t''} Q^\nu(t)\,dt 
    \qquad  t',\ t''\in (t^{n-1},t^n)\quad \forall\, n\ge 1\,.
\end{equation}

\smallskip
\textbf{Step 2.}\quad 
Now, let's denote
$$a_n = |\II^\nu(t^n+)|\qquad n\ge 0\,.$$
As a consequence of \eqref{eq:u-v_fractional-step}, one has that 
\begin{align*}
\II^\nu(t^n+)=(1-M\DT_\nu) \int_{a^\nu(t^n)}^{b^\nu(t^n)} \mm^\nu  (x,t^n-)\,dx + Q^\nu(t^n)
\end{align*}
since at time steps, the map $t\mapsto Q^\nu(t)$ is Lipschitz continuous. From Lemma~\ref{S3:lemma1}, this is also true for the approximate free boundaries $a^\nu$ and $b^\nu$.
 Hence, using~\eqref{eq:bound-int-rho-v}, we reach
\begin{align*}
&\II^\nu(t^n+)=(1-M\DT_\nu)[\II^\nu(t^n-) - Q^\nu(t^n)] + Q^\nu(t^n)\\
&~ \le (1-M\DT_\nu) \left[\II^\nu(t^{n-1}+)+ \widetilde C \eta_\nu\DT_\nu\right] - M \underbrace{\int_{t^{n-1}}^{t^n} \left[(1-M\DT_\nu) Q^\nu(t) - Q^\nu(t^n)\right]\,dt}_{(*)}.
\end{align*}
Now we compute the derivative of $Q^\nu$, which is 
\begin{equation*}
    (Q^\nu)'(t) = -M Q^\nu(t) + \left[ p(\rho_b^\nu(t)) - p(\rho_a^\nu(t))\right]\,.
\end{equation*}
Recalling that there exists $\widehat C$ such that $|Q^\nu(t)|\le \widehat C$ for all $t$ and that $\rho^\nu$ is globally bounded, we deduce that $Q^\nu$ is Lipschitz continuous over $[0,+\infty)$ and let $\widehat C_1$ be its Lipschitz constant. Therefore
\begin{equation*}
    |(*)|\le M (\DT_\nu)^2 \widehat C + \widehat C_1 (\DT_\nu)^2 \,.
\end{equation*}
Thus
\begin{align*}
    a_n 
      &\le \left(1-M\DT_\nu\right) a_{n-1} +  \widetilde C\, ( \eta_\nu +\DT_\nu)\, \DT_\nu\,,
\end{align*}
where we assumed that $\DT_\nu>0$ and it satisfies $1-M\DT_\nu>0$ and, for convenience, we let 
$M^2 \widehat C + M\widehat C_1 \le\widetilde C$.
Applying the iteration formula with $0<r<1$
\begin{equation*}
    a_n\le r a_{n-1} + b \quad\Rightarrow\quad a_n \le r^n a_0 + b \frac{1-r^n} {1-r} \le  r^n a_0 +  \frac{b} {1-r} \,,
\end{equation*}
we find that
\begin{equation*}
   a_n \le \left(1-M\DT_\nu\right)^n a_0 +  \frac{ \widetilde C} {M} \, (\eta_\nu +\DT_\nu)
   \le e^{-Mn\DT_\nu } a_0 +  \frac{ \widetilde C} {M} \, ( \eta_\nu +\DT_\nu)
    \,,\qquad n\ge 1\,.
\end{equation*}
Recalling Step 1 in Subsection~\ref{S2.3}, we find that $a_0\le M/\nu$ and therefore \eqref{eq:time-estimate-total-momentum} holds for $t=t^n+$.

\smallskip
\textbf{Step 3.}\quad Finally, let $t\in (t_n,t_{n+1})$. By using \eqref{eq:bound-int-rho-v}, we find that
\begin{align*}
\II^\nu(t)-\II^\nu(t^{n}+)&\le  \widetilde C\, \eta_\nu \,(t-t^n)- M\int_{t^n}^{t} Q^\nu(t)\,dt\\
&\le\widetilde C\, \eta_\nu \,\DT_\nu+ \widetilde C \DT_\nu
\end{align*}
since $Q^\nu(t)$ is uniformly bounded. Hence, using Step 2, we find that \eqref{eq:time-estimate-total-momentum} 
holds for every $t\in (t_n,t_{n+1})$, $n\ge 1$.
\end{proof}

\subsection{Conclusion: Proof of Theorem~\ref{Th-1}}\label{Sect:3.6}
Let $(u,v)(y,t)$ be the entropy weak solution of \eqref{eq:system_Lagrangian}-\eqref{eq:init-data-lagr}-\eqref{eq:bc-lagrangian}, 
established in Theorem~\ref{Th-1-lagr}, and let $(u^{\nu}, v^{\nu})$ be a subsequence which converges, as $\nu\to\infty$, to 
$(u,v)$ in $L^1_{loc}\left((0,M)\times[0,+\infty) \right)$. By Lemma~\ref{lem:conv-a-b}, the corresponding subsequence 
$(\rho^{\nu}, \mm^{\nu})$ defined on $\R\times [0,+\infty)$ converges, as $\nu\to\infty$, to a function $(\rho,\mm)$ in
$L^1_{loc}\left(\R\times [0,+\infty)\right)$. Having defined
$$
\Omega=\{ (x,t);\ t\ge 0\,,\  x\in (a(t),b(t))\}\subset \R\times[0,+\infty)\,,
$$
we observe that $\rho^\nu$, $\mm^\nu$ converge to 0 pointwise outside $\overline \Omega$ and hence $\rho(x,t)=\mm(x,t) =0$ on 
$\overline \Omega^c$\,. Let's set $ \widehat \mm(\cdot, t)$ according to~\eqref{def:m-hat} and $ P_b(t)$ and $ P_a(t)$ as given
in~\eqref{def:Pnu-intro}. From \eqref{eq:constant-mass-nu} and the convergence
$$
\int_{a^\nu(t)}^{b^\nu(t)} \rho^\nu(x,t) \,dx \to \int_{a(t)}^{b(t)} \rho(x,t) \,dx\qquad \nu\to\infty\,,
$$
we deduce that the total mass is conserved $\forall\, t$, i.e. \eqref{cons-of-mass} holds true.

Next, recall that we set $M_1=0$ (at the beginning of Sect.~\ref{Sect:2}). Therefore we need to prove that the total momentum 
is constantly equal to 0. First, we observe that 
\begin{equation*}
    P^\nu_b(t)\to P_b(t)\,,\qquad P^\nu_a(t)\to P_a(t)\;, \qquad \nu\to\infty \,
\end{equation*}
because of Theorem~\ref{Th-1-lagr}, b). From Lemma~\ref{lem:conv-a-b}, we obtain that
\begin{equation*}
\II^\nu(t)\longrightarrow
 \int_{a(t)}^{b(t)} \mm(x,t)\,dx +  P_b(t) - P_a(t)=<\widehat\mm(\cdot,t),\phi_1 >\,, \qquad \nu\to\infty \,,
\end{equation*}
with $\phi_1=\phi_1(x)$ a test function being equal to $1$ on $I(t)$.
On the other hand, by \eqref{eq:time-estimate-total-momentum}, one has that $\II^\nu(t) \to 0$ for every $t\ge0$. 
Thus the total momentum is conserved as well $\forall\, t$, i.e. \eqref{cons-of-momentum} is established. 

Thanks to the properties of $(\rho^{\nu}, \mm^{\nu})$ and~\eqref{S4.3limits}, one can show that $(\rho,\mm)$ is an entropy weak solution
with concentration of the Cauchy problem to~\eqref{eq:system_Eulerian_K=1} in the sense of Definition~\ref{entropy-sol}
satisfying~\eqref{S1:rho-eq-phi2} and \eqref{S1:m-eq-phi} and in particular, the structure \eqref{solution-structure} holds. 
Observe that due to conservation of mass and momentum, the integral identity \eqref{S1:m-eq-phi} yields \eqref{S1:m-eq-phi2}.

Moreover, \eqref{entropy-cond_rho-m} reads as
\begin{equation}\label{entropy-cond_rho-m_M1=0} 
       \partial_t \eta(\rho,\mm) + \partial_x q(\rho,\mm)\le - \eta_{\mm} M \mm
\end{equation}
due to conservation of mass and momentum, and the fact that $M_1=0$. Thanks to the analogous property in the $(u,v)$ variables, 
stated in Theorem~\ref{newthm}, we conclude that \eqref{entropy-cond_rho-m_M1=0} holds in the sense of distributions in the interior 
of $\Omega$, $\{ (x,t);\ t> 0\,,\  x\in (a(t),b(t))\}$.


\Section{Time-asymptotic flocking}\label{Sect:4}  

In this section, we prove Theorem~\ref{Th-2}. The aim is to introduce new supporting functionals assembled using the linear functionals of Section~\ref{subsect:Lxi}
that better capture the structure of the solution, especially as time goes to infinity.

\subsection{Decay estimates for the undamped system}\label{S4.1}
We give the definition of the generation order $g\ge 1$ of a wave present in the approximate solution contructed in Section~\ref{Sect:2}.
We assign inductively to each wave $\alpha$ a generation order $g_\alpha\ge 1$ as in Refs.~\cite{AC_SIMA_2008,ABCD_JEE_2015} according to the following procedure: First, at time $t=0$ each wave has order equal to $1$. Then, assume that two waves $\alpha$ and $\beta$ interact at time $t$ having generation order $g_\alpha$ and $g_\beta$, respectively. If $\alpha$ and $\beta$ belong to different families, then the outgoing waves of those families inherit the same order of the corresponding incoming waves. On the other hand, if $\alpha$ and $\beta$ belong to the same family, then the outgoing wave of that family takes the order $\min\{g_\alpha,\,g_\beta\}$, while the other new outgoing wave is assigned the order $\max\{g_\alpha,\,g_\beta\}+1$.

Next, we define the weighted total variation functional $F_k(t)$ of generation $g=k$ to be
\begin{equation}
F_k(t) \dot = \sum_{\eps>0, \ g_\eps=k} |\eps| + \xi \sum_{\eps<0, \ g_\eps=k} |\eps|\,,\qquad k\ge 1\;,
\end{equation}
while the weighted total variation functional $\widetilde F_k(t)$ of generation $\ge k$ is
\begin{equation}
\widetilde F_k(t) \ \dot =\  \sum_{j\ge k} F_j(t) = \sum_{\eps>0, \ g_\eps\ge k} |\eps| + \xi \sum_{\eps<0, \ g_\eps\ge k} |\eps|\,,\qquad k\ge 1\;.
\end{equation}
In the following analysis throughout this section, the weight $\xi$ of the part of shocks in $F_k$ and $\tilde F_k$ is greater or equal to $ 1$. The aim is to estimate the change of these functionals in time for some values of $\xi\ge 1$ depending on the initial bulk $q$ and the generation order. This will allow us to capture the behavior of the solution at $t=\infty$.

Now, for $k\in\N$, we define

\begin{itemize}
\item[(i)] $I_{k,\ell}(t)$: the set of times $\tau<t$, different from the time steps, at which an interaction between two waves of the same family,
one of order $k$ and the other of order $\ell$, with $k\ge \ell\ge 1$
\item[(ii)] ${\II}_k(t)$: the set of times $\tau<t$, again different from the time steps, when two waves of the same family interact, 
with maximum generation order $=k$, that is
\begin{equation*}
\II_k =\cup_{\ell=1}^k I_{k,\ell} = \left(I_{k,k}\right) \cup  \left(\cup_{\ell=1}^{k-1} I_{k,\ell}\right) \,.
\end{equation*}
\end{itemize}
We recall a result from Refs.~\cite{AC_SIMA_2008,ABCD_JEE_2015} for the case of $M=0$, that is, no source term in \eqref{eq:system_Lagrangian}.
The following proposition is stated, in a slight different context, in Prop. 6.5, p. 155 of Ref.~\cite{AC_SIMA_2008}, with an improvement 
on the range of the parameter $\xi$ that follows from Prop. 5.8 in Ref.~\cite{ABCD_JEE_2015} (restated in the current paper in
Lemma~\ref{lem:Delta-L-xi}).

\begin{lemma} \label{prop:F-k}
For $q>0$, assume that $L(0+)\le q$ and that
\begin{equation}\label{eq:stronger-hyp-on-xi}
1\le \xi \le 1/{\sqrt{c(q)}}\,.
\end{equation}
Then the followings hold true, 
\begin{equation}\label{eq:I_kk}
\xi [\Delta F_{k+1}]_+ \le [\Delta F_k]_-\,, \qquad  \tau\in I_{k,k}, \quad k\ge 1
\end{equation}
and 
\begin{equation}\label{eq:I_kl}
\xi [\Delta F_{k+1}]_+ \le  \Bigl([\Delta F_k]_-   -  [\Delta F_\ell]_+ \Bigr)\,,\qquad   \tau\in I_{k,\ell}\,,\quad  1\le \ell<k\,.
\end{equation}
\end{lemma}

\subsection{Decay estimates for the damped system}\label{S4.2}
In this subsection, we analyze the variation of the functionals $L$, $L_\xi$, $\widetilde F_k$ for the approximate solutions 
defined in Subsection~\ref{S2.3} for  system~\eqref{eq:system_Lagrangian}, by using as an intermediate step a new functional weighted by the generation order of the fronts. 

\smallskip\par\noindent
Now we define
\begin{equation}\label{Vweightedgen}
V(t)=\sum_{k\ge 1} \xi^k F_k(t)
\end{equation}
with $\xi\in[1,{c(q)}^{-1/2}]$ and call $V$ \emph{the total variation that is weighted by generation order}. 
In the next lemma, we estimate the variation of $V$ for the approximate solution $(u,v)$, i.e. for a fixed time step $\DT>0$. In this subsection, we skip the index $\nu$ again for the approximate solution $(u,v)$.

\begin{lemma}
Let $(u,v)$ be the approximate solution to~\eqref{eq:system_Lagrangian} with a time step $\Delta t>0$. Then for each $n$, it holds
\begin{align}\label{DeltaV-lemma}
V(t)\le \left(1+\frac{(\xi^2-1)}{2} M\Delta t\right)^n V(0+)\;,
\end{align}
while $t\in[t^n,t^{n+1})$ and $\xi\in[1,{c(q)}^{-1/2}]$.
\end{lemma}

\begin{proof}
We estimate the change of $V$ first at interaction times and then at time steps. 

First, let $\tau\in I_{k,k}$ be an interaction time when two of the same family and same generation order interact. Then, from~\eqref{eq:I_kk}, we have
\begin{align*}
\Delta V(\tau)&=\xi^{k+1}\Delta F_{k+1}+\xi^k\Delta F_k\\
&=\xi^{k+1}[\Delta F_{k+1}]_{+}-\xi^k[\Delta F_k]_{-}\\
&= \xi^k \left(  \xi [\Delta F_{k+1}]_{+}-[\Delta F_k]_{-} \right) 
\le 0. 
\end{align*}
Also, at an interaction time $\tau\in I_{k,\ell}$, with $1\le \ell< k$, we apply~\eqref{eq:I_kl} to estimate the change in $V$ as follows:
\begin{align*}
\Delta V(\tau)&=\xi^{k+1}\Delta F_{k+1}+\xi^k\Delta F_k+\xi^\ell\Delta F_\ell\nonumber\\
&=\xi^{k+1}[\Delta F_{k+1}]_{+}-\xi^k[\Delta F_k]_{-}+\xi^\ell\Delta F_\ell\nonumber\\
&= \xi^k \left(\xi [\Delta F_{k+1}]_{+}- [\Delta F_k]_{-}  \right) +\xi^\ell\Delta F_\ell\\
&\le - \xi^{k}[\Delta F_\ell]_{+} +\xi^\ell\Delta F_\ell \\
&\le - \xi^{k}[\Delta F_\ell]_{+} +\xi^\ell [\Delta F_\ell]_{+}\,.\nonumber
\end{align*}
Therefore, since $\xi\ge 1$ and $k>\ell$, we are able to conclude that
\begin{align}
\Delta V(\tau) 
&\le[\Delta F_\ell]_{+}\, (\xi^\ell-\xi^k)\le 0\;. \label{DeltaV-1}
\end{align}
At interaction times of waves of different families, it holds $\Delta F_k=0$ for all $k$ and this yields $\Delta V(\tau)=0 $. 
It remains to check the variation of $V$ at the time steps $t^n$. Let $t^n$ be a time step and a front of strength $\eps^-$ and generation $g^{-}=g_{\eps^-}$
that is present at time $t=t^n-$ is updated at the time step $t^n$ according to~\eqref{eq:u-v_fractional-step} and produces a front of strength $\eps^+$  
and generation $g^{-}$ again and a reflected one of strength $\eps_{refl}$ and generation $g^{-}+1$. Then
\begin{align}\label{DeltaV-3}
\Delta V(t^n)&=\sum_{k\ge 1} \xi^k\Delta F_k(t)=\sum_{k\ge 1}  \xi^k\left [A_k+B_k\right]
\end{align}
where $A_k$ is the variation of strengths when the front $\eps^{-}$ is of generation order $g^{-}=k\ge 1$, i.e.
\begin{equation}
    A_k=\sum_{\eps^->0, g^{-}=k} (|\eps^+|-|\eps^{-}|)+\xi\sum_{\eps^-<0, g^{-}=k} |\eps^+|-|\eps^{-}|)
\end{equation}
while $B_k$ accounts to fronts $\eps^{-}$ of generation order $g^{-}=k-1\ge 1$, i.e.
\begin{equation}
B_k=\xi \sum_{\eps^->0, g^{-}=k-1} |\eps_{refl}|+\sum_{\eps^-<0, g^{-}=k-1} |\eps_{refl}|\;,\qquad k\ge 2
\end{equation}
while $B_1\equiv 0$.
By rearranging the terms $B_k$ in the sum, we can rewrite~\eqref{DeltaV-3} as
\begin{align}\label{DeltaV-4}
\Delta V(t^n)&=\sum_{k\ge 1}  \xi^k\left [A_k+\xi B_{k+1}\right]\;.
\end{align}
Now,
\begin{align}\nonumber 
A_k+\xi B_{k+1} 
&=  -  \sum_{\eps^->0, g^{-}=k}|\eps_{refl}|-\xi   \sum_{\eps^-<0, g^{-}=k} |\eps_{refl}|\\
& \qquad +\xi^2  \sum_{\eps^->0, g^{-}=k}|\eps_{refl}|+\xi \sum_{\eps^-<0, g^{-}=k} |\eps_{refl}|\nonumber\\
& =(\xi^2-1)  \sum_{\eps^->0, g^{-}=k} |\eps_{refl}|\nonumber\\
&\le  \frac{(\xi^2-1)}{2} M\Delta t \sum_{\eps^->0, g^{-}=k} |\eps^{-}|\nonumber\\
 & \le \frac{(\xi^2-1)}{2} M\Delta t F_k(t^n-)\;, \label{DeltaV-5}
\end{align}
by Proposition~\ref{prop:estimate-time-step}. Combining~\eqref{DeltaV-4} with~\eqref{DeltaV-5}, we arrive at
\begin{equation}\label{DeltaV-6}
\Delta V(t^n)\le \frac{(\xi^2-1)}{2} M\Delta t \, V(t^n-).
\end{equation}

Taking into account the change of $V$ at interaction times and at time steps, we conclude
\begin{align}\label{DeltaV-7}
V(t)\le &\left(1+\frac{(\xi^2-1)}{2} M\Delta t\right) V(t^n-)\nonumber\\
\le &\left(1+\frac{(\xi^2-1)}{2} M\Delta t\right) V(t^{n-1}+)\nonumber\\
\le &\left(1+\frac{(\xi^2-1)}{2} M\Delta t\right)^n V(0+)\;.
\end{align}
for all $t\in[t^n,t^{n+1})$. The proof is complete.
\end{proof}

\subsection{Conclusion: Proof of Theorem~\ref{Th-2}}
Let us recall that the approximate solution satisfies $u^\nu(y,t)\in[u_{inf}^\nu,u_{sup}^\nu] $ as shown in Lemma~\ref{Lemmauinfsup}.
Taking into account that the characteristic speeds $\lambda$ (c.f.~Section \ref{S2.2}) range between $\lambda_{min}^\nu=\frac{\alpha}{u_{sup}^\nu}>0$ and 
$\lambda_{max}^\nu=\frac{\alpha}{u_{inf}^\nu}$, we get that the maximal time length $T_1^\nu$ for the waves of first generation to reach the boundaries $y=0,M$ is
\begin{align}\label{T1-new}
T_1^\nu&=\frac{M}{\alpha}u_{sup}^\nu =\frac{e^{2q} M}{\alpha} \min\{\tilde{u}_0^\nu, \tilde{u}_M^\nu\},
\end{align}
where $\tilde{u}_0^\nu, \tilde{u}_M^\nu$ are given in~\eqref{def:u0-nu_uM-nu}. In other words, there are no waves of generation order $k=1$ present for times $t>T_1^\nu$, hence $\tilde{F}_1(t)=\tilde{F}_2(t)$ for $t>T_1^\nu$.
The waves of generation order $k>1$ inherit this property in the following sense: 
\begin{equation}\label{S4: k gen prop-1}
F_k(t)=0 \quad\text{for  } t>k T_1^\nu
\end{equation}
\begin{equation}\label{S4: k gen prop-2}
\tilde{F}_1(t)=\tilde{F}_2(t)=\dots=\tilde{F}_k(t) \quad\text{for  } t\in[(k-1) T_1^\nu, k T_1^\nu]\,.
\end{equation}

Now, we observe that $V(0+)=\xi F_1(0)$ since $F_k(0+)=0$ for $k\ge 2$. Using that $\xi\ge 1$, we have the relation
\begin{align}\label{tildeFkt-V}
\tilde{F}_k (t) \le \frac{1}{\xi^k} \sum_{j\ge k} \xi^j F_j(t)
\le  \frac{1}{\xi^k} V(t)
\end{align}
between $\tilde{F}_k$ and $V$ at any time $t$.
Combining \eqref{DeltaV-lemma} with~\eqref{tildeFkt-V}, we arrive at the estimate
\begin{align}
    \tilde{F}_k (t) &\le\left( \frac{1}{\xi}\right)^{k-1} \left(1+\frac{(\xi^2-1)}{2} M\Delta t\right)^n F_1(0)\;\nonumber\\
    & \le \left( \frac{1}{\xi}\right)^{k-1} \exp\left(\frac{(\xi^2-1)}{2}M t \right) F_1(0)\nonumber\\
   & =  \xi \exp\left(\frac{(\xi^2-1)}{2}M t - k \log\xi\right) F_1(0) \label{eq:exp-xi-k}
\end{align}
for all times $t\in[t^n, t^{n+1})$. 
Now, we claim that 
\begin{equation}\label{S3.claimlim}
\lim_{t\to\infty} \tilde{F}_1(t)=0.
\end{equation}
Indeed, using the time $T_1^\nu$ elapsed for a generation order to abandon the domain $[0,M]$ and 
by property~\eqref{S4: k gen prop-1}--\eqref{S4: k gen prop-2}, we observe that
\begin{equation}\label{S4: k gen prop-3-NEW}
    \tilde{F}_1 (t)=\tilde{F}_k (t)\qquad t\in ((k-1) T_1^\nu, kT_1^\nu]\,. 
\end{equation}
For every $t>0$, let $k\in \N$ be the only value such that $t\in ((k-1) T_1^\nu, kT_1^\nu]$.  
Hence, one has $- k \le - t/ T_1^\nu$ and from \eqref{eq:exp-xi-k} we obtain

\begin{equation}\label{eq:bound-tilde-F1-1}
\tilde{F}_1 (t) \le  \xi e^{- t\, \lambda^\nu(\xi)  } F_1(0)\;,
\end{equation}
where
\begin{equation*}
 \lambda^\nu(\xi)\dot = \frac{(1- \xi^2)}{2}M  +  \frac{\log\xi}{T_1^\nu}\,.
\end{equation*}
We first show that, under condition~\eqref{Th-2assumption}, the quantity $\lambda^\nu$ is positive for suitable values of $\xi$.
%
Indeed, \eqref{T1-new} and~\eqref{eq:limits-init-data} imply
$$
\lim_{\nu\to\infty} T_1^\nu=\frac{e^{2q} M}{ \alpha\max\left\{\rho_0(a_0+),\rho_0(b_0-)\right\}}  ~ \dot = ~T_1^*\,.
$$
Moreover, condition~\eqref{Th-2assumption} yields $M T_1^* <1$. Then for all $\nu$ sufficiently large, i.e. for all $\nu\ge \nu_0$ with $\nu_0$ large enough, one has $M T_1^\nu <1$.
Now we consider $\lambda^\nu(\xi)$ for $\nu\ge \nu_0$
and observe that
$$
\lambda^\nu(1) = 0\,,\qquad \frac{d\lambda^\nu}{d\xi}= - \xi M + \frac{1}{\xi T_1^\nu} = \frac M \xi \left( \frac 1 {MT_1^\nu}- \xi^2 \right)\,.
$$
Therefore, $ \lambda^\nu(\xi)$ is strictly increasing on the interval $[1, (M T_1^\nu)^{-1/2}]$, with maximum value
\begin{equation*}
\lambda^\nu( (M T_1^\nu)^{-1/2} ) =\frac 1 {2 T_1^\nu} \left( M T_1^\nu -1 - \log(M T_1^\nu) \right)>0\,.
\end{equation*}
Hence, for every $\nu\ge \nu_0$, we restrict the range of values of $\xi$ to the interval
$$
\xi\in (1,\bar \xi_\nu]\qquad \bar \xi_\nu \dot =  \min\{  \frac 1{\sqrt {c(q)}}, \frac 1 {\sqrt {M T_1^\nu}} \} >1 \,,
$$
and this yields $\lambda^\nu (  \xi)>0$.
Now let's define 
$$
C_1^\nu = \lambda^\nu ( \bar \xi_\nu)\,,\qquad C^\nu_2 = (\bar \xi_\nu)^2\,.
$$
As $\nu\to\infty$, one has 
\begin{equation*}
\bar \xi_\nu\to\bar\xi,\qquad \lambda^\nu(\xi) \to \lambda(\xi),\quad\text{for }\xi\in (1,\bar \xi_\nu]
\end{equation*}
and
\begin{equation*}
 C_1^\nu \to C_1 ~\dot =~  \lambda(\bar \xi) \,,\qquad C_2^\nu \to C_2 ~\dot =~ (\bar \xi)^2\;,
\end{equation*}
where
\begin{align*}
 \bar\xi & \dot =  \min\{  \frac 1{\sqrt {c(q)}}, \frac 1 {\sqrt {M T_1^*}} \}>1\,,\qquad  \lambda(\xi) \dot = \frac{(1- \xi^2)}{2}M  +  \frac{\log\xi}{T_1^*}\qquad\xi\in (1,\bar \xi]\;.
\end{align*}
Therefore, for some $\nu_1\ge \nu_0$,  one has
$C_1^\nu \ge C_1 /2$\,, $C_2^\nu \le 2 C_2$, for all $\nu\ge \nu_1$.
Using the above estimates and \eqref{bound-on-Lxi}, inequality \eqref{eq:bound-tilde-F1-1} yields for $\xi=\bar \xi_\nu$ and $\nu\ge \nu_1$:
\begin{align}
L_{in}(t) \le \tilde{F}_1 (t) &\le \xi e^{- t\, C_1^\nu} F_1(0)\le 2C_2 e^{- \frac{C_1}2 t} q \;, \label{S4:Lindecay}
\end{align}
since $F_1(0)\le\xi L_{in}(0+) \le q$ by~\eqref{LLin-decreases}.
Here, $C_1$ and $C_2$ are the positive constants independent of $\nu$ and $t$ obtained above in the limit $\nu\to\infty$. Thus, claim~\eqref{S3.claimlim} follows and more precisely, the total variation of the approximate sequence decays exponentially fast as time tends to infinity.

Having now~\eqref{S4:Lindecay}, we can deduce by Lemmas~\ref{prop:equivalence-Lin},~\ref{lem:bounds-on-bv} 
and estimate~\eqref{eq:tv-v_m}, that as $t\to \infty$, the approximate solution $v^\nu$ satisfies 
\begin{equation}\label{S4:vnutinfinity}
\tv v^\nu(t)\le C_2' e^{-  \frac{C_1}2 
t}\to 0\end{equation}
independently of $\nu$ and $v^{\nu}(t)\to v_\infty^\nu$ for some constant $v_\infty^\nu$, 
and that $u^{\nu}(t)$ tends to a constant $u_\infty^\nu$, i.e.
\begin{equation}\label{S4:unutinfinity}
u^{\nu}(y,t)\to u_\infty^\nu\qquad\forall\, y\in(0,M)\,.
\end{equation}
It should be noted that $u^\nu_\infty\in[u^\nu_{inf},u^\nu_{sup}]$ by~\eqref{uinfsup}.

Furthermore, by~\eqref{eq:bound-on-TV_at_y} and~\eqref{S4:Lindecay}, we have
$$
    W^\nu_y(t)\le  \widetilde C_1  L_{in}(0)  
    + 4 \widetilde C_2 M \frac{C_2}{C_1}=:K
$$
for all $t>0$, with $K$ being a positive constant independent of time and $\nu$. Therefore, applying then the arguments in the proof of Lemma~\ref{lem:vertical-bounds}, there exist $K_1$, $L_1$ independent of $\nu\in\N$ such that \eqref{bound-bv-time} and \eqref{bound-integral-time}
 hold for every $T>0$.
As a consequence, for all $y\in[0,M]$
\begin{align}\label{S4bound-bv-time}
    \tv\{u^\nu(y,\cdot);[0,\infty)]\} \le K_1\,,\quad
     \tv\{v^\nu(y,\cdot);[0,\infty)]\} \le K_1\,,
\end{align}
and 
\begin{align}\label{S4bound-integral-time}
   \int_0^\infty \left| v^\nu(y_1,t) -  v^\nu(y_2,t)\right|\, dt \le L_1 \left| y_2 - y_1  \right|\qquad \forall\, y_1\,,\, y_2\in [0,M]\,.
\end{align}
The same property as \eqref{S4bound-integral-time} holds for $u^\nu$ as well. The improvement here, is that under condition~\eqref{Th-2assumption}, bounds~\eqref{S4bound-bv-time}--\eqref{S4bound-integral-time} hold uniformly in time. 
In particular, we obtain that $$(u^\nu,v^\nu)(y,\cdot) - (u_\infty^\nu,v_\infty^\nu)\in L^1((0,\infty))\;.$$

In view of the above analysis and Theorem~\ref{Th-1-lagr}, there exists a subsequence of $\{(u^{\nu}, v^{\nu})\}$ 
which converges, as $\nu\to\infty$, to a function $(u,v)$ in $L^1_{loc}\left((0,M)\times[0,+\infty) \right)$. By possibly passing to a subsequence,
the sequences $\{u_\infty^\nu\}$, $\{v_\infty^\nu\}$ converge as $\nu\to\infty$:
$$u_\infty~\dot=\lim_{\nu\to\infty} u_\infty^\nu>0\,, \qquad v_\infty~\dot=\lim_{\nu\to\infty} v_\infty^\nu\,,$$
with $u_\infty\in[u_{inf},u_{sup}]$.

Moreover, by the analysis in Section~\ref{Sect:3}, we get
\begin{equation}\label{S4eq:intnubdd}
b^\nu(t) - a^\nu(t)\to M u_\infty^\nu\qquad t\to\infty
\end{equation} 
from~\eqref{eq:support-length} and hence the length of the interval $I^\nu(t)$ remains bounded independently of time and $\nu$. 
Also, the approximate sequence $(\rho^\nu, \vv^\nu,\mm^\nu )$ defined in~\eqref{def-rho-v-m-in}--\eqref{def-rho-m-out} 
via $(u^\nu, v^\nu)$ satisfying~\eqref{S4:vnutinfinity}--\eqref{S4:unutinfinity}, has the properties

\begin{align}\nonumber
&\sup_{x\in I^\nu(t)} \left| \rho^\nu(x,t) - \frac{1}{u_\infty^\nu}\right|  \to 0,\quad
\sup_{x\in I^\nu(t)} \left| \vv^\nu(x,t)-v^\nu_\infty\right|\to 0, \\ \label{S4eq:supeulernu}
&\sup_{x\in I^\nu(t)} \left| \mm^\nu(x,t) - \frac{v^\nu_\infty}{u_\infty^\nu}\right|\to 0\;,
\end{align}
as $t\to\infty$. Moreover, thanks to \eqref{eq:identity-tvlnu}, the identity $u^\nu=\{\rho^\nu\}^{-1}$ and \eqref{S4:Lindecay}\,, 
we find that
\begin{align}\label{eq:tv-rho-decays}
\tv \{\rho^\nu(\cdot,t); I^\nu(t)\} &\le u^\nu_{sup} \tv\{\ln(\rho^\nu)(\cdot,t);I^\nu(t)\} \nonumber\\
&= \frac12  {u^\nu_{sup}} L_{in}(t) \le C_3 e^{- \frac{C_1}2 t} \,,
\end{align}
for some $C_3>0$ constant, which is independent of $\nu$. Under the notation
$$ \osc \{\vv^\nu; I^\nu(t)\} := \esssup_{x_1,x_2\in I^\nu(t)} |\vv^\nu(x_1,t) - \vv^\nu(x_2,t)|\,,$$
and the analogue notation for $ \osc \{v^\nu(t); (0,M) \}$, it should be noted that \eqref{S4:vnutinfinity} yields
\begin{equation}\label{S4:vnutinfinitytilde}
\osc \{\vv^\nu(t);I^\nu(t)\} =  \osc \{v^\nu(t); (0,M) \} \le \tv v^\nu(t)\le C_2' e^{-  \frac{C_1}2 
t}\to 0\;.\end{equation}
Furthermore, by Lemma~\ref{lem:int-of-v}, we immediately get
\begin{equation}
\label{S4eq:time-estimate-total-momentum}
    \limsup_{t\to\infty}\left|\II^\nu(t)\right|
    \le  \frac{\widetilde C}M (\eta_\nu+ \DT_\nu ) + \widetilde C \eta_\nu \DT_\nu+  \widetilde C\DT_\nu\;.
\end{equation}
On the other hand, the \textit{approximate total momentum} $\II^\nu(t)$, see \eqref{def:Jnu}, can be written
\begin{align*}
\II^\nu(t)&= \int_{a^\nu(t)}^{b^\nu(t)} \left(\mm^\nu(x,t)- \frac{v^\nu_\infty}{u_\infty^\nu}\right)\,dx   + \frac{v^\nu_\infty}{u_\infty^\nu}\left(b^\nu(t)-a^\nu(t) \right) + P^\nu_b(t) - P^\nu_a(t)
\\   \\
& \dot = (A) + (B)+ (C)
\,.
\end{align*}
Using~\eqref{S4eq:supeulernu} and~\eqref{S4eq:intnubdd}, as $t\to\infty$ we have
$$
|(A)|  \le \sup_{x\in I^\nu(t)} \left| \mm^\nu(x,t) - \frac{v^\nu_\infty}{u_\infty^\nu}\right| \cdot 
\left( {b^\nu(t)-a^\nu(t)}\right)\to 0\,,\quad  (B)\to M v^\nu_\infty
$$
since the factor $(b^\nu(t)-a^\nu(t))$ is bounded from~\eqref{S4eq:intnubdd}. Moreover, by using \eqref{def:Pnu} and \eqref{eq:tv-rho-decays}, we find that
\begin{align*}
\left|(C) \right| \le {\alpha^2} \int_0^t  e^{-M(t-\tau)} \tv \{\rho^\nu(\cdot,\t); I^\nu(\tau)\} \, d\t
\le {\alpha^2} C_3 e^{-Mt} \int_0^t e^{(M- \frac{C_1}2 )\tau} \,d\tau \to0
\end{align*}
as $t\to\infty$. In conclusion, the limit of $\II^\nu(t)$ exists,
$
\lim_{t\to\infty}\II^\nu(t)= M v^\nu_\infty 
$, and from \eqref{S4eq:time-estimate-total-momentum}, we have
\begin{equation*}
   \lim_{t\to\infty}\left|\II^\nu(t)\right|= M |v^\nu_\infty| \le  \frac{\widetilde C}M (\eta_\nu+\DT_\nu) + \widetilde C \eta_\nu \DT_\nu+  \widetilde C\DT_\nu\;.
\end{equation*}
Passing to the limit as $\nu\to\infty$, we arrive at
\begin{equation}\label{eq:v-infty-0}
   |v_\infty|  =   \lim_{\nu\to\infty} |v^\nu_\infty| \le   \lim_{\nu\to\infty} \left[\frac{\widetilde C}{M^2} (\eta_\nu+\DT_\nu)  +\frac{ \widetilde C }{M} (\eta_\nu +1)\DT_\nu\right]  = 0\;,
\end{equation}
and hence, conclude that $v_\infty=0$. 

Next, we  claim that
\begin{equation}
\label{S4eq:conv-int_v-0_dt-bis}
    \lim_{\nu\to\infty} \int_0^{T_\nu} v^\nu(0+,s) \, ds = \int_0^\infty v(0+,s) \, ds\,,\qquad T_\nu = (\eta_\nu)^{-1/2}\,,
\end{equation}
from which we obtain that $v(0+,\cdot) \in L^1(\R_+)$. Therefore, since $\lim_{t\to\infty} v(0+,t)$ exists, 
we conclude that $v(0+,t)$ converges to 0 as $t\to\infty$. 

To prove~\eqref{S4eq:conv-int_v-0_dt-bis}, we express the integral as:
\begin{align*}
 \int_0^{T_\nu} v^\nu(0+,s) \, ds  &=  \int_0^{T_\nu} \left(v^\nu(0+,s) - v_\infty^\nu  \right)\, ds + \underbrace{v_\infty^\nu T_\nu}_{\to 0}
\end{align*}
as $\nu\to \infty$ and the last term tends to $0$ thanks to \eqref{eq:v-infty-0}. 

Let's examine the first integral on the right hand side. 
First we observe that, thanks to \eqref{eq:conv-int_v-0_dt}, we can take a subsequence as $\nu\to\infty$, 
still denoted by $\left(u^\nu,v^\nu\right)$ such that $\left(v^\nu(0+,s) - v_\infty^\nu\right) \to v (0+,s)$ pointwise a.e. 
on $[0,+\infty)$, as $\nu\to\infty$.

Moreover we recall \eqref{eq:v-along-x=0} and find that
\begin{equation*}
\left| v^\nu(0+,t-) -  v^\nu(0+, s) \right|   \le  2\alpha \cosh(q) L_{in}(s)  \le C''_2 e^{- \frac{C_1}2 s}  \qquad \forall~t\ge s\,,
\end{equation*}
with $C''_2 = 4\alpha \cosh(q) q \, C_2$,  thanks to \eqref{S4:Lindecay}. As $t\to\infty$, we get
\begin{equation*}
\left| v_\infty^\nu  -  v^\nu(0+, s) \right|  \le C''_2 e^{- \frac{C_1}2 s} 
\end{equation*}
which is independent of $\nu$. Therefore by dominated convergence theorem we conclude that
$$
\int_0^{T_\nu} \left(v^\nu(0+,s) - v_\infty^\nu  \right)\, ds \to \int_0^\infty v(0+,s) \, ds\qquad \nu\to\infty
$$ 
and claim~\eqref{S4eq:conv-int_v-0_dt-bis} holds. 

Last, taking the limit as $\nu\to\infty$ in the approximate sequence $(\rho^\nu,\mm^\nu)$ as it is shown in Section~\ref{Sect:3}, 
we recover an entropy weak solution with concentration $(\rho,\mm)$ to~\eqref{eq:system_Eulerian_K=1} obtained in Theorem~\ref{Th-1},
with $M_1=0$ and compact support within $I(t)=[a(t),b(t)]$. Moreover, in view of the above analysis,
the free boundaries $a(t)$ and $b(t)$ converge to a finite limit:
$$\lim_{t\to\infty} a(t)=a_\infty\dot=\,a_0 + \int_0^\infty v(0+,s) ds,\qquad \lim_{t\to\infty} b(t)=b_\infty\dot =
    a_\infty + M u_\infty\;. $$
Also, thanks to bounds \eqref{S4:vnutinfinitytilde}, having possibly redefined $v$ on a set of measure $0$, 
we conclude that 
\begin{equation*}
\osc \{\vv(t);I(t)\} =  \osc \{v(t); (0,M) \} \le \tv v(t)
\le C_2' e^{-  \frac{C_1}2 t}\to 0\,.
\end{equation*}
and
 \begin{equation}
 \rho_\infty\dot =\lim_{t\to\infty}\rho(x,t) = \frac{1}{u_\infty},\qquad \lim_{t\to\infty}\vv(x,t)=\lim_{t\to\infty}\mm(x,t) =0,
 \end{equation}
 for all $x\in I(t)$, with $u_\infty\in[u_{inf},u_{sup}]$. Thus, the entropy weak solution $(\rho,\mm)$ with concentration to~\eqref{eq:system_Eulerian_K=1} under the condition~\eqref{Th-2assumption} satisfies~\eqref{eq:v-shrinks-to-0} and this immediately
 implies that $(\rho,\mm)$ admits time-asymptotic flocking. The proof of Theorem~\ref{Th-2} is complete. 


\appendix

\Section{Proof of technical Lemmas}

\subsection{Proof of Proposition~\ref{S2:prop:estimate-time-step}}\label{subsec:app1}
$(a)$ 
Recalling \eqref{eq:u-v_fractional-step}, then equation~\eqref{eq:one-timestep} is obtained by equating $u_r - u_\ell$ before and after the time step and by using the definition \eqref{eq:strengths} of the strengths. On the other hand, ~\eqref{eq:two-timestep} follows by the equation $v^+_r - v^+_\ell  = (1-M\DT) \left(v^-_r - v^-_\ell\right)$ and the parametrization of the wave curves \eqref{eq:lax13}.

\medskip
$(b)$ We first set  
\begin{align*}
& x=\eps_2^-\,,& \eps_1^+ =y(x,s) \,,\\
& s=\DT\,,  &\eps_2^+ = x + y(x,s)\,.
\end{align*}
Then to prove $(b)$, it suffices to prove that $y(x,s)$ is well defined and it satisfies
\begin{equation}\label{time-step_sign-properties}
x y(x,s) <0\,,\qquad   x\left(  x + y(x,s)\right)>0\qquad \mbox{ if } x\not=0\not=s\,.
\end{equation}

Identities~\eqref{eq:one-timestep}, \eqref{eq:two-timestep} lead to the implicit equation for $y(x,s)$
\begin{equation}\label{eq:implicit-time-step}
h\left(y(x,s)\right) + h\left(x+y(x,s)\right) = h(x)\left( 1 - Ms \right)\,. 
\end{equation}
Hence, define 
\begin{equation*}
F(x,y,s) = h\left(y\right) + h\left(x+y\right) - h(x)\left( 1 - Ms \right)\qquad x,~y\in \R\,,~ 0\le s < M^{-1}.
\end{equation*}
It is easy to check that the following properties hold,
\begin{align*}
&\partial_y F = h'(y) + h' (x+y)>0\\
&\lim_{y\to\pm\infty}  F(x,y,s) = \pm\infty\qquad \forall\, (x,s)\,.
\end{align*}
By the implicit function theorem, there exists a unique implicit function $y=y(x,s)$. Next, let's deduce other properties of $y(x,t)$:
\begin{itemize}
\item[(i)]since
$F(x,0,0) =F(0,0,s) =0$\,,
then
$$
y(x,0)=0= y(0,s)\qquad \forall\, x,s
$$
which implies that $y(x,s)={\mathcal O}(1) xs$ in a neighborhood of the origin;

\item[(ii)] for every $(x,s)$, $x\not=0\not = s$ and recalling that $h(x)x > 0$ for $x\not =0$, one has
\begin{align*}
x F(x,0,s) &= x\left[h(0) + h\left(x\right) - h(x)\left( 1 - Ms \right) \right] \\
& = x h(x) Ms >0
\end{align*}
and 
\begin{align*}
x F(x,-x,s) &=  x  h\left(-x\right) - x h(x)\left( 1 - Ms \right)  <0\,.
\end{align*}
Using the (increasing) monotonicity of $y\mapsto F(x,y,s)$, we obtain that
\begin{equation}\label{A1yestmx}
 \begin{cases} -x <y(x,s)<0& \mbox{ if } x>0\\[1mm]
-x >y(x,s)>0 & \mbox{ if } x<0\,,
\end{cases}
\end{equation}
and properties in \eqref{time-step_sign-properties} hold true. The proof is complete.
\end{itemize}

\subsection{Proof of Proposition~\ref{prop:estimate-time-step}}\label{subsec:app2}
We adopt the notation in Appendix~\ref{subsec:app1} for $x$, $s$ and $y$ and assume that $|x|\le q$ for some $q>0$. We rewrite the implicit equation 
\eqref{eq:implicit-time-step} in the two cases, depending on the sign of $x$.

\medskip
$\bullet$\quad For $0<x\le q$ we get
\begin{equation}\label{eq:implicit-time-step-x>0}
\sinh\left(y(x,s)\right) + y(x,s) = - M x s  \,. 
\end{equation}
Set $\gamma(y) = \sinh(y)+y$ and use the elementary identity $\gamma(y)=\gamma(y)-\gamma(0)=\gamma'(c)y$, with
\begin{equation*}
2 \le \gamma'(c) = 1 +\cosh(c) \le 1 +\cosh q\,,
\end{equation*}
since $|c|\le |y|\le q$ from~\eqref{A1yestmx}.
We deduce easily that
\begin{equation*}
 2|y|\le |\gamma(y)| = Mxs \le \left(1 +\cosh q\right) |y|\,,\qquad y=y(x,s)
\end{equation*}
from which we conclude that \eqref{time-step-reflected-wave} holds for $x=\eps_2^->0$.

\medskip
$\bullet$\quad On the other hand, for  $-q\le x<0$, the implicit equation \eqref{eq:implicit-time-step} rewrites as
\begin{equation}\label{eq:implicit-time-step-x<0}
\sinh\left( x+y(x,s)\right) - \sinh x + y(x,s) =   - Ms \sinh x  \,,\qquad x<0  \,. 
\end{equation}
As before, the function 
$$
\widetilde \gamma(y;x) \,\dot =\, \sinh(x+y)  - \sinh x + y\qquad  x+y<0
$$ 
satisfies 
$$
\widetilde \gamma(0;x) =0\,,\qquad \partial_y \widetilde \gamma = \cosh(x+y) +1\,.
$$
By using the sign rules for $x$, $y$ and $x+y$ established in  Proposition~\ref{S2:prop:estimate-time-step} (b), we find that $0\le |x+y| = |x|-y \le q$, and hence 
\begin{equation*}
2|y| \le |\widetilde \gamma(y;x)|  =  Ms \sinh {|x|}     \le \left(1 +\cosh q\right) |y|\,.
\end{equation*}
From the two inequalities above we easily deduce, on one hand, that
\begin{equation*}
 Ms|x|  \le Ms \sinh {|x|}     \le \left(1 +\cosh q\right) |y| = c_1(q)^{-1}  |y|
\end{equation*}
and therefore that
\begin{equation*}
c_1(q) \le \frac{|y|}{Ms|x|};
\end{equation*}
on the other hand, that
\begin{equation*}
\frac{|y|}{Ms|x|}\le \frac {\sinh {|x|} }{2 |x|} \le \frac {\cosh {q} }{2}\,. 
\end{equation*}
This completes the proof of \eqref{time-step-reflected-wave} for $x=\eps_2^- <0$.

\subsection{Proof of Lemma~\ref{lem:Delta-L-xi}}\label{subsec:app3}
This proof is an adaptation of the one of Prop. 5.8, Ref.~\cite{ABCD_JEE_2015} and takes into account the possible wave configurations.
We use the notation of Lemma \ref{lem:shock-riflesso} and assume $i=2$. Then the following basic identities hold, 
\begin{align}
\eps_2 -  \eps_1 & =   \alpha_2 + \beta_2\,,
\label{tre-uno}
\\
h(\eps_1) + h(\eps_2) & =  h(\alpha_2) + h(\beta_2)\,. \label{tre-due}
\end{align}
Namely, 

$\bullet$~~ \eqref{tre-uno} is obtained by equating $(u_r - u_\ell)$ before and after the interaction and by using the definition \eqref{eq:strengths} of the strengths, 

$\bullet$~~ while \eqref{tre-due} is obtained by equating $(v_r - v_\ell)$ before and after the interaction and by using the parametrization \eqref{eq:lax13}.

\medskip
\paragraph{\fbox{$SS\to RS$}}\quad In this case, two shocks interact and give rise to a propagating shock of size $\eps_2$ and a rarefaction of size $\eps_1$.
The signs are 
$$
\alpha_2<0\,,\qquad  \beta_2<0\,,\qquad \eps_2<0 < \eps_1\,.
$$
Then, by \eqref{tre-uno}, we get
$$
\Delta L_{in} = |\eps_1| + |\eps_2|-|\alpha_2|- |\beta_2| = 0\,,
$$
and hence
\begin{align}\nonumber
    \Delta L_\xi + |\eps_1|(\xi - 1) & =  |\eps_1| + \xi \left(|\eps_2|-|\alpha_2|- |\beta_2| \right)  + |\eps_1|(\xi - 1)\\
    &= \xi (|\eps_1| + |\eps_2|-|\alpha_2|- |\beta_2|)= 0\,,
    \label{Delta-Lxi-SS-RS}
\end{align}
that is, \eqref{eq:refined-decay-Lxi} holds. Notice that the above argument is valid for every $\xi\ge 1$\,.

\medskip
\paragraph{\fbox{$SR,\, RS\to SR,\, SS$}}\quad In this case, the incoming waves have different sign; the resulting wave 
may be positive or negative, while the reflected wave is necessarily a shock.

Without loss of generality, assume $\alpha_2<0<\beta_2$. We claim that the following inequality (which is stronger that \eqref{eq:refined-decay-Lxi}) holds,
\begin{equation}\label{Delta_L_xi_13-SR}
\Delta L_\xi + |\eps_1|\xi (\xi - 1)  \le  0\,.
\end{equation}

\fbox{$SR \to SS$}\quad Indeed, if $\eps_2$ is a shock, then the signs are 
$$
\alpha_2<0< \beta_2 \,,\qquad \eps_2<0\,,\qquad  \eps_1<0
$$
and so \eqref{tre-uno} gives
$$
 |\eps_1| - |\eps_2|  =  -| \alpha_2| + |\beta_2|\,,
$$
while
$$
\Delta L_\xi = \xi( |\eps_1| + |\eps_2|-|\alpha_2|) - |\beta_2|\,.
$$
Therefore,
\begin{align*}
\Delta L_\xi + |\eps_1|\xi (\xi - 1) &=\xi^2 |\eps_1| + \xi(|\eps_2|-|\alpha_2|) - |\beta_2|\\
&= \xi^2 |\eps_1| + \xi(|\eps_1|-|\beta_2|) - |\beta_2|\\
&= (\xi+1) (\xi|\eps_1| - |\beta_2|)\,.
\end{align*}
By \eqref{eq:chi_def}, we find that  $|\eps_1| \le c(m) |\beta_2|$ and hence
\begin{align*}
\Delta L + |\eps_1|\xi (\xi - 1) 
&\le  (\xi+1) \left(\xi c(m) - 1\right) |\beta_2|\le 0
\end{align*}
because of bounds on $\xi$ in \eqref{bounds-on-xi}. Therefore \eqref{Delta_L_xi_13-SR} holds in this case.

\medskip
\fbox{$SR\to SR$}\quad On the other hand, assume that the resulting wave, of size $\eps_2$, is a rarefaction:
$$
\alpha_2<0< \beta_2 \,,\qquad  \eps_1<0 < \eps_2\,.
$$
Then
$$
\Delta L_\xi = \xi( |\eps_1| -|\alpha_2|)  + |\eps_2| - |\beta_2|
$$
and the left hand side of \eqref{Delta_L_xi_13-SR} turns out to be 
\begin{align*}
\Delta L_\xi + |\eps_1|\xi (\xi - 1)  &= \xi^2 |\eps_1| + |\eps_2| -\xi|\alpha_2| - |\beta_2| \\
& = \xi \underbrace{\left(\xi |\eps_1|   - |\alpha_2|  \right)}_{(A)} +  \underbrace{|\eps_2| - |\beta_2|}_{(B)}\,.
\end{align*}
As above, from \eqref{eq:chi_def} and \eqref{bounds-on-xi}, we obtain  
$$
(A) \le \xi c(m)  |\alpha_2| -  |\alpha_2| \le 0\,,
$$
while by Lemma~\ref{lem:shock-riflesso}, (b), the amount of shocks and of rarefaction decreases across the interaction, that is
$$
|\eps_1| < |\alpha_2|\,,\qquad |\eps_2| < |\beta_2|
$$
that implies $(B)<0$.

This completely proves \eqref{Delta_L_xi_13-SR}, and concludes the proof of the Lemma~\ref{lem:Delta-L-xi}.

\section*{Acknowledgment}
The research that led to the present paper was partially supported by 2020 INdAM-GNAMPA Project "Buona posi\-tu\-ra, regolarit\`a e controllo per alcune equazioni di
evolu\-zio\-ne". The first author kindly acknowledges the hospitality of the University of Cyprus, where part of this work was done. 
The authors would like to thank Prof. Alberto Bressan for a useful discussion on the setup of the model.

\end{document}